\documentclass[review]{elsarticle}
\usepackage{lineno,hyperref}

\usepackage{amsfonts,amsmath,amssymb,amsthm,mathrsfs}
\usepackage{latexsym}
\usepackage{graphicx}
\usepackage{indentfirst}
\usepackage{color}
\usepackage{fancyhdr}

\theoremstyle{plain}                       
\newtheorem{thm}{Theorem}[section]
\newtheorem{lemma}[thm]{Lemma}
\newtheorem{cor}[thm]{Corollary}
\newtheorem{definition}[thm]{Definition}
\newtheorem{remark}[thm]{Remark}

\theoremstyle{remark}

\def\Xint#1{\mathchoice
  {\XXint\displaystyle\textstyle{#1}}%
  {\XXint\textstyle\scriptstyle{#1}}%
  {\XXint\scriptstyle\scriptscriptstyle{#1}}%
  {\XXint\scriptscriptstyle\scriptscriptstyle{#1}}%
  \!\int}
\def\XXint#1#2#3{{\setbox0=\hbox{$#1{#2#3}{\int}$}
  \vcenter{\hbox{$#2#3$}}\kern-.5\wd0}}

\def\dashint{\Xint-}

\modulolinenumbers[5]

\numberwithin{equation}{section}

\journal{Journal of \LaTeX\ Templates}

\begin{document}

\begin{frontmatter}

\title{Convergence rates and $W^{1,p}$ estimates in homogenization theory of \\
Stokes systems in Lipschitz domains \tnoteref{mytitlenote}}
\tnotetext[mytitlenote]{
The research was supported
by the National Natural Science Foundation of China (Grant No. 11471147).}

\author{\large Qiang~ Xu \fnref{myfootnote}}
\address{\normalsize Department of Mathematics, Peking University, Beijing, 100871, PR China.}
\fntext[myfootnote]{E-mail:~xuqiang@math.pku.edu.cn}




\begin{abstract}
 Concerned with the Stokes systems with rapidly oscillating periodic coefficients,
 we mainly extend the recent works in \cite{SGZWS,G} to those in term of Lipschitz domains.
 The arguments employed here are quite different from theirs, and the basic idea comes from \cite{QX2}, originally motivated by
 \cite{SZW2,SZW12,TS}.
 We obtain an almost-sharp $O(\varepsilon\ln(r_0/\varepsilon))$ convergence rate in $L^2$ space, and
 a sharp $O(\varepsilon)$ error estimate in $L^{\frac{2d}{d-1}}$ space by a little stronger assumption.
 Under the dimensional condition $d=2$,
 we also establish the optimal $O(\varepsilon)$ convergence rate on pressure terms in $L^2$ space. Then utilizing the convergence rates we can derive the $W^{1,p}$ estimates
 uniformly down to microscopic scale $\varepsilon$ without any smoothness assumption on the coefficients,
 where $|\frac{1}{p}-\frac{1}{2}|<\frac{1}{2d}+\epsilon$ and $\epsilon$ is a positive constant independent of
 $\varepsilon$. Combining the local estimates, based upon $\text{VMO}$ coefficients,
 consequently leads to the uniform $W^{1,p}$ estimates. Here the proofs do not rely on
 the well known compactness methods.
\end{abstract}


\end{frontmatter}


\section{Introduction and main results}

In recent years the study of quantitative homogenization of Stokes systems in smooth domains has received
an important development in \cite{SGZWS,G}.
However there is few related research involving non-smooth ones.
Based on the weighted-type estimates and duality methods investigated by the author in \cite{QX2},
essentially motivated by \cite{SZW2,SZW12,TS},
this paper primarily studies the sharp convergence rates in $L^2$ space for homogenization theory of
Stokes systems in a bounded Lipschitz domain. As an application,
one may derive the uniform $W^{1,p}$ estimates with $|\frac{1}{p}-\frac{1}{2}|<\frac{1}{2d}+\epsilon$
by an additional smoothness assumption on coefficients.
Here we improved the arguments used in \cite{SZW10,JGZSLS}.
In fact, we can employ the convergence rates to establish the $W^{1,p}$ estimates uniformly down
to the microscopic scale, and then together with the corresponding local estimates
arrive at the full-type estimates. We mention that the idea is motivated by S. Armstrong and Z. Shen in
\cite{SZ,SACS,SZW12}.

More precisely, given
$F\in H^{-1}(\Omega;\mathbb{R}^d)$,
$h\in L^2(\Omega)$ and $g\in H^{\frac{1}{2}}(\Omega;\mathbb{R}^d)$ with a proper compatibility condition,
we consider the following compressible Stokes systems with Dirichlet
boundary condition
\begin{equation*}\label{pde:1.1}
(\mathbf{DS_\varepsilon})\left\{
\begin{aligned}
\mathcal{L}_\varepsilon(u_\varepsilon) + \nabla p_\varepsilon &= F &\quad &\text{in}~~\Omega, \\
 \text{div} (u_\varepsilon) &= h &\quad&\text{in} ~~\Omega,\\
 u_\varepsilon &= g &\quad&\text{on} ~\partial\Omega,
\end{aligned}\right.
\end{equation*}
where $\varepsilon>0$, and
\begin{equation*}
 \mathcal{L}_\varepsilon = - \text{div}\big[A(x/\varepsilon)\nabla\big]
 = -\frac{\partial}{\partial x_i}
 \Big[a_{ij}^{\alpha\beta}\Big(\frac{x}{\varepsilon}\Big)\frac{\partial}{\partial x_j}\Big].
\end{equation*}

Let $d\geq 2$ and $1 \leq i,j,\alpha,\beta\leq d$.
The summation convention for repeated indices is used throughout and $\Omega$ is always
supposed to be a bounded Lipschitz domain, unless otherwise stated.
Assume that the coefficient matrix $A = (a_{ij}^{\alpha\beta})$ is real and satisfies
the uniform ellipticity condition
\begin{equation}\label{a:1}
 \mu |\xi|^2 \leq a_{ij}^{\alpha\beta}(y)\xi_i^\alpha\xi_j^\beta\leq \mu^{-1} |\xi|^2
 \quad \text{for}~y\in\mathbb{R}^d~\text{and}~\xi=(\xi_i^\alpha)\in \mathbb{R}^{d\times d},
 \quad\text{where}~\mu>0;
\end{equation}
and the periodicity condition
\begin{equation}\label{a:2}
A(y+z) = A(y)
\qquad\text{for}~y\in \mathbb{R}^d ~\text{and}~ z\in \mathbb{Z}^d.
\end{equation}
Then it is well known that $u_\varepsilon \rightharpoonup u_0$ weakly in $H^1(\Omega;\mathbb{R}^d)$ and
$p_\varepsilon - \dashint_\Omega p_\varepsilon \rightharpoonup p_0 -\dashint_\Omega p_0 $ weakly in
$L^2(\Omega)$ as $\varepsilon \to 0$ (see for example \cite{ABJLGP,SGZWS}),
where the notation $\dashint_\Omega$ denotes
the average integral over $\Omega$, and the pair $(u_0,p_0)\in H^1(\Omega;\mathbb{R}^d)\times L^2(\Omega)$ is the weak solution of the homogenized
system
\begin{equation*}
(\mathbf{DS_0})\left\{
\begin{aligned}
\mathcal{L}_0(u_0) + \nabla p_0 &= F &\quad &\text{in}~~\Omega, \\
 \text{div} (u_0) &= h &\quad&\text{in} ~~\Omega,\\
 u_0 &= g &\quad&\text{on} ~\partial\Omega.
\end{aligned}\right.
\end{equation*}
Here the homogenized operator $\mathcal{L}_0$ is an elliptic operator with constant coefficients satisfying
$\eqref{a:1}$ and depending only on the matrix $A$ (see \cite{SGZWS,ABJLGP}).
Besides, we impose the symmetry condition $A=A^*$, i.e.,
\begin{equation}\label{a:3}
 a_{ij}^{\alpha\beta}(y) = a_{ji}^{\beta\alpha}(y) \quad \text{for}~ y\in\mathbb{R}^d.
\end{equation}
To guarantee the existence of the solutions of $(\mathbf{DS_\varepsilon})$ and $(\mathbf{DS_0})$,
it is also necessary to introduce the compatibility condition for the given data $h$ and $g$ such that
\begin{equation}\label{a:4}
\int_\Omega h dx = \int_{\partial\Omega} n\cdot g dS,
\end{equation}
where $n=(n_1,\cdots,n_d)$ denotes the outward unit normal vector to $\partial\Omega$ throughout.

We present the quantitative results in the following,
and some unfamiliar notation will be explained later.
\begin{thm}[Convergence rates]\label{thm:1.1}
Suppose that $A$ satisfies $\eqref{a:1}-\eqref{a:3}$. Assume $F\in L^2(\Omega;\mathbb{R}^d)$,
$h\in H^1(\Omega)$ and $g\in H^1(\partial\Omega;\mathbb{R}^d)$ with the compatibility condition $\eqref{a:4}$.
Let $(u_\varepsilon,p_\varepsilon)$ and $(u_0,p_0)$ in $H^1(\Omega;\mathbb{R}^d)\times L^2(\Omega)/\mathbb{R}$
be the weak solutions of the Dirichlet problems
$(\textbf{DS})_\varepsilon$ and $(\textbf{DS})_0$, respectively. Then
we have
\begin{equation}\label{pri:1.1}
\big\|u_\varepsilon - u_0\big\|_{L^2(\Omega)}\leq C\varepsilon\ln(r_0/\varepsilon)
\Big\{\|F\|_{L^2(\Omega)}+\|h\|_{H^1(\Omega)}+\|g\|_{H^1(\partial\Omega)}\Big\},
\end{equation}
and
\begin{equation}\label{pri:1.3}
\|p_\varepsilon-p_0-\pi(\cdot/\varepsilon)
S_\varepsilon(\psi_{2\varepsilon}\nabla u_0)\|_{L^2(\Omega)/\mathbb{R}}
\leq C\varepsilon^{\frac{1}{2}}\Big\{\|F\|_{L^2(\Omega)}+\|h\|_{H^1(\Omega)}
+\|g\|_{H^1(\partial\Omega)}\Big\}.
\end{equation}
Moreover, if $u_0\in H^2(\Omega;\mathbb{R}^d)$, then we have
  \begin{equation}\label{pri:1.2}
  \big\|u_\varepsilon - u_0\big\|_{L^{\frac{2d}{d-1}}(\Omega)}\leq C\varepsilon \|u_0\|_{H^2(\Omega)},
  \end{equation}
while in the special case of $d=2$, we derive a sharp estimate
  \begin{equation}\label{pri:1.4}
  \big\|u_\varepsilon-u_0\big\|_{L^{p}(\Omega)}
  + \|p_\varepsilon-p_0-\pi(\cdot/\varepsilon)
  \nabla u_0\|_{L^2(\Omega)/\mathbb{R}}
  \leq C\varepsilon\|u_0\|_{H^2(\Omega)}
  \end{equation}
for $1\leq p<\infty$, where $C$ depends only on $\mu,d$ and $\Omega$.
\end{thm}

A few remarks on notation are in order. $L^2(\Omega)/\mathbb{R}$ represents the quotient space of $L^2(\Omega)$ with respect to the relation
$u\sim v \Leftrightarrow u-v\in\mathbb{R}$, while
$\|\cdot\|_{L^2(\Omega)/\mathbb{R}}$ is the corresponding quotient norm, given by
$\inf_{c\in\mathbb{R}}\|\cdot-c\|_{L^2(\Omega)}$.
The function $\pi=(\pi_k^\gamma)$ coupled with $\chi=(\chi_k^{\beta\gamma})$
is referred to as the correctors associated with the problem $(\mathbf{DS}_\varepsilon)$,
and they are the solution of the cell problem $\eqref{pde:2.1}$.
$S_\varepsilon$ denotes the smoothing operator at scale $\varepsilon$ (see Definition
$\ref{def:2.1}$), and we mention that V.V. Zhikov and S.E. Pastukhova originally applied the so-called
Steklov smoothing operator to the homogenization problem in \cite{ZVVPSE}.
Here $\psi_{2\varepsilon}$ is a cut-off function whose definition is given in $\eqref{def:2.5}$.

Without any smoothness assumption on coefficients,
Theorem $\ref{thm:1.1}$ remarkably extends the results obtained by \cite{G}
in two perspectives: lower regularity assumptions on domains and given data, and sharp convergence rates for pressure term in the case of $d=2$.
The approach to attack the problems related to non-smooth domains
is usually more complicated than that to smooth one.
There are three crucial analysis tools devoted to complete the proof of Theorem $\ref{thm:1.1}$.
The first one is the weighted-type estimates for smoothing operator $S_\varepsilon$ at scale $\varepsilon$
(see Subsection $\ref{subsection:2.1}$).
The second one is the so-called duality lemmas, i.e., Lemmas $\ref{lemma:4.1}$ and $\ref{lemma:4.2}$,
which actually motivated by T. Suslina in \cite{TS}. The last one is
non-tangential maximal function coupled with radial maximal function, which is of help to the so-called
``layer type'' and ``co-layer type'' estimates (see Lemmas $\ref{lemma:3.3}$ and $\ref{lemma:4.3}$),
and this tool is originally employed by C. Kenig, F. Lin and Z. Shen in \cite{SZW2}.
We mention that radial maximal function plays an important part in controlling the behavior of
the pressure term near the boundary.
The methods developed in \cite{SZW2} were totally designed for non-smooth domains. Nevertheless,
we can not apply it directly since we lack the non-tangential maximal function estimates
for the solution of
$(\textbf{DS}_\varepsilon)$, and this will be established in a separate work.
Here on account of the weighted-type estimates and duality lemmas,
it is possible to transfer all the estimates from the problem $(\textbf{DS}_\varepsilon)$ to
the homogenized one $(\textbf{DS}_0)$,
while we have already had many useful estimates for $(\textbf{DS}_0)$, e.g., \cite{EBFCEKGCV,RMBZS,OAL}.
We end this paragraph by mention that the results of Theorem $\ref{thm:1.1}$ may be extended to
the Neumann boundary value problem without any real difficulty.

Note that the convergence rate on the pressure term $\eqref{pri:1.3}$
is not sharp except of the special case $d=2$,
since this result actually relies on the error estimate of
$u_\varepsilon-u_0-\varepsilon\chi(\cdot/\varepsilon)S_\varepsilon(\psi_{2\varepsilon}\nabla u_0)$
in $H^1_0(\Omega;\mathbb{R}^d)$ and its result is merely $O(\varepsilon^{\frac{1}{2}})$
in the case of $d\geq3$.
We find that if the corrector $\chi=(\chi_k^{\beta\gamma})$ is H\"older continuous in $\mathbb{R}^d$,
then it is not very hard to derive that
$\|u_\varepsilon-u_0-\varepsilon\chi(\cdot/\varepsilon)\nabla u_0\|_{H^1(\Omega)}
= O(\varepsilon)$.
In this case the duality methods is even not employed.
However, we can not count on $\chi_k^{\beta\gamma}\in C^{0,\sigma}(\mathbb{R}^d)$ for some
$\sigma\in(0,1)$ without any smoothness assumption
on $a_{ij}^{\alpha\beta}$, because of the absence of the De Giorgi-Nash-Moser theory.
Fortunately, there still exists an exceptional case $d=2$, in which the hole-filling technique (see \cite{MGLM,KOW})
guarantees the H\"older continuity
of $\chi_k^{\beta\gamma}$ as long as $A$ satisfies the assumptions $\eqref{a:1}$ and $\eqref{a:2}$.
Consequently the estimate $\eqref{pri:1.4}$ follows, and we remark that the proof, in fact,
does not rely on the symmetry condition $\eqref{a:3}$. Although the estimate $\eqref{pri:1.3}$ is not
optimal, it is sufficient to derive
uniform global $L^\infty$ estimates for pressure terms in term of smooth domains,
and we will address this topic in a forthcoming paper.

To make the statements of the paper well-founded,
we actually ask the Lipschitz domain $\Omega$ without external cusps,
since there is a counterexample (see \cite[pp.374-375]{GARGDMAM}) to show that
the desired estimates $\eqref{pri:2.2}$, related to divergence operator ``div'', is not true
if the domain has an external cusp.
However this is not a very severe restriction, and star-shaped domains are still valid,
as well as most of the Lipschitz domains even with large character constant.
We refer the reader to \cite{GARGDMAM} and the references therein for more details.
Finally, we remark that the topic on convergence rates in homogenization theory has extensively
been studied in recent years, and without attempting to be exhaustive we refer the reader to
\cite{SZ,SACS,ABJLGP,BMSHSTA,JGZSLS,GG1,G,VSO,SZW2,SZW16,ODVB,SZW12,SZW10,SZW13,QXS1,QX2} for more results.

As we mentioned before, on account of the convergence rates,
we find another way of leading to $W^{1,p}$ estimates uniformly down to scale $\varepsilon$.
The idea comes from the recent work (see \cite{SZ,SZW12}).
To obtain the full-type estimates, we need the corresponding local estimate at scale $\varepsilon$,
and this is exactly where the smoothness of the coefficient works.
Here the coefficient $A$ is required to belong
to $\text{VMO}(\mathbb{R}^d)$ class, and its definition and the notation $\omega$
may be found in \cite[pp.2283]{SZW10}.
The following theorem concerns a uniform regularity estimate.

\begin{thm}[$W^{1,p}$ estimates]\label{thm:1.2}
Assume that $A\in\emph{VMO}(\mathbb{R}^d)$ satisfies $\eqref{a:1}-\eqref{a:3}$. Then exists
$\epsilon>0$ depending only on $\mu,d$ and $\Omega$, such that for any
$F\in W^{-1,p}(\Omega;\mathbb{R}^d)$ with $\big|\frac{1}{p}-\frac{1}{2}\big|<\frac{1}{2d}+\epsilon$,
$h\in L^p(\Omega)$ and $g\in B^{1-\frac{1}{p},p}(\partial\Omega;\mathbb{R}^d)$ with
the compatibility condition $\eqref{a:4}$,
the unique solution
$(u_\varepsilon,p_\varepsilon)\in W^{1,p}(\Omega;\mathbb{R}^d)\times L^p(\Omega)/\mathbb{R}$ to
the Dirichlet problem $(\textbf{DS}_\varepsilon)$ satisfies the uniform estimate
\begin{equation}\label{pri:1.5}
\|\nabla u_\varepsilon\|_{L^{p}(\Omega)} + \|p_\varepsilon\|_{L^p(\Omega)/\mathbb{R}}
\leq C\Big\{\|F\|_{W^{-1,p}(\Omega)} + \|h\|_{L^p(\Omega)}
+ \|g\|_{B^{1-1/p,p}(\partial\Omega)}\Big\},
\end{equation}
where $\epsilon,C$ depends only on $\mu,\omega,d$ and $\Omega$.
\end{thm}

Here $B^{\sigma,p}$ denotes the $L^p$ Besov space of order $\sigma$, and
$B^{\frac{1}{2},2}(\partial\Omega;\mathbb{R}^d) = H^{\frac{1}{2}}(\partial\Omega;\mathbb{R}^d)$. Under the assumption that
$\partial\Omega\in C^1$, S. Gu and Z. Shen have established the estimate $\eqref{pri:1.5}$ for
$1<p<\infty$ in \cite[Theorem 1.4]{SGZWS}. To obtain the crucial step in the proof, i.e.,
the reverse H\"older's inequality, their method relies on an interior Lipschitz estimate
and a H\"older estimate near boundary. However the methods are not quite suitable for non-smooth domain,
since it is impossible to acquire the boundary H\"older estimates with the H\"older exponents
in full range (0,1) for general Lipschitz domains.
If the domian is depicted by a small Lipschitz constant, G.P. Galdi, C.G. Simader and H. Sohr
have $W^{1,p}$ estimate with $1<p<\infty$ for Stokes systems with constant coefficients
(see \cite[Theorem 2.1]{GPGCGHS}).
To the extent that the author knows, the range of $p$ in Theorem $\ref{thm:1.2}$
is the best result even for the constant coefficient setting, and we refer the reader
to \cite[Theorem 1.3]{JGJK} for the recent work of J. Geng and J. Kilty,
and to \cite{SZ,MAFHL,RMBZS,LI,EBFCEKGCV,GPGCG,OAL,RT,SZW17} for some related results.

The basic idea to treat the estimate $\eqref{pri:1.5}$ here may be found in \cite{SZW10,JGZSLS}, and its
principal ingredient is the decay estimate of $u_\varepsilon$, i.e, $\eqref{pri:5.5}$.
The novelty here is that we provide a considerably different proof from a technical standpoint.
Compared to the methods developed in \cite{SZW10,JGZSLS},
we do not adopt the well known compactness methods.
Instead by using the derived convergence result $\|u_\varepsilon-u_0\|_{L^{p}(\Omega)}=O(\varepsilon^{\frac{1}{2}})$
with $1\leq p<\frac{2d}{d-2}$ in the first part of the paper (see Corollary $\ref{cor:3.1}$),
it is natural to think of transferring the corresponding decay estimates for $u_\varepsilon$
to a similar one for $u_0$. We emphasize that the convergence rate above, as a matter of fact,
play a role in the domains from the microscopic scale
to macroscopic one (see the estimate $\eqref{f:5.2}$).
Then the fact that any weakly convergent sequence is bounded suggests that the estimate for $u_0$
may go back to that for $u_\varepsilon$ in the macroscopic scale, and this completes the whole argument.
Besides, the duality argument is used in the proof of the theorem, by which the proof related to
$F\in W^{-1,p}(\Omega;\mathbb{R}^d)$ can be reduced to prove the same type estimate
for the source term $\text{div}(f)$ with $f\in L^p(\Omega;\mathbb{R}^{d\times d})$.
On the other hand, due to Lemma $\ref{lemma:2.2}$ we can address the incompressible Stokes system with
zero boundary value at first, and then study the compressible case with nonzero boundary value.
In the end we remark that there is a strong probability extending the proof to non-periodic settings
(see for example \cite{SZ,SACS}).

The paper is organized as follows.
Section 2 is divided into three subsections, which involve the smoothing operator, correctors and
non-tangential $\&$ radial maximal functions, respectively. Also, the notation and definitions are
introduced there. In Section 3 we establish the corresponding convergence rates in $H^1$-norm,
and consequently prove the estimates $\eqref{pri:1.1}$, $\eqref{pri:1.2}$ and $\eqref{pri:1.3}$ in
Section 4. The special case of $d=2$ is discussed in Section 5.
As an application, we will verify Theorem $\ref{thm:1.2}$ in Section 6, which includes two subsections.
Subsection 6.1 studies the $W^{1,p}$ estimates without any smoothness assumption on $A$, and
Subsection 6.2 handle the corresponding local estimates.

We end this section with some notation that will be used throughout the paper.
\begin{itemize}
  \item $\nabla v = (\nabla_1 v, \cdots, \nabla_d v)$ is the gradient of $v$, where
  $\nabla_i v = \partial v /\partial x_i$ denotes the $i^{\text{th}}$ derivative of $v$. \\
  $\nabla^2 v = (\nabla^2_{ij} v)_{d\times d}$  denotes the Hessian matrix of $v$, where
  $\nabla^2_{ij} v = \frac{\partial^2 v}{\partial x_i\partial x_j}$. \\
  $\text{div}(\textbf{v})=\sum_{i=1}^d \nabla_i v_i$ denotes the divergence of $\textbf{v}$, where
  $\textbf{v} = (v_1,\cdots,v_d)$ is a vector-valued function.
  \item $L^p(\Omega)/\mathbb{R}=\{f\in L^p(\Omega):\int_\Omega f(x)dx = 0\}$,
  and $\|f\|_{L^p(\Omega)/\mathbb{R}} = \inf_{c\in\mathbb{R}}\|f-c\|_{L^2(\Omega)}$, where $p\in[1,\infty)$.
  \item $\delta(x) = \text{dist}(x,\partial\Omega)$ denotes the distance function for $x\in\Omega$, and
  we set $\delta(x) = 0$ if $x\in\mathbb{R}^d\setminus\Omega$.
  \item $S_r = \{x\in\Omega:\text{dist}(x,\partial\Omega) = r\}$ denotes the level set.
  \item $\Omega\setminus\Sigma_r$ denotes the boundary layer with thickness $r>0$, where
  $\Sigma_r = \{x\in\Omega:\text{dist}(x,\partial\Omega)>r\}$.
  \item $r_0$ is the diameter of $\Omega$, and the internal diameter $r_{00}$ is defined by
  $\max\{r>0:B(x,r)\subset\Omega,\forall x\in S_r\}$.
  Let $c_0 = r_{00}/10$ denote the layer constant of $\Omega$.
  \item Let $B=B(x,r)=B_r(x)$, and $nB=B(x,nr)$ denote the concentric balls as $n>0$ varies.
  \item Let $\vartheta:\mathbb{R}^{d-1}\to\mathbb{R}$ be a Lipschitz function such that $\vartheta(0) = 0$ and
$\|\nabla\vartheta\|_{L^\infty(\mathbb{R}^{d-1})}\leq M$. For any $r>0$, let
\begin{equation*}
\begin{aligned}
&\Delta_r = \big\{(x^\prime,\vartheta(x^\prime))\in\mathbb{R}^d:|x^\prime|<r\big\},\\
& D_r= \big\{(x^\prime,t)\in\mathbb{R}^d:|x^\prime|<r~\text{and}~\vartheta(x^\prime)<t<
\vartheta(x^\prime)+m_0r\big\},
\end{aligned}
\end{equation*}
where $M$ is called the Lipschitz character constant, and the constant $m_0 = M+10d$ is used throughout.
  \item Let $\psi_r$ denote the cut-off function associated with $\Sigma_r$, such that
  \begin{equation}\label{def:2.5}
   \psi_{r} = 1 \quad\text{in}~\Sigma_{2r}, \qquad
   \psi_{r} = 0 \quad\text{outside}~\Sigma_r, \qquad\text{and}\quad |\nabla\psi_r|\leq C/r.
  \end{equation}
  \item The weighted-type norms are defined by
  \begin{equation}\label{def:2.2}
   \|f\|_{L^2(\Sigma_r;\delta)} = \Big(\int_{\Sigma_r}|f(x)|^2\delta(x)dx\Big)^{1/2},
   \qquad
   \|f\|_{L^2(\Sigma_r;\delta^{-1})} = \Big(\int_{\Sigma_r}|f(x)|^2\delta^{-1}(x)dx\Big)^{1/2}.
  \end{equation}
  Furthermore, we define $\|f\|_{H^2(\Sigma_r;\delta)}
  = \sum_{i=1}^2\|\nabla^k f\|_{L^2(\Sigma_r;\delta)} + \|f\|_{L^2(\Sigma_r;\delta)}$, and
  the definition of $\|f\|_{H^2(\Sigma_r;\delta^{-1})}$ is given by a
  similar way.
\end{itemize}

Throughout the paper, the constant $C$ never depends on $\varepsilon$. Finally
we mention that we shall make a little effort to distinguish vector-valued functions or
function spaces from their real-valued counterparts, and they will be clear from the context.

\section{Preliminaries}

\begin{definition}
We say that $(u_\varepsilon,p_\varepsilon)\in H^1(\Omega;\mathbb{R}^d)\times L^2(\Omega)$ is a weak solution
to $(\textbf{DS}_\varepsilon)$,
if $(u_\varepsilon,p_\varepsilon)$ satisfies
\begin{enumerate}
  \item \quad $B_\varepsilon[u_\varepsilon,\phi] -\int_\Omega p_\varepsilon\emph{div}(\phi) dx
    = \big<F,\phi\big>_{H^{-1}(\Omega)\times H_0^1(\Omega)}
   \quad \text{for~any} ~\phi\in H_0^1(\Omega;\mathbb{R}^d),$
  \item \quad $\emph{div}(u_\varepsilon) = h$ in the distribution sense in $\Omega$,
  \item \quad $u_\varepsilon = g$ in the trace sense on $\partial\Omega$,
\end{enumerate}
where $B_\varepsilon[\cdot,\cdot]$ is the bilinear form defined by
\begin{equation*}
 B_{\varepsilon}[v,w] = \int_\Omega a_{ij}^{\alpha\beta}\Big(\frac{x}{\varepsilon}\Big)
 \frac{\partial v^\beta}{\partial x_j}\frac{\partial w^\alpha}{\partial x_i} dx
 \quad \text{for~any} ~u,v\in H^1(\Omega;\mathbb{R}^d).
\end{equation*}
\end{definition}

\begin{lemma}\label{lemma:2.2}
Let $1<p<\infty$. For any $f\in L^p(\Omega)/\mathbb{R}$,
there exists a unique solution $u\in W^{1,p}_0(\Omega;\mathbb{R}^d)$ such that
$\emph{div}(u) =f$ in $\Omega$, and we have
\begin{equation}\label{pri:2.2}
 \|u\|_{W^{1,p}_0(\Omega)} \leq C\|f\|_{L^p(\Omega)},
\end{equation}
where $C$ depends on $d,p$ and $\Omega$.
\end{lemma}

\begin{proof}
The proof may be found in \cite[Theorem 4.1]{GARGDMAM} for the so-called John domains,
and this class contains the Lipschitz domains but it is much larger. However the external cusps are
not allowed either.
\end{proof}

\begin{lemma}\label{lemma:2.1}
Let $F\in W^{-1,q}(\Omega;\mathbb{R}^d)$ with $1<q<\infty$ be a vector-valued functional.
A necessary and sufficient condition to find a solution $P\in L^q(\Omega)/\mathbb{R}$ to
$\nabla P = F$ in $\Omega $,
is that $\big<F,\phi\big>=0$ for any $\phi\in W^{1,p}_0(\Omega;\mathbb{R}^d)$ satisfying
$\emph{div}(\phi) = 0$ in $\Omega$, where $1/p+1/q=1$. Moreover, we have
\begin{equation}\label{pri:2.1}
 \|P\|_{L^q(\Omega)/\mathbb{R}} \leq C\|F\|_{W^{-1,q}(\Omega)},
\end{equation}
where $C$ depends on $d$ and $\Omega$.
\end{lemma}

\begin{proof}
The first part of the lemma actually follows from \cite[Proposition 1.1]{RT}, and we give some remarks here.
Let $1<p,q<\infty$ and $1/p+1/q =1$. It is well known that $\text{div}: W^{1,p}_0(\Omega;\mathbb{R}^d)\to L^p(\Omega)/\mathbb{R}$
is an unbounded linear operator, and closed. By noting that
the dual spaces of $W^{1,p}_0(\Omega;\mathbb{R}^d)$ and $L^p(\Omega)/\mathbb{R}$ are
$W^{-1,q}(\Omega;\mathbb{R}^d)$ and $L^q(\Omega)/\mathbb{R}$, respectively,
it is well defined that $\nabla:L^q(\Omega)/\mathbb{R}\to W^{-1,q}(\Omega;\mathbb{R}^d)$,
which is the adjoint operator of ``$\text{div}$'' satisfying
\begin{equation}\label{f:2.1}
 \big<\nabla v, u\big>_{W^{-1,q}(\Omega)\times W^{1,p}_0(\Omega)}
 = -\big<v,\text{div} (u)\big>_{(L^q(\Omega)/\mathbb{R})\times (L^p(\Omega)/\mathbb{R})}
 \qquad \forall~ v\in L^q(\Omega)/\mathbb{R},\quad \forall~ u\in W^{1,p}_0(\Omega;\mathbb{R}^d).
\end{equation}
By some standard orthogonality relations between ranges and kernels,
we have $\mathcal{R}(\nabla) = \mathcal{N}(\text{div})^{\bot}$,
where $\mathcal{R}(\nabla)$ denotes the range of ``$\nabla$'',
and $\mathcal{N}(\text{div})$ represents the kernel of ``$\text{div}$''.
Note that $\mathcal{N}(\text{div})^{\bot}=
\big\{f\in W^{-1,p}(\Omega;\mathbb{R}^d): \big<f,u\big>=0 ~~\forall u\in\mathcal{N}(\text{div})\big\}$.
That means $P$ will be a solution of $\nabla P = F$ in $\Omega$, if and only if
$F\in\mathcal{N}(\text{div})^{\bot}$. Also, it is not hard to see that $\mathcal{N}(\nabla)
= \mathcal{R}(\text{div})^{\bot}=\{0\}$,
and this implies the uniqueness of the solution $P$ in $L^q(\Omega)/\mathbb{R}$.

Now we turn to show the estimate $\eqref{pri:2.1}$. This estimate could be derived by a standard functional
analysis argument as that in \cite{MGMG,RT}. Due to Lemma $\ref{lemma:2.2}$, we take a constructive way
to prove it so that the constant in the estimate can be clearly tracked.
Let $f\in L^p(\Omega)/\mathbb{R}$, and $u\in W^{1,p}_0(\Omega;\mathbb{R}^d)$ be the solution to
$\text{div}(u) = f$ in $\Omega$. Then it follows from the equality $\eqref{f:2.1}$ that
\begin{equation*}
\big<P-c,f\big>_{L^q(\Omega)\times L^p(\Omega)}
= \big<P,f\big>_{(L^q(\Omega)/\mathbb{R})\times (L^p(\Omega)/\mathbb{R}) }
= -\big<F, u\big>_{W^{-1,q}(\Omega)\times W^{1,p}_0(\Omega)}
\end{equation*}
for any $c\in \mathbb{R}$, and in view of the estimate $\eqref{pri:2.2}$ we have
\begin{equation*}
\left|\big<P-c,f\big>_{L^q(\Omega)\times L^p(\Omega)}\right|
 \leq \|F\|_{W^{-1,q}(\Omega)}\|u\|_{W^{1,p}_0(\Omega)}
 \leq C\|F\|_{W^{-1,q}(\Omega)}\|f\|_{L^{p}_0(\Omega)},
\end{equation*}
and this implies that $\inf_{c\in\mathbb{R}}\|P-c\|_{L^q(\Omega)}\leq C\|F\|_{W^{-1,q}(\Omega)}$.
Thus the estimate $\eqref{pri:2.1}$ holds, and the proof is complete.
\end{proof}

\begin{thm}\label{thm:2.1}
Suppose $A$ satisfies $\eqref{a:1}$. Let $F\in H^{-1}(\Omega;\mathbb{R}^d)$,
$h\in L^2(\Omega)$ and $g\in H^{\frac{1}{2}}(\partial\Omega;\mathbb{R}^d)$ with the compatibility condition
$\eqref{a:4}$. Then the Dirichlet problem $(\textbf{DS}_\varepsilon)$ has
a unique weak solution $(u_\varepsilon,p_\varepsilon)\in H^1(\Omega;\mathbb{R}^d)\times L^2(\Omega)/\mathbb{R}$,
and we have the uniform estimate
\begin{equation}\label{pri:2.6}
\|u_\varepsilon\|_{H^1(\Omega)} + \|p_\varepsilon\|_{L^2(\Omega)/\mathbb{R}}
\leq C\big\{\|F\|_{H^{-1}(\Omega)}+\|h\|_{L^2(\Omega)}+\|g\|_{H^{\frac{1}{2}}(\partial\Omega)}\big\},
\end{equation}
where $C$ depends only on $\mu$ and $\Omega$.
\end{thm}

\begin{proof}
The proof is standard and may be found in \cite[pp.22-23]{RT}. This is the special case of $p=2$ in Theorem
$\ref{thm:1.2}$, and we provide a proof for the sake of completeness.
Due to Lemma $\ref{lemma:2.2}$ we may assume $h=0$ in $\Omega$ and $g=0$ on $\partial\Omega$.
Define $H = \big\{u\in H_0^1(\Omega;\mathbb{R}^d);\text{div}(u)=0~\text{in}~\Omega\big\}$, and
it is well defined here (see \cite[Theorem 1.6]{RT}). Let $B_\varepsilon[\cdot,\cdot]: H\times H\to\mathbb{R}$
be the bilinear form shown in Definition $\ref{def:2.1}$. It follows from the condition $\eqref{a:1}$ that
the bilinear form $B_\varepsilon[\cdot,\cdot]$ holds the coercivity and the boundedness, respectively.
In view of the Lax-Milgram theorem (see for example \cite{MGLM}),
for any $F\in H^{-1}(\Omega;\mathbb{R}^d)$, there exists a unique solution $u_\varepsilon\in H$
such that $B_\varepsilon[u_\varepsilon,v]
= \big<F,v\big>_{H^{-1}(\Omega)\times H_0^{1}(\Omega)}$ for any $v\in H$.
Meanwhile it is not hard to derive that
$\|u_\varepsilon\|_{H^1_0(\Omega)}\leq C\|F\|_{H^{-1}(\Omega)}$.
Now it is clear to see that $F-\mathcal{L}_\varepsilon(u_\varepsilon)$ belongs to
$H^{-1}(\Omega;\mathbb{R}^d)$. From Lemma $\ref{lemma:2.1}$ we know that there exists
$p_\varepsilon\in L^2(\Omega)/\mathbb{R}$ satisfying
$\nabla p_\varepsilon = F-\mathcal{L}_\varepsilon(u_\varepsilon)$ and the estimate
$\|p_\varepsilon\|_{L^2(\Omega)/\mathbb{R}}
\leq C\|\nabla p_\varepsilon\|_{H^-1(\Omega)}\leq C\big\{\|F\|_{H^{-1}(\Omega)}
+\|u_\varepsilon\|_{H^1_0(\Omega)}\big\}\leq C\|F\|_{H^{-1}(\Omega)}$.
Thus we acquire the existence of the solution
$(u_\varepsilon,p_\varepsilon)\in H\times L^2(\Omega)/\mathbb{R}$
to $(\textbf{DS}_\varepsilon)$ and the estimate $\eqref{pri:2.6}$
with the conditions $h=0$ in $\Omega$ and $g=0$ on $\partial\Omega$.
The nonhomogeneous cases follows the same arguments used in the proof of Theorem $\ref{thm:1.2}$, and
we have completed the proof.
\end{proof}

\subsection{Smoothing operator and its properties}\label{subsection:2.1}

\begin{definition}\label{def:2.1}
Fix $\zeta\in C_0^\infty(B(0,1/2))$, and $\int_{\mathbb{R}^d}\zeta(x)dx = 1$. Define the smoothing operator
\begin{equation}
S_\varepsilon(f)(x) = f*\zeta_\varepsilon(x) = \int_{\mathbb{R}^d} f(x-y)\zeta_\varepsilon(y) dy,
\end{equation}
where $\zeta_\varepsilon=\varepsilon^{-d}\zeta(x/\varepsilon)$.
\end{definition}

\begin{lemma}
Let $\Psi\in L^p(\mathbb{R}^d)$ for some $1\leq p<\infty$. Then for any $\rho\in L_{per}^p(\mathbb{R}^d)$,
\begin{equation}\label{pri:2.7}
\big\|\rho(\cdot/\varepsilon)S_\varepsilon(\Psi)\big\|_{L^p(\mathbb{R}^d)}
\leq C\big\|\rho\big\|_{L^p(Y)}\big\|\Psi\big\|_{L^p(\mathbb{R}^d)},
\end{equation}
where $C$ depends only on $d$.
\end{lemma}

\begin{proof}
See \cite[Lemma 2.1]{SZW12}.
\end{proof}

\begin{lemma}
Let $\Psi\in W^{1,p}(\mathbb{R}^d)$ for some $1<p<\infty$. Then we have
\begin{equation}\label{pri:2.8}
\big\|S_\varepsilon(\Psi)-\Psi\big\|_{L^p(\mathbb{R}^d)}
\leq C\varepsilon\big\|\nabla \Psi\big\|_{L^p(\mathbb{R}^d)},
\end{equation}
and further obtain
\begin{equation}\label{pri:2.9}
\big\|S_\varepsilon(\Psi)\big\|_{L^2(\mathbb{R}^d)}\leq C\varepsilon^{-1/2}\big\|\Psi\big\|_{L^q(\mathbb{R}^d)}
\quad\text{and}\quad
\big\|S_\varepsilon(\Psi)-\Psi\big\|_{L^2(\mathbb{R}^d)}
\leq C\varepsilon^{1/2}\big\|\nabla \Psi\big\|_{L^q(\mathbb{R}^d)},
\end{equation}
where $q = 2d/(d+1)$, and $C$ depends only on $d$.
\end{lemma}

\begin{proof}
See \cite[Lemma 2.2]{SZW12}
\end{proof}

\begin{lemma}
 Let $\Psi\in L^2(\Omega)$ be supported in $\Sigma_{2\varepsilon}$, and $\rho\in L_{per}^2(Y)$, then we have
 \begin{equation}\label{pri:2.10}
 \big\|\rho(\cdot/\varepsilon)S_\varepsilon(\Psi)\big\|_{L^2(\Sigma_{2\varepsilon};\delta)}
 \leq C\big\|\rho\big\|_{L^2(Y)}\big\|\Psi\big\|_{L^2(\Sigma_{2\varepsilon};\delta)},
 \end{equation}
 and
 \begin{equation}\label{pri:2.11}
 \big\|\rho(\cdot/\varepsilon)S_\varepsilon(\Psi)\big\|_{L^2(\Sigma_{2\varepsilon};\delta^{-1})}
 \leq C\big\|\rho\big\|_{L^2(Y)}\big\|\Psi\big\|_{L^2(\Sigma_{2\varepsilon};\delta^{-1})},
 \end{equation}
 where $C$ depends at most on $d$ and $\|\zeta\|_{L^\infty(B(0,1/2))}$.
\end{lemma}

\begin{proof}
See \cite[Lemma 3.2]{QX2}.
\end{proof}

\begin{lemma}
Let $\Psi\in H^1(\Omega)$ be supported in $\Sigma_\varepsilon$, then we obtain
\begin{equation}\label{pri:2.12}
\big\|\Psi-S_\varepsilon(\Psi)\big\|_{L^2(\Sigma_{2\varepsilon};\delta)}
\leq C\varepsilon\|\nabla \Psi\|_{L^2(\Sigma_\varepsilon;\delta)},
\end{equation}
where $C$ depends only on $d$.
\end{lemma}

\begin{proof}
See \cite[Lemma 3.3]{QX2}.
\end{proof}

\begin{remark}
Let $f,g\in L^2(\Omega)$, it follows from H\"older's inequality that
\begin{equation}\label{pri:2.14}
\|fg\|_{L^1(\Sigma_{r})}\leq \|f\|_{L^2(\Sigma_r;\delta)}\|g\|_{L^2(\Sigma_r;\delta^{-1})}.
\end{equation}
Moreover, from the estimates $\eqref{pri:2.10}$, $\eqref{pri:2.11}$ and $\eqref{pri:2.12}$ we have
\begin{equation*}
\|\rho(\cdot/\varepsilon)S_\varepsilon(\Psi)f\|_{L^1(\Sigma_{2\varepsilon})}
\leq C\|\rho\|_{L^2(Y)}\min\Big\{\|\Psi\|_{L^2(\Sigma_{2\varepsilon};\delta^{-1})}
\|f\|_{L^2(\Sigma_{2\varepsilon};\delta)},
~\|\Psi\|_{L^2(\Sigma_{2\varepsilon};\delta)}
\|f\|_{L^2(\Sigma_{2\varepsilon};\delta^{-1})}
\Big\}
\end{equation*}
and
\begin{equation*}
\big\|\big[\Psi-S_\varepsilon(\Psi)\big]f\big\|_{L^1(\Sigma_{2\varepsilon})}
\leq C\varepsilon\|\nabla\Psi\|_{L^2(\Sigma_{\varepsilon};\delta)}\|f\|_{L^2(\Sigma_{2\varepsilon};\delta^{-1})}.
\end{equation*}
In fact, the above two inequalities will be employed frequently in the proof of Lemma $\ref{lemma:4.1}$.
\end{remark}

\subsection{Correctors and its properties}
Let $Y = [0,1)^d\backsimeq \mathbb{R}^d/\mathbb{Z}^d$.
Define the correctors
$(\chi_k^{\beta\gamma},\pi_k^\gamma)\in H^1_{per}(Y)\times L_{per}^2(Y)$
associated with the Stokes
system $(\textbf{DS})_\varepsilon$ by the following cell problem:
\begin{equation}\label{pde:2.1}
\left\{\begin{aligned}
\mathcal{L}_1(\chi_k^\gamma+P_k^\gamma) + \nabla \pi_k^\gamma & = 0 \qquad \text{in}~~\mathbb{R}^d,\\
\text{div}(\chi_k^\gamma) &  = 0  \qquad \text{in}~~\mathbb{R}^d,\\
\int_Y \chi_k^\gamma dy &= 0, \qquad k,\gamma = 1,\cdots,d, 
\end{aligned}\right.
\end{equation}
where $P_k^\gamma = y_k e^\gamma = y_k(0,\cdots,1,\cdots,0)$ with 1 in the $\gamma^{\text{th}}$ position,
and $\chi_k^\gamma = (\chi_k^{1\gamma},\cdots,\chi_k^{d\gamma})$. It follows
from Theorem $\ref{thm:2.1}$ that
\begin{equation}\label{pri:2.4}
 \|\nabla \chi_k^{\gamma}\|_{L^2(Y)} + \|\pi_k^\gamma\|_{L^2(Y)} \leq C,
\end{equation}
where $C$ depends only on $\mu$ and $d$.
Then the homogenized operator is given by $\mathcal{L}_0 = -\text{div}(\widehat{A}\nabla)$,
where $\widehat{A} = (\hat{a}_{ij}^{\alpha\beta})$ and
\begin{equation}\label{eq:2.1}
 \hat{a}_{ij}^{\alpha\beta} = \int_Y \Big[a_{ij}^{\alpha\beta}
 + a_{ik}^{\alpha\gamma}\frac{\partial}{\partial y_k}\big(\chi_j^{\gamma\beta}\big) \Big]dy
\end{equation}
(see \cite{SGZWS,ABJLGP}).

\begin{lemma}\label{lemma:2.3}
Let
\begin{equation*}
 b_{ik}^{\alpha\gamma}(y) = \hat{a}_{ik}^{\alpha\gamma}
 - a_{ik}^{\alpha\gamma}(y) - a_{ij}^{\alpha\beta}(y)\frac{\partial\chi_k^{\beta\gamma}}{\partial y_j}(y),
\end{equation*}
where $y=x/\varepsilon$. Then the quantity $b_{ik}^{\alpha\gamma}$ satisfies two properties:
\emph{(i)} $\int_Y b_{ik}^{\alpha\gamma}(y) dy = 0;$
\emph{(ii)} $\nabla_i b_{ik}^{\alpha\gamma} = -\nabla_\alpha\pi_k^\gamma$.
Moreover,
there exist $E_{jik}^{\alpha\gamma}\in H^1_{per}(Y)$ and
$q_{ik}^\gamma\in H_{per}^1(Y)$ such that
\begin{equation}\label{eq:2.2}
 b_{ik}^{\alpha\gamma} = \nabla_j E_{jik}^{\alpha\gamma} - \nabla_\alpha q_{ik}^{\gamma},
 \qquad
 E_{jik}^{\alpha\gamma} = - E_{ijk}^{\alpha\gamma},
 \qquad \text{and}\qquad
 \nabla_i q_{ik}^\gamma = \pi_k^\gamma,
\end{equation}
and $E_{jik}^{\alpha\gamma}$ and $q_{ik}^{\gamma}$ admit the priori estimate
\begin{equation}\label{pri:2.3}
\|E_{jik}^{\alpha\gamma}\|_{L^2(Y)} + \|q_{ik}^\gamma\|_{L^2(Y)}\leq C,
\end{equation}
where $C$ depends only on $\mu$ and $d$.
\end{lemma}

\begin{proof}
The proof may be found in \cite[Lemma 3.1]{SGZWS}, we provide a proof for the sake of the completeness.
First of all, it is clear to see that the formula $\eqref{eq:2.1}$ implies the property (i), and then
the first line of $\eqref{pde:2.1}$ admits the property (ii).
Let $b_{ik}^\gamma = (b_{ik}^{1\gamma},\cdots,b_{ik}^{d\gamma})$,
and we construct the auxiliary cell problem as
follows
\begin{equation}\label{pde:2.2}
\left\{
\begin{aligned}
\Delta T_{ik}^{\gamma} + \nabla q_{ik}^{\gamma} & = b_{ik}^{\gamma} \qquad~\text{in}~~Y,\\
\text{div}(T_{ik}^{\gamma}) &= 0 \qquad\quad \text{in}~~Y,\\
\int_Y T_{ik}^\gamma(y) dy &= 0,~\int_Y q_{ik}^\gamma(y) dy = 0,
 ~\text{and}~T_{ik}^\gamma,q_{ik}^\gamma \text{~are~Y-periodic},
\end{aligned}\right.
\end{equation}
where $i,k,\gamma = 1,\cdots,d$.
The existence of the solution
$(T_{ik}^\gamma,q_{ik}^\gamma)\in H^1_{loc}(\mathbb{R}^d;\mathbb{R}^d)\times L^2_{loc}(\mathbb{R}^d)$
to the equation $\eqref{pde:2.2}$ is based upon
the property (i) and Lemma $\ref{lemma:2.1}$. In fact, the solution $(T_{ik}^{\gamma},q_{ik}^\gamma)$
belongs to $H^2_{loc}(\mathbb{R}^d;\mathbb{R}^d)\times H^1_{loc}(\mathbb{R}^d)$ according
to $H^2$-regularity theory (see \cite[Theorem 1.3]{MGMG}),
since $b_{ik}^{\alpha\gamma}\in L^2_{loc}(\mathbb{R}^d)$.

Then we proceed to prove $\eqref{eq:2.2}$.
Set $E_{jik}^{\alpha\gamma} = \nabla_j T_{ik}^{\alpha\gamma} - \nabla_i T_{jk}^{\alpha\gamma}$, and
a direct result is $E_{jik}^{\alpha\gamma} = - E_{ijk}^{\alpha\gamma}$.
Then we find
\begin{equation*}
\frac{\partial}{\partial y_j}\big(E_{jik}^{\alpha\gamma}\big) =
 \Delta{T_{ik}^{\alpha\gamma}} - \frac{\partial^2}{\partial y_j\partial y_i}\big(T_{jk}^{\alpha\gamma}\big)
 = b_{ik}^{\alpha\gamma} - \frac{\partial}{\partial y_\alpha}\big(q_{ik}^\gamma\big)
 - \frac{\partial^2}{\partial y_j\partial y_i}\big(T_{jk}^{\alpha\gamma}\big).
\end{equation*}
To obtain $\nabla_j E_{jik}^{\alpha\gamma} = b_{ik}^{\alpha\gamma} - \nabla_\alpha q_{ik}^\gamma$,
it only needs to prove $\nabla^2_{ij} T_{jk}^{\alpha\gamma} = 0$. In view of (ii) and
$\eqref{pde:2.2}$, we have
\begin{equation*}
\left\{
\begin{aligned}
 \Delta\Big(\frac{\partial T_{ik}^{\alpha\gamma}}{\partial y_i}\Big)
 + \nabla_\alpha\Big(\frac{\partial q_{ik}^\gamma}{\partial y_i}\Big)
 = \frac{\partial b_{ik}^{\alpha\gamma}}{\partial y_i}
&= \nabla_\alpha \pi_k^\gamma &\quad &\text{in}\quad Y, \\
 \nabla_\alpha \Big(\frac{\partial T_{ik}^{\alpha\gamma}}{\partial y_i}\Big) &= 0 &\quad&\text{in}\quad Y.
\end{aligned}\right.
\end{equation*}
This implies $\nabla_i T_{ik}^{\alpha\gamma}$ is a constant,
(taking $\nabla_i T_{ik}^{\alpha\gamma}$ as a test function and integrating by parts, it is not hard to
derive $\int_Y |\nabla(\nabla_i T_{ik}^{\alpha\gamma})|^2dy = 0$,)
thus we have $\nabla^2_{ij}T_{jk}^{\alpha\gamma} = 0$. Also, the above equation shows
the difference between $\pi_k^\gamma$ and $\nabla_i q_{ik}^\gamma$ is a constant.
However the fact that $\int_Y\pi_{k}^\gamma(y)dy = 0$ and $\int_Y \nabla_i q_{ik}^\gamma(y) dy = 0$
(because of its periodicity) gives $\nabla_i q_{ik}^\gamma = \pi_k^\gamma$.
Finally, it follows from the equation $\eqref{pde:2.2}$ and Theorem $\ref{thm:2.1}$ that
\begin{equation}\label{f:2.2}
 \|\nabla T_{ik}^{\gamma}\|_{L^2(Y)} + \|q_{ik}^\gamma\|_{L^2(Y)} \leq C\|b_{ik}^\gamma\|_{L^2(Y)},
\end{equation}
and this together with $\eqref{pri:2.4}$ gives the estimate $\eqref{pri:2.3}$. This proof is complete.
\end{proof}

\begin{lemma}[Caccioppoli's inequality]
Suppose that $A$ satisfies $\eqref{a:1}$ and $\eqref{a:2}$.
Let the corrector $(\chi_k^{\gamma},\pi_k^\gamma)\in H^1_{per}(Y;\mathbb{R}^d)\times L^2_{per}(Y)/\mathbb{R}$ be the
solution to $\mathcal{L}_1(\chi_k^\gamma+P_k^\gamma) + \nabla \pi_k^\gamma = 0$ and
$\emph{div}(\chi_k^\gamma) = 0$ in $\mathbb{R}^d$.
Then for any $B\subset2B$ with $r>0$ and for any $c\in \mathbb{R}^d$, we have
\begin{equation}\label{pri:2.13}
\int_{B} |\nabla \chi_k^\gamma|^2 dy
\leq \frac{C}{r^2}\Big\{\int_{2B\setminus B}|\chi_k^\gamma - c|^2 dy + r^{d+2}\Big\},
\end{equation}
where $C$ depends on $\mu$ and $d$.
\end{lemma}

\begin{proof}
The proof is standard, and we provide a proof for the sake of completeness.
Let $\psi\in C^\infty_0(\mathbb{R})$ be a cut-off function such that $\psi = 1$ in $B$ and
$\psi = 0$ outside $2B$, and $|\nabla\psi|\leq C/r$. Then let $\psi^2(\chi_k^{\gamma} -c)^\alpha$ be a
test function, where $c\in\mathbb{R}^d$, and we have
\begin{equation*}
\begin{aligned}
\int_{\mathbb{R}^d}\psi^2
\Big[a_{ij}^{\alpha\beta}(y)\frac{\partial\chi_k^{\beta\gamma}}{\partial y_j} +
a_{ik}^{\alpha\gamma}(y)\Big]\frac{\partial \chi_k^{\alpha\gamma}}{\partial y_i}dy
+\int_{\mathbb{R}^d}\psi\Big[a_{ij}^{\alpha\beta}(y)\frac{\partial\chi_k^{\beta\gamma}}{\partial y_j}
&+ a_{ik}^{\alpha\gamma}(y)\Big]\frac{\partial\psi}{\partial x_i}\big(\chi_{k}^{\gamma}
-c\big)^\alpha dy\\
&-2\int_{\mathbb{R}^d}(\pi_k^\gamma-c_1)\frac{\partial\psi}{\partial x_\alpha}\psi
\big(\chi_k^{\gamma}-c\big) dy =0.
\end{aligned}
\end{equation*}
for any $c_1\in\mathbb{R}$. It follows from Young's inequality that
\begin{equation*}
\int_{\mathbb{R}^d}\psi^2|\nabla \chi_k^{\gamma}|^2 dy
\leq C\bigg\{\int_{\mathbb{R}^d}|\nabla\psi|^2|\chi_k^\gamma-c|^2 dy
+\int_{\mathbb{R}^d}\psi^2 dy\bigg\}
+\theta\int_{\mathbb{R}^d} \psi^2|\pi_k^\gamma-c_1|^2 dy
\end{equation*}
where $C$ depends on $\mu$, $\theta$ and $d$. This implies
\begin{equation}\label{f:2.3}
\int_{B}|\nabla\chi_k^{\gamma}|^2 dy \leq C\bigg\{\frac{1}{r^2}\int_{2B\setminus B}|\chi_k^\gamma-c|^2dy
+r^d\bigg\} + \theta\|\pi_k^\gamma\|_{L^2(2B)/\mathbb{R}}^2.
\end{equation}
Due to $\nabla \pi_k^\gamma = -\mathcal{L}_1(\chi_k^\gamma+P_k^\gamma)$, it follows from the estimate
$\eqref{pri:2.1}$ that
\begin{equation}\label{f:2.4}
 \|\pi_k^\gamma\|_{L^2(2B)/\mathbb{R}}
 \leq C\|\mathcal{L}_1(\chi_k^\gamma+P_k^\gamma)\|_{H^{-1}(2B)}
 \leq C\Big\{\|\nabla\chi_k^\gamma\|_{L^2(2B)}
 +r^{\frac{d}{2}}\Big\}.
\end{equation}
Combining $\eqref{f:2.3}$ and $\eqref{f:2.4}$ leads to
\begin{equation*}
\int_{B}|\nabla\chi_k^{\gamma}|^2 dy \leq C\bigg\{\frac{1}{r^2}\int_{2B\setminus B}|\chi_k^\gamma-c|^2dy
+r^d\bigg\} + \theta^\prime\int_{2B}|\nabla \chi_k^{\gamma}|^2 dy
\end{equation*}
where we may let $\theta^\prime \ll 1$ by choosing $\theta$.
This together with \cite[Lemma 0.5]{MGMG} gives the desired estimate $\eqref{pri:2.13}$, and
we have completed the proof.
\end{proof}

\subsection{Non-tangential $\&$ radial  maximal functions}

\begin{definition}\label{def:2.4}
\emph{The non-tangential maximal function of $u$ is defined by
\begin{equation}
(u)^*(Q) = \sup\big\{ |u(x)|:x\in \Gamma_{N_0}(Q)\big\} \qquad\quad \forall~ Q\in\partial\Omega,
\end{equation}
where $\Gamma_{N_0}(Q) = \{x\in\Omega:|x-Q|\leq N_0\delta(x)\}$ is the cone with vertex $Q$ and aperture $N_0$,
and $N_0>1$ depends on the character of $\Omega$.}
\end{definition}

\begin{remark}\label{re:2.2}
\emph{For $0\leq r< c_0$,
we may assume that there exist homeomorphisms $\Lambda_r: \partial\Omega\to \partial\Sigma_r = S_r$ such that $\Lambda_0(Q) = Q$,
$|\Lambda_r(Q) - \Lambda_t(P)| \sim |r-t| + |Q-P|$ and
$|\Lambda_r(Q) - \Lambda_t(Q)|\leq C\text{dist}(\Lambda_r(Q),S_t)$ for any
$r>s$ and $P,Q\in\partial\Omega$ (which are bi-Lipschitz maps, see \cite[pp.1014]{SZW2}).
Especially, we may have
$\max_{r\in[0,c_0]}\{\|\nabla\Lambda_r\|_{L^\infty(\partial\Omega)},
\|\nabla(\Lambda_r^{-1})\|_{L^\infty(\partial\Omega)}\} \leq C$, where $C$
depends only on the Lipschitz character of $\Omega$.
We also refer the reader to \cite[Theorem 5.1]{CTM} for the existence of such bi-Lipschitz maps.}
\end{remark}

\begin{definition}\label{def:2.3}
We define the radial maximal function
$\mathcal{M}(h)$ on $\partial\Omega$ as
\begin{equation}\label{eq:2.4}
 \mathcal{M}(h)(Q) = \sup\big\{|h(\Lambda_r(Q))|: 0\leq r\leq  c_0\big\} \quad\qquad \forall ~Q\in\partial\Omega.
\end{equation}
\end{definition}
We mention that the radial maximal function will play an important role in
the study of convergence rates for Lipschitz domains,
and we refer the reader to\cite{SZW2} for the original idea.

\begin{remark}
\emph{Let $h\in L^p(\Omega)$ with $1\leq p<\infty$, and
$\Lambda_r$ be given in Remark $\ref{re:2.2}$. For any $r\in(0,c_0)$
we can show that
\begin{equation}\label{pri:2.5}
\begin{aligned}
 \int_{\Omega\setminus\Sigma_r} |h|^p dx
 &= \int_0^r\int_{S_t=\Lambda_{t}(\partial\Omega)} |h(y)|^p dS_t(y)dt \\
 & = \int_0^r\int_{\partial\Omega} |h(\Lambda_t(z))|^p |\nabla\Lambda_{t}|dS(z)dt \leq Cr\int_{\partial\Omega} |\mathcal{M}(h)|^p dS
 \leq  Cr\int_{\partial\Omega} |(h)^{*}|^p dS,
\end{aligned}
\end{equation}
where we note that
$h(\Lambda_r(x))\leq \mathcal{M}(h)(x)$ a.e. $x\in\partial\Omega$ for all $r\in (0,c_0)$ in view of
$\eqref{eq:2.4}$, and
$C$ depends only on $p$ and the Lipschitz character of $\Omega$. Concerning the above estimate,
we note that the first equality is based on
the co-area formula $\eqref{eq:2.3}$, and we use the change of variable in the second one. Besides,
the first inequality follows from $\eqref{eq:2.4}$. In the last one,
it is not hard to see $\mathcal{M}(h)(Q)\leq (h)^*(Q)$ by comparing Definition $\ref{def:2.4}$ with Definition $\ref{def:2.3}$.}

\emph{We now explain the co-area formula used here.
Let $Z(0;r)=\{x\in\Omega:0<\delta(x)\leq r\}$,
then $Z(0;r) = \Omega\setminus\Sigma_r$.
Here we point out $|\nabla\delta(x)| =1$ a.e. $x\in\Omega$ without the proof (see \cite[pp.142]{LCE1}).
In view of co-area formula (see \cite[Theorem 3.13]{LCE1}), we have
\begin{equation}\label{eq:2.3}
 \int_{\Omega\setminus\Sigma_r} |h|^p dx = \int_{Z(0;r)} |h|^p dx
 = \int_{0}^{r}\int_{\{x\in\Omega:\delta(x)=t\}}\frac{|h|^p}{|\nabla\delta|}d\mathcal{H}^{d-1}dt
 =\int_{0}^{r}\int_{S_t}|h|^p dS_tdt,
\end{equation}
where $d\mathcal{H}^{d-1}$ is the ($d-1$)-dimensional Hausdorff measure,
and $dS_t=d\mathcal{H}^{d-1}(S_t)$ denotes the surface measure of $S_t$.}
\end{remark}

\begin{lemma}
Let the radial maximal operator $\mathcal{M}$ be given in Definition $\ref{def:2.3}$.
Then for any $h\in H^1(\Omega)$, we have the following estimate
\begin{equation}\label{pri:2.15}
 \|\mathcal{M}(h)\|_{L^2(\partial\Omega)}
 \leq C\|h\|_{H^1(\Omega\setminus\Sigma_{c_0})},
\end{equation}
where $C$ depends only on $d,c_0$ and the character of $\Omega$.
\end{lemma}

\begin{proof}
The proof may be found in \cite[Lemma 2.24]{QXS1}.
\end{proof}

\section{$O(\varepsilon^{\frac{1}{2}})$ convergence rate in $H_0^1(\Omega)\times L^2(\Omega)/\mathbb{R}$}

\begin{lemma}\label{lemma:3.1}
Suppose that $(u_\varepsilon,p_\varepsilon),(u_0,p_0)\in H^1(\Omega;\mathbb{R}^d)\times L^2(\Omega)$ satisfy
\begin{equation*}
\left\{\begin{aligned}
\mathcal{L}_\varepsilon(u_\varepsilon) + \nabla p_\varepsilon &=
\mathcal{L}_0(u_0) + \nabla p_0 & \quad & \emph{in}~~\Omega,\\
\emph{div}(u_\varepsilon) & = \emph{div}(u_0) &\quad & \emph{in}~~\Omega,\\
u_\varepsilon & = u_0 &\quad& \emph{on}~\partial\Omega.
\end{aligned}
\right.
\end{equation*}
Let $w_\varepsilon = (w_\varepsilon^\beta)$ with $w_\varepsilon^\beta = u_\varepsilon^\beta - u_0^\beta
- \varepsilon \chi_{k}^{\beta\gamma}(\cdot/\varepsilon)\varphi_k^\gamma$, where
$\varphi = (\varphi_k^\gamma)\in H_0^1(\Omega;\mathbb{R}^{d\times d})$. Then we have
\begin{equation}\label{pde:3.1}
\left\{\begin{aligned}
\mathcal{L}_\varepsilon(w_\varepsilon) + \nabla (p_\varepsilon - p_0) &=
-\emph{div}(f) & \quad & \emph{in}~~\Omega,\\
\emph{div}(w_\varepsilon) & = \eta &\quad & \emph{in}~~\Omega,\\
w_\varepsilon & = 0 &\quad& \emph{on}~\partial\Omega,
\end{aligned}
\right.
\end{equation}
and the compatibility condition
\begin{equation}\label{eq:3.3}
 \int_\Omega \eta(x) dx = 0,
\end{equation}
where $\eta = -\varepsilon\chi_{k}^{\beta\gamma}(\cdot/\varepsilon)\nabla_\beta\varphi_k^\gamma$, and
$f = (f_i^\alpha)$ with
\begin{equation*}
f_i^\alpha = b_{ik}^{\alpha\gamma}(\cdot/\varepsilon)\varphi_k^\gamma
+ \big[\hat{a}_{ij}^{\alpha\beta} - a_{ij}^{\alpha\beta}(\cdot/\varepsilon)\big]
\big[\nabla_j u_0^\beta - \varphi_j^\beta\big]
-\varepsilon a_{ij}^{\alpha\beta}(\cdot/\varepsilon)
\chi_{k}^{\beta\gamma}(\cdot/\varepsilon)\nabla_j\varphi_k^\gamma.
\end{equation*}
\end{lemma}

\begin{proof}
Based on the first observation that
$\mathcal{L}_\varepsilon(u_\varepsilon) = \mathcal{L}_0(u_0) + \nabla (p_0 - p_\varepsilon)$, we have
\begin{equation}\label{eq:3.1}
\begin{aligned}
\big[\mathcal{L}_\varepsilon(w_\varepsilon)\big]^\alpha + \nabla_\alpha\big(p_\varepsilon - p_0\big)&=
 \big[\mathcal{L}_\varepsilon(u_\varepsilon)\big]^\alpha -
 \big[\mathcal{L}_\varepsilon(u_0)\big]^\alpha -
 \big[\mathcal{L}_\varepsilon\big(\varepsilon\chi_{k,\varepsilon}^{\gamma}\varphi_k^\gamma\big)\big]^\alpha
 + \nabla_\alpha\big(p_\varepsilon - p_0\big)\\
& =  \big[\mathcal{L}_0(u_0)\big]^\alpha -
\big[\mathcal{L}_\varepsilon(u_0)\big]^\alpha -
\big[\mathcal{L}_\varepsilon\big(\varepsilon\chi_{k,\varepsilon}^{\gamma}\varphi_k^\gamma\big)\big]^\alpha,
\end{aligned}
\end{equation}
where the notation $\chi_{k,\varepsilon}^\gamma$ denotes
$\big[\chi_{k}^{1\gamma}(\cdot/\varepsilon),\cdots,\chi_{k}^{d\gamma}(\cdot/\varepsilon)\big]$
throughout the proof. The right-hand side of $\eqref{eq:3.1}$ is equal to
\begin{equation}\label{eq:3.2}
\begin{aligned}
-\frac{\partial}{\partial x_i}\bigg\{\big[\hat{a}_{ik}^{\alpha\gamma} - a_{ik}^{\alpha\gamma}(y)
&-a_{ij}^{\alpha\beta}(y)\frac{\partial\chi_k^{\beta\gamma}}{\partial y_j}(y)\big]\varphi_k^\gamma
 \\
& + \big[\hat{a}_{ij}^{\alpha\beta}-a_{ij}^{\alpha\beta}(y)\big]
\big[\frac{\partial u_0^\beta}{\partial x_j} - \varphi_j^\beta\big]
-\varepsilon a_{ij}^{\alpha\beta}(y)\chi_{k}^{\beta\gamma}(y)
\frac{\partial\varphi_k^\gamma}{\partial x_j}\bigg\}
\end{aligned}
\end{equation}
where $y = x/\varepsilon$. By noting that
$b_{ik}^{\alpha\gamma}(y) = \hat{a}_{ik}^{\alpha\gamma} - a_{ik}^{\alpha\gamma}(y)
-a_{ij}^{\alpha\beta}(y)\frac{\partial\chi_k^{\beta\gamma}}{\partial y_j}(y)$, we actually obtain the
first line of the equation $\eqref{pde:3.1}$, and its right-hand side $-\text{div}(\tilde{f})$ is exactly
expressed by $\eqref{eq:3.2}$.

Then we turn to study the second line of the equation $\eqref{pde:3.1}$, and
due to $\text{div}(u_\varepsilon) = \text{div}(u_0)$ it follows that
\begin{equation*}
\frac{\partial w_\varepsilon^\beta}{\partial x_\beta}
= \frac{\partial u_\varepsilon^\beta}{\partial x_\beta}
-\frac{\partial u_0^\beta}{\partial x_\beta}
-\frac{\partial}{\partial x_\beta}\big(\varepsilon\chi_{k,\varepsilon}^{\beta\gamma}\varphi_k^\gamma\big)
= -\frac{\partial \chi_k^{\beta\gamma}}{\partial y_\beta}(y)\varphi^\gamma_k
- \varepsilon\chi_{k,\varepsilon}^{\beta\gamma}\frac{\partial \varphi_k^\gamma}{\partial x_\beta}
= - \varepsilon\chi_{k,\varepsilon}^{\beta\gamma}\frac{\partial \varphi_k^\gamma}{\partial x_\beta}
\quad \text{in}~\Omega,
\end{equation*}
where $\chi_{k,\varepsilon}^{\beta\gamma}$ denotes $\chi_k^{\beta\gamma}(\cdot/\varepsilon)$, and
the last equality is because of $\text{div}(\chi^\gamma_k) = 0$ in $\mathbb{R}^d$
with $k,\gamma = 1,\cdots,d$. Moreover, since $u_\varepsilon = u_0$ on $\partial\Omega$
we have
\begin{equation*}
w_\varepsilon = -\varepsilon\chi_{k,\varepsilon}^{\beta\gamma}\varphi_k^\gamma = 0
\quad \text{on}~\partial\Omega,
\end{equation*}
and the last equality is due to $\varphi= (\varphi_k^\gamma)\in H_0^1(\Omega;\mathbb{R}^{d\times d})$.
We complete this proof by checking the compatibility condition
\begin{equation*}
\int_\Omega \eta(x) dx =
-\varepsilon\int_\Omega \chi_{k,\varepsilon}^{\beta\gamma}\nabla_\beta\varphi_k^\gamma dx
= \int_\Omega \text{div}_y(\chi_{k}^\gamma)(y)\varphi_k^\gamma dx
- \varepsilon\int_{\partial\Omega} n_\beta \chi_{k,\varepsilon}^{\beta\gamma}\varphi_k^\gamma dS = 0,
\end{equation*}
where the second equality follows from integration by parts,
and the last equality is due to
$\text{div}(\chi^\gamma_k) = 0$ in $\mathbb{R}^d$ and $\varphi_k^\gamma\in H_0^1(\Omega)$.
\end{proof}


\begin{lemma}\label{lemma:3.2}
Suppose that $A$ satisfies $\eqref{a:1}$ and $\eqref{a:2}$.
Assume that $(u_\varepsilon,p_\varepsilon)$, $(u_0,p_0)$ are two weak solutions to
$(\textbf{DS}_\varepsilon)$ and
$(\textbf{DS}_0)$, respectively.
Let $w_\varepsilon = (w_\varepsilon^\beta)$ with
\begin{equation}\label{eq:3.4}
w_\varepsilon^\beta = u_\varepsilon^\beta - u_0^\beta
- \varepsilon\chi_{k}^{\beta\gamma}(\cdot/\varepsilon)\varphi_k^\gamma
\qquad
\text{and}\qquad
z_\varepsilon = p_\varepsilon - p_0 - \pi_{k}^\gamma(\cdot/\varepsilon)\varphi_k^\gamma
-\varepsilon q_{ik}^\gamma(\cdot/\varepsilon)\nabla_i\varphi_k^\gamma,
\end{equation}
where $\varphi = (\varphi_k^\gamma)\in H_0^1(\Omega;\mathbb{R}^{d\times d})$.
Then $(w_\varepsilon,z_\varepsilon)$ satisfies
\begin{equation}\label{pde:3.2}
\left\{\begin{aligned}
\mathcal{L}_\varepsilon(w_\varepsilon)
+\nabla z_\varepsilon
& = \emph{div}(\tilde{f})
&\quad& \emph{in}~~\Omega,\\
\emph{div}(w_\varepsilon) &= \eta &\quad& \emph{in}~~\Omega,\\
w_\varepsilon &=0 &\quad& \emph{on}~\partial\Omega,
\end{aligned}\right.
\end{equation}
where $\eta = -\varepsilon\chi_{k}^{\beta\gamma}(\cdot/\varepsilon)\nabla_\beta\varphi_k^\gamma$,
and $\tilde{f} = (\tilde{f}_i^\alpha)$ with
\begin{equation}\label{eq:3.5}
\begin{aligned}
\tilde{f}_i^\alpha
= \varepsilon \big[E_{jik}^{\alpha\gamma}(y)+a_{ij}^{\alpha\beta}(y)\chi_k^{\beta\gamma}(y)\big]
\frac{\partial\varphi_k^\gamma}{\partial x_j}
-\varepsilon q_{ik}^\gamma(y)\frac{\partial\varphi_k^\gamma}{\partial x_\alpha}
-\big[\hat{a}_{ij}^{\alpha\beta}-a_{ij}^{\alpha\beta}(y)\big]
\Big[\frac{\partial u_0^\beta}{\partial x_j}-\varphi_j^\beta\Big]
\quad y=x/\varepsilon.
\end{aligned}
\end{equation}
Furthermore, we have
\begin{equation}\label{pri:3.1}
 \|w_\varepsilon\|_{H^1_0(\Omega)} + \|z_\varepsilon\|_{L^2(\Omega)/\mathbb{R}}
 \leq C\Big\{\varepsilon\big\|\varpi(\cdot/\varepsilon)\nabla\varphi\big\|_{L^2(\Omega)}
 + \big\|\nabla u_0 - \varphi\big\|_{L^2(\Omega)}\Big\}
\end{equation}
where the notation $\varpi(\cdot/\varepsilon)$ is explained in Remark $\ref{remark:3.2}$,
and $C$ depends only on $\mu,d$ and $\Omega$.
\end{lemma}

\begin{remark}\label{remark:3.2}
\emph{ All the periodic functions in the paper have already emerged.
They include the coefficient $a_{ij}^{\alpha\beta}$, the correctors $(\chi_k^\gamma,\pi_k^{\gamma})$,
the auxiliary functions $b_{ik}^{\alpha\gamma}, T_{ik}^{\gamma}, q_{ik}^\gamma, E_{jik}^{\alpha\gamma}$.
For simplicity of presentation,
let $\varpi$ denote a kind of universal periodic function,
which plays a similar role as the constant $C$ does in the paper. For example, the notation
$\varpi(\cdot/\varepsilon)$ may represent the terms ``$E_{jik}^{\alpha\gamma}(\cdot/\varepsilon)
+ a_{ij}^{\alpha\beta}(\cdot/\varepsilon)\chi_k^{\beta\gamma}(\cdot/\varepsilon)$'',
``$q_{ik}^\gamma(\cdot/\varepsilon)$'' and ``$\chi_k^{\beta\gamma}(\cdot/\varepsilon)$''.
In view of $\eqref{a:1}$, $\eqref{pri:2.4}$ and $\eqref{pri:2.3}$, it is not hard to see
$\|\varpi\|_{L^2(Y)}\leq C(\mu,d)$. In such a case, one also says $\varpi$ is an universal periodic function
determined by $\mu$ and $d$. The following proof will show how to use the notation $\varpi$, and
we consequently do not repeat it in other places.}
\end{remark}

\begin{proof}
It follows from the first line of the equation $\eqref{pde:3.1}$ that
\begin{equation}\label{eq:3.7}
 \big[\mathcal{L}_\varepsilon(w_\varepsilon)\big]^\alpha
 = -\nabla_if_i^\alpha- \nabla_\alpha(p_\varepsilon - p_0),
 \qquad\text{in}~\Omega.
\end{equation}
To obtain the first line of $\eqref{pde:3.2}$ and the formula $\eqref{eq:3.5}$,
one only needs to check the term
$-\frac{\partial}{\partial x_i}\big[b_{ik}^{\alpha\gamma}(\cdot/\varepsilon)\varphi_k^\gamma\big]$
in $f_i^\alpha$.
In view of $\eqref{eq:2.2}$, we have
\begin{equation}\label{eq:3.6}
\begin{aligned}
\frac{\partial}{\partial x_i}\big[b_{ik}^{\alpha\gamma}(y)\varphi_k^\gamma\big]
&= \varepsilon\frac{\partial}{\partial x_i}
\bigg\{\frac{\partial}{\partial x_j}\big[E_{jik}^{\alpha\gamma}(y)\big]\varphi_k^\gamma
-\frac{\partial}{\partial x_\alpha}\big[q_{ik}^\gamma(y)\big]\varphi_k^\gamma\bigg\} \\
& = \varepsilon\frac{\partial}{\partial x_i}
\bigg\{\frac{\partial}{\partial x_j}\big[E_{jik}^{\alpha\gamma}(y)\varphi_k^\gamma\big]
-E_{jik}^{\alpha\gamma}(y)\frac{\partial\varphi_k^\gamma}{\partial x_j}\bigg\}
  -\varepsilon\frac{\partial}{\partial x_i}
\bigg\{\frac{\partial}{\partial x_\alpha}\big[q_{ik}^{\gamma}(y)\varphi_k^\gamma\big]
-q_{ik}^\gamma(y)\frac{\partial\varphi_k^\gamma}{\partial x_\alpha}\bigg\}.
\end{aligned}
\end{equation}
The second line of $\eqref{eq:3.6}$ is equal to
\begin{equation*}
\varepsilon\frac{\partial^2}{\partial x_i\partial x_j}\big[E_{jik}^{\alpha\gamma}(y)\varphi_k^\gamma\big]
-\varepsilon\frac{\partial}{\partial x_i}
\Big\{E_{jik}^{\alpha\beta}(y)\frac{\partial\varphi^\gamma_k}{\partial x_j}\Big\}
-\varepsilon\frac{\partial^2}{\partial x_\alpha\partial x_i}\big[q_{ik}^\gamma(y)\varphi_k^\gamma\big]
+\varepsilon\frac{\partial}{\partial x_i}
\Big\{q_{ik}^\gamma(y)\frac{\partial\varphi_k^\gamma}{\partial x_\alpha}\Big\},
\end{equation*}
where the first term vanishes because of $E_{jik}^{\alpha\gamma} = - E_{ijk}^{\alpha\gamma}$, and the
third term becomes
\begin{equation*}
-\frac{\partial}{\partial x_\alpha}\Big\{\frac{\partial q_{ik}^\gamma}{\partial y_i}\varphi_k^\gamma
 + \varepsilon q_{ik}^\gamma(y)\frac{\partial\varphi_k^\gamma}{\partial x_i}\Big\}
= -\frac{\partial}{\partial x_\alpha}\Big\{\pi_k^\gamma(y)\varphi_k^\gamma
 + \varepsilon q_{ik}^\gamma(y)\frac{\partial\varphi_k^\gamma}{\partial x_i}\Big\}
\end{equation*}
according to $\nabla_i q_{ik}^\gamma = \pi_k^\gamma$ in $\eqref{eq:2.2}$. Thus we have
\begin{equation}
-\frac{\partial}{\partial x_i}\big[b_{ik}^{\alpha\gamma}(y)\varphi_k^\gamma\big]
= \varepsilon\frac{\partial}{\partial x_i}
\bigg\{E_{jik}^{\alpha\beta}(y)\frac{\partial\varphi^\gamma_k}{\partial x_j}
- q_{ik}^\gamma(y)\frac{\partial\varphi_k^\gamma}{\partial x_\alpha} \bigg\}
+ \frac{\partial}{\partial x_\alpha}\bigg\{\pi_k^\gamma(y)\varphi_k^\gamma
 + \varepsilon q_{ik}^\gamma(y)\frac{\partial\varphi_k^\gamma}{\partial x_i}\bigg\},
\end{equation}
and the right-hand side of $\eqref{eq:3.7}$ is exactly written by the force term
(denoted by $\text{div}(\tilde{f})$)
\begin{equation*}
\frac{\partial}{\partial x_i}\bigg\{
\varepsilon \big[E_{jik}^{\alpha\gamma}(y)+ a_{ij}^{\alpha\beta}(y)\chi_k^{\beta\gamma}(y)\big]
\frac{\partial\varphi_k^\gamma}{\partial x_j}
-\varepsilon q_{ik}^\gamma(y)\frac{\partial\varphi_k^\gamma}{\partial x_\alpha}
-\big[\hat{a}_{ij}^{\alpha\beta}-a_{ij}^{\alpha\beta}(y)\big]
\Big[\frac{\partial u_0}{\partial x_j}-\varphi_j^\beta\Big]\bigg\}
\end{equation*}
adding the pressure term (denoted by $\nabla z_\varepsilon$)
\begin{equation*}
 -\frac{\partial}{\partial x_\alpha}\bigg\{p_\varepsilon-p_0
 -\pi_k^\gamma(y)\varphi_k^\gamma
 - \varepsilon q_{ik}^\gamma(y)\frac{\partial\varphi_k^\gamma}{\partial x_i}
 \bigg\}.
\end{equation*}
Moving the pressure term to the left-hand side,
we consequently obtain the first line of $\eqref{pde:3.2}$ as well as the formula $\eqref{eq:3.5}$.
Now, it is the position to verity the estimate $\eqref{pri:3.1}$. Due to Theorem $\ref{thm:2.1}$,
\begin{equation}\label{f:3.1}
 \|w_\varepsilon\|_{H^1_0(\Omega)} + \|z_\varepsilon\|_{L^2(\Omega)/\mathbb{R}}
 \leq C\Big\{\|\tilde{f}\|_{L^2(\Omega)}+\|\eta\|_{L^2(\Omega)}\Big\}.
\end{equation}
Observing $\eqref{eq:3.5}$, we write that
\begin{equation*}
\big[E_{jik}^{\alpha\gamma}(\cdot/\varepsilon)+ a_{ij}^{\alpha\beta}(\cdot/\varepsilon)
\chi_k^{\beta\gamma}(\cdot/\varepsilon)\big]
\nabla_j\varphi_k^\gamma =: \big[\varpi(\cdot/\varepsilon)\nabla\varphi\big]^\alpha_i,
\end{equation*}
where $\big[\varpi(\cdot/\varepsilon)\big]^{\alpha\gamma}_{ijk}$ represents the term
$E_{jik}^{\alpha\gamma}(\cdot/\varepsilon)+ a_{ij}^{\alpha\beta}(\cdot/\varepsilon)
\chi_k^{\beta\gamma}(\cdot/\varepsilon)$. By the principle explained in Remark $\ref{remark:3.2}$ again,
we denote the second term in $\eqref{eq:3.5}$ by
$q_{ik}^\gamma(\cdot/\varepsilon)\nabla_\alpha\varphi_k^\gamma
=: \big[\varpi(\cdot/\varepsilon)\nabla\varphi\big]_i^\alpha$, where $\big[\varpi\big]_{ik}^\gamma$ denotes
$q_{ik}^\gamma$. Hence we have
\begin{equation}\label{f:3.2}
\|\tilde{f}\|_{L^2(\Omega)}
\leq C\Big\{\varepsilon\|\varpi(\cdot/\varepsilon)\nabla\varphi\|_{L^2(\Omega)}
 + \|\nabla u_0 - \varphi\|_{L^2(\Omega)}\Big\},
\end{equation}
By noting that $\eta = -\varepsilon\chi_k^{\beta\gamma}(\cdot/\varepsilon)\nabla_\beta\varphi_k^\gamma
:=\varepsilon \varpi(\cdot/\varepsilon)\nabla\varphi$,
combining $\eqref{f:3.1}$ and $\eqref{f:3.2}$
gives the estimate $\eqref{pri:3.1}$, and we are done.
\end{proof}

\begin{lemma}
Let $1<p<\infty$. Given $F\in W^{-1,p}(\mathbb{R}^d;\mathbb{R}^d)$, $h\in L^p(\mathbb{R}^d)$, there exists
the unique weak solution $(u,p)\in W^{1,p}(\mathbb{R}^d;\mathbb{R}^d)\times L^{p}(\mathbb{R}^d)$ to the following problem
\begin{equation}\label{pde:3.3}
-\Delta u  + \nabla p = F, \quad \emph{div}(u) = h \qquad \emph{in}~~\mathbb{R}^d,
\end{equation}
satisfying
\begin{equation}\label{pri:3.8}
\|\nabla u\|_{L^p(\mathbb{R}^d)} + \|p\|_{L^p(\mathbb{R}^d)}
\leq C\Big\{\|F\|_{W^{-1,p}(\mathbb{R}^d)} + \|h\|_{L^{p}(\mathbb{R}^d)}\Big\}.
\end{equation}
Moreover, if we assume $F\in L^p(\mathbb{R}^d;\mathbb{R}^d)$, $h\in W^{1,p}(\mathbb{R}^d)$, then the solution $(u,p)$ satisfies
\begin{equation}\label{pri:3.3}
\|\nabla^2 u\|_{L^p(\mathbb{R}^d)} + \|\nabla p\|_{L^p(\mathbb{R}^d)}
\leq C\Big\{\|F\|_{L^p(\mathbb{R}^d)} + \|h\|_{W^{1,p}(\mathbb{R}^d)}\Big\},
\end{equation}
where $C$ depends only on $p$ and $d$.
\end{lemma}

\begin{proof}
The estimate $\eqref{pri:3.8}$ and
the existence and uniqueness of the weak solution to $\eqref{pde:3.3}$ were proved in
\cite[Theorem 3.1]{GPGCG}. Inspired by their arguments
we focus on verifying the estimate $\eqref{pri:3.3}$. It is suffices to prove the lemma when
$F^\alpha,h\in C_0^\infty(\mathbb{R}^d)$, where $F^\alpha$ is the $\alpha^{\text{th}}$ component of $F$.
First of all,
due to the linearity of the equation $\eqref{pde:3.3}$
we can divide it into three parts as follows:
\begin{equation}\label{eq:3.8}
\text{(i)}\left\{\begin{aligned}
-\Delta u_1 + \nabla p_1 &= F &\quad&\text{in}~\mathbb{R}^d,\\
\text{div}(u_1) &= 0 &\quad&\text{in}~\mathbb{R}^d,
\end{aligned}\right.
\qquad
\text{(ii)}\left\{\begin{aligned}
-\Delta u_2 + \nabla p_2 &= \Delta \psi &\quad&\text{in}~\mathbb{R}^d,\\
\text{div}(u_2) &= 0 &\quad&\text{in}~\mathbb{R}^d,
\end{aligned}\right.
\qquad
\text{(iii)} ~\psi = -\nabla\Gamma*h\quad\text{in}~\mathbb{R}^d,
\end{equation}
where $\Gamma$ is the fundamental solution to $-\Delta$.
It is not hard to see $u = u_1 + u_2 + \psi$ and $p = p_1+p_2$.
In fact, according
to fundamental solution to Stokes systems (see \cite[Chapter 3]{OAL})
we are able to figure out the solutions $(u_1,p_1)$ and $(u_2,p_2)$ as following:
\begin{equation}
\left\{\begin{aligned}
u_1 = \mathbf{U}*F, \\
p_1 = \mathbf{Q}*F,
\end{aligned}\right.
\qquad \text{and}\qquad
\left\{\begin{aligned}
u_2 = \mathbf{U}*(\Delta\psi), \\
p_2 = \mathbf{Q}*(\Delta\psi),
\end{aligned}\right.
\end{equation}
where $\mathbf{U} = (U_{ij})$ and $\mathbf{Q} = (Q_{i})$ are fundamental solution to Stokes systems, and
their components are formulated  by
\begin{equation}\label{eq:3.9}
U_{ij}(x) = \left\{ \begin{aligned}
&c_1(d)\bigg[\kappa_{ij}|x|^{2-d}+(d-2)\frac{x_ix_j}{|x|^d}\bigg] &\quad& \text{if}~ d\geq 3,\\
&\frac{1}{4\pi}\bigg[\kappa_{ij}\log|x| - \frac{x_ix_j}{|x|^2}\bigg] &\quad& \text{if}~ d=2;
\end{aligned}
\right.
\qquad
Q_i(x) = \left\{\begin{aligned}
& c_2(d)\frac{x_i}{|x|^d} &\quad& \text{if}~d\geq 3,\\
& \frac{x_i}{2\pi|x|^2}   &\quad& \text{if}~d= 2.
\end{aligned}
\right.
\end{equation}
Here $\kappa_{ij}$ is the Kronecker symbol and they satisfy
$-\Delta \mathbf{U}(x,\cdot) + \nabla \mathbf{Q}(x,\cdot) = \delta_x(\cdot)\mathbf{I}$ and
$\text{div}(\mathbf{U}) = 0$ in $\mathbb{R}^d$, where $\delta_x(\cdot)$ is the Dirac delta function
concentrated at $x$ and $\mathbf{I}$ is the $d\times d$ identity matrix. We mention that we only address
the cases $d\geq 3$ in the following, and the case of $d=2$ holds in the same way,
and we will leave it to the reader.

We now investigate the equation (iii) of $\eqref{eq:3.8}$. Since $h\in C_0^\infty(\mathbb{R}^d)$, it is not
hard to see $\psi = -\nabla\Gamma* h = -\Gamma*\nabla h$ in $\mathbb{R}^d$.
It is well know by singular integral estimate (see \cite[Chapters 4,5]{JD}) that
\begin{equation}\label{pri:3.4}
  \|\nabla^2\psi\|_{L^p(\mathbb{R}^d)}\leq C_p\|\nabla h\|_{L^p(\mathbb{R}^d)}
\end{equation}
holds for $1<p<\infty$, where $C_p$ depends only on $p$ and $d$.
Observing the formulas of the fundamental solutions in $\eqref{eq:3.9}$,
we know that $\mathbf{U}$ is homogeneous of degree $2-d$, and $\mathbf{Q}$ is homogeneous of degree $1-d$,
and then their Fourier transform $\mathcal{F}(\mathbf{Q})$ is homogeneous of degree $-2$ and $-1$, respectively.
Furthermore, it follows from the properties of Fourier transform that
$|\mathcal{F}(\nabla^2\mathbf{U})|$ and $|\mathcal{F}(\nabla\mathbf{Q})|$ are bounded by a constant only
depending on $d$, which actually guarantees the corresponding singular integrals is $L^2$-bounded.
By writing $K_1(x)=\nabla^2\mathbf{U}(x)$ and $K_2(x)=\nabla\mathbf{Q}(x)$,
we have the H\"ormander condition
$\max\{|\nabla K_1(x)|,|\nabla K_2(x)|\}\leq C|x|^{-d-1}$ for every $x\not=0$.
According to the notation above, one immediately obtains $\nabla^2 u_1 = K_1*F$
and $\nabla p_1 = K_2*(\Delta\psi)$, and it is follows from \cite[Theorem 5.1]{JD} that
\begin{equation}\label{pri:3.5}
\begin{aligned}
 &\|\nabla^2 u_1\|_{L^p(\mathbb{R}^d)}+\|\nabla p_1\|_{L^p(\mathbb{R}^d)}\leq C\|F\|_{L^p(\mathbb{R}^d)},\\
 &\|\nabla^2 u_2\|_{L^p(\mathbb{R}^d)}+\|\nabla p_2\|_{L^p(\mathbb{R}^d)}
 \leq C\|\Delta\psi\|_{L^p(\mathbb{R}^d)}
 \leq C\|\nabla h\|_{L^p(\mathbb{R}^d)},
\end{aligned}
\end{equation}
where we use the estimate $\eqref{pri:3.4}$ in the last inequality of the second line of $\eqref{pri:3.5}$,
and $C$ depends on $p$ and $d$. As a consequence,
the estimates $\eqref{pri:3.4}$ and $\eqref{pri:3.5}$ implies the desired result $\eqref{pri:3.3}$ and
we have completed the proof.
\end{proof}

\begin{lemma}\label{lemma:3.3}
Suppose that $A$ satisfies $\eqref{a:1}-\eqref{a:3}$.
Assume $F\in L^q(\Omega;\mathbb{R}^d)$ with $q=\frac{2d}{d+1}$, $h\in W^{1,q}(\Omega)$
and $g\in H^1(\partial\Omega;\mathbb{R}^d)$ with the compatibility condition
$\int_\Omega h dx = \int_{\partial\Omega} n\cdot gdS$.
Let $(u_0,p_0)\in H^1(\Omega;\mathbb{R}^d)\times L^2(\Omega)$
be the weak solution to $(\textbf{DS}_0)$. Then we have
\begin{equation}\label{pri:3.6}
\|\nabla u_0\|_{L^2(\Omega\setminus\Sigma_{p_1\varepsilon})}
+ \|p_0\|_{L^2(\Omega\setminus\Sigma_{p_1\varepsilon})} \leq C\varepsilon^{\frac{1}{2}}
\Big\{\|F\|_{L^q(\Omega)} + \|h\|_{W^{1,q}(\Omega)}+\|g\|_{H^1(\partial\Omega)}\Big\},
\end{equation}
and further assuming $F\in L^2(\Omega;\mathbb{R}^d)$ and $h\in H^1(\Omega)$, there holds
\begin{equation}\label{pri:3.7}
\|\nabla^2u_0\|_{L^2(\Sigma_{p_2\varepsilon})} \leq C\varepsilon^{-\frac{1}{2}}
\Big\{\|F\|_{L^2(\Omega)} + \|h\|_{H^{1}(\Omega)}+\|g\|_{H^1(\partial\Omega)}\Big\},
\end{equation}
where $p_1,p_2>0$ are fixed real number, and $C$ depends on $\mu,d,p_1,p_2$ and $\Omega$.
\end{lemma}

\begin{remark}
\emph{The results similar to $\eqref{pri:3.6}$ and $\eqref{pri:3.7}$ were originally investigated by Z. Shen
in \cite{SZW12} for linear elasticity systems. The author utilizes the radial maximal function to extend
his results to general elliptic systems in \cite{QX2}. We call
$\eqref{pri:3.6}$ the ``layer type'' estimate, while $\eqref{pri:3.7}$ is regarded as
the ``co-layer type'' one, where ``co-layer'' means the complementary layer for short.}
\end{remark}

\begin{proof}
We first address the estimate $\eqref{pri:3.6}$.
Let $\tilde{F}$ be the $0$-extension of $F$ to $\mathbb{R}^d$ such that $\tilde{F} = 0$
on $\mathbb{R}^d\setminus\Omega$, and $\tilde{h}$ is the $W^{1,q}$-extension of $h$ to $\mathbb{R}^d$ such
that $\tilde{h} = h$ a.e. in $\Omega$ and $\|\tilde{h}\|_{W^{1,q}(\mathbb{R}^d)}\leq C\|h\|_{W^{1,q}(\Omega)}$.
Then we consider $u_0 = v+w$ and $p_0=p_{0,1}+p_{0,2}$, and they satisfy
\begin{equation}\label{pde:3.4}
\text{(HP)}\left\{\begin{aligned}
\mathcal{L}_0(v) + \nabla p_{0,1} & = \tilde{F} &\quad& \text{in}~~\mathbb{R}^d,\\
\text{div}(v) &= \tilde{h} &\quad& \text{in}~~\mathbb{R}^d,
\end{aligned}\right.
\qquad \text{and} \qquad
\text{(BVP)}\left\{\begin{aligned}
\mathcal{L}_0(w) + \nabla p_{0,2} & = 0 &\quad& \text{in}~~\Omega,\\
\text{div}(w) &= 0 &\quad& \text{in}~~\Omega,\\
w &= g -v &\quad& \text{on}~\partial\Omega.
\end{aligned}\right.
\end{equation}
First, by the orthogonal transformation and dilation,
the equation (HP) will be the form of $\eqref{pde:3.3}$, and in view of $\eqref{pri:3.3}$ we arrive at
\begin{equation}\label{f:3.3}
\|\nabla^2 v\|_{L^q(\mathbb{R}^d)} + \|\nabla p_{0,1}\|_{L^q(\mathbb{R}^d)}
\leq C\Big\{\|\tilde{F}\|_{L^q(\mathbb{R}^d)}+\|\tilde{h}\|_{W^{1,q}(\mathbb{R}^d)}\Big\}
\leq C\Big\{\|F\|_{L^q(\Omega)}+\|h\|_{W^{1,q}(\Omega)}\Big\},
\end{equation}
where $C$ depends on $\mu,d$ and $\Omega$.
Due to the Sobolev's inequality, we also have
\begin{equation}\label{f:3.4}
\|\nabla v\|_{L^{q^\prime}(\mathbb{R}^d)} + \|p_{0,1}\|_{L^{q^\prime}(\mathbb{R}^d)}
\leq C\big\{\|\nabla^2 v\|_{L^{q}(\mathbb{R}^d)}+ \|\nabla p_{0,1}\|_{L^q(\mathbb{R}^d)}\big\}
\leq C\Big\{\|F\|_{L^q(\Omega)}+\|h\|_{W^{1,q}(\Omega)}\Big\},
\end{equation}
where $1/q^\prime =1/q-1/d$.
Set $\varrho=(\varrho_1,\cdots,\varrho_d)\in C^{1}_0(\mathbb{R}^d;\mathbb{R}^d)$ be a vector field such that
$[\varrho, n] \geq c>0$ on $\partial\Omega$ and $|\nabla\varrho|\leq Cr_0^{-1}$,
where $n$ is the outward unit normal vector to
$\partial\Omega$, and $[\cdot,\cdot]$ is inner product.
Meanwhile, it follows from the estimate $\eqref{pri:3.8}$ that
\begin{equation}\label{f:3.7}
\begin{aligned}
\|\nabla v\|_{L^q(\mathbb{R}^d)} + \|p_{0,1}\|_{L^q(\mathbb{R}^d)}
\leq C\Big\{\|\tilde{F}\|_{W^{-1,q}(\mathbb{R}^d)}+\|\tilde{h}\|_{L^{q}(\mathbb{R}^d)}\Big\}
\leq C\Big\{\|F\|_{L^{q}(\Omega)}+\|h\|_{L^{q}(\Omega)}\Big\}.
\end{aligned}
\end{equation}

From the divergence theorem, it follows that
\begin{equation}\label{f:3.5}
\begin{aligned}
c\int_{\partial\Omega} |\nabla v|^2 dS
&\leq \int_{\partial\Omega}[\varrho,n]|\nabla v|^2 dS
= \int_\Omega\text{div}(\varrho)|\nabla v|^2 dx + 2\int_\Omega \varrho_i[\nabla_i\nabla v,\nabla v] dx\\
&\leq C\bigg\{r_0^{-1}\int_\Omega|\nabla v|^2 dx + \int_\Omega|\nabla^2 v||\nabla v|dx\bigg\} \\
&\leq C\Big\{r_0^{-1}\big\|\nabla v\big\|_{L^2(\Omega)}^2 + \big\|\nabla^2 v\big\|_{L^{q}(\Omega)}
\big\|\nabla v\big\|_{L^{q^\prime}(\Omega)}\Big\}\\
&\leq C\Big\{\|F\|_{L^q(\Omega)}^2+\|h\|_{W^{1,q}(\Omega)}^2\Big\},
\end{aligned}
\end{equation}
where $1/q^\prime + 1/q = 1$ since $q=2d/(d+1)$,
and we use the estimates $\eqref{f:3.3}$ and $\eqref{f:3.4}$ in the last inequality.
Moreover, it is convenient to assume that $[\varrho,n]\geq c/2>0$ on $S_t$
for any $t\in[0,c_0]$, and we obtain
\begin{equation}\label{f:3.6}
\int_{S_t}|\nabla v|^2 dS \leq C\Big\{\|F\|_{L^q(\Omega)}^2+\|h\|_{W^{1,q}(\Omega)}^2\Big\},
\end{equation}
where $C$ depends on $c_0,c,d$, independent of $t$. Hence by the co-area formula $\eqref{eq:2.3}$ and
the estimate $\eqref{f:3.6}$, we reach
\begin{equation}\label{f:3.9}
\|\nabla v\|_{L^2(\Omega\setminus\Sigma_{p_1\varepsilon})}
=\bigg(\int_0^{p_1\varepsilon}\int_{\partial S_t}|\nabla v|^2 dS\bigg)^{\frac{1}{2}}
\leq C\varepsilon^{\frac{1}{2}}\Big\{\|F\|_{L^q(\Omega)}+\|h\|_{W^{1,q}(\Omega)}\Big\},
\end{equation}
by noting that $0<p_1\varepsilon<c_0$.

Now we turn to study the quantity $\|\nabla w\|_{L^2(\Omega\setminus\Sigma_{p_1\varepsilon})}$
by considering (BVP) in $\eqref{pde:3.4}$. It follows from \cite[Theorem 4.15]{EBFCEKGCV} that
\begin{equation}\label{f:3.8}
\begin{aligned}
\|(\nabla w)^*\|_{L^2(\partial\Omega)} + \|(p_{0,2})^*\|_{L^2(\partial\Omega)}
&\leq C\Big\{\|g\|_{H^1(\partial\Omega)}+\|v\|_{H^1(\partial\Omega)}\Big\} \\
&\leq C\Big\{\|g\|_{H^1(\partial\Omega)}+\|v\|_{L^2(\partial\Omega)}+\|\nabla v\|_{L^2(\partial\Omega)}\Big\} \\
&\leq C\Big\{\|g\|_{H^1(\partial\Omega)}+\|F\|_{L^q(\Omega)}+\|h\|_{W^{1,q}(\Omega)}\Big\}.
\end{aligned}
\end{equation}
In the last step above, we use $\eqref{f:3.5}$ and the following fact that
\begin{equation}\label{f:3.14}
\begin{aligned}
\|v\|_{L^2(\partial\Omega)}
\leq C\Big\{\|v\|_{L^2(\Omega)}+\|v\|_{L^{q^\prime}(\Omega)}^{\frac{1}{2}}
\|\nabla v\|_{L^q(\Omega)}^{\frac{1}{2}}\Big\}
\leq C\|\nabla v\|_{L^q(\mathbb{R}^d)}
\leq C\Big\{\|F\|_{L^{q}(\Omega)}+\|h\|_{L^{q}(\Omega)}\Big\},
\end{aligned}
\end{equation}
where the first inequality is based on the similar arguments employed by $\eqref{f:3.5}$, and
we use H\"older's inequality and Sobolev's inequality in the second step,
and the estimate $\eqref{f:3.7}$ in the last one.

Thus it follows from the estimates $\eqref{pri:2.5}$ and $\eqref{f:3.8}$ that
\begin{equation}\label{f:3.22}
\|\nabla w\|_{L^2(\Omega\setminus\Sigma_{p_1\varepsilon})}
\leq C\varepsilon^{\frac{1}{2}}\|(\nabla w)^*\|_{L^2(\partial\Omega)}
\leq C\varepsilon^{\frac{1}{2}}
\Big\{\|F\|_{L^q(\Omega)}+\|h\|_{W^{1,q}(\Omega)}+\|g\|_{H^1(\partial\Omega)}\Big\}.
\end{equation}
This together with $\eqref{f:3.9}$ implies
\begin{equation}\label{f:3.18}
\|\nabla u_0\|_{L^2(\Omega\setminus\Sigma_{p_1\varepsilon})}
\leq C\varepsilon^{\frac{1}{2}}
\Big\{\|F\|_{L^q(\Omega)}+\|h\|_{W^{1,q}(\Omega)}+\|g\|_{H^1(\partial\Omega)}\Big\}.
\end{equation}
Meanwhile we can employ the radial maximal function to handle the pressure term $p_0$ in
the estimate $\eqref{pri:3.6}$. It follows from the co-area formula $\eqref{eq:2.3}$
and the estimate $\eqref{pri:2.5}$ that
\begin{equation}\label{f:3.17}
\begin{aligned}
\|p_0\|_{L^2(\Omega\setminus\Sigma_{p_1\varepsilon})}
&\leq \|p_{0,1}\|_{L^2(\Omega\setminus\Sigma_{p_1\varepsilon})}
+\|p_{0,2}\|_{L^2(\Omega\setminus\Sigma_{p_1\varepsilon})}\\
&\leq \Big(\int_{0}^{p_1\varepsilon}\int_{S_t}|p_{0,1}|^2 dS_t dt\Big)^{\frac{1}{2}}
+C\varepsilon^{\frac{1}{2}}\|\mathcal{M}(p_{0,2})\|_{L^2(\partial\Omega)}\\
&\leq C\varepsilon^{\frac{1}{2}}\sup_{0\leq t\leq p_1\varepsilon}\|p_{0,1}\|_{L^2(S_t)}
+ C\varepsilon^{\frac{1}{2}}
\Big\{\|F\|_{L^q(\Omega)}+\|h\|_{W^{1,q}(\Omega)}+\|g\|_{H^1(\partial\Omega)}\Big\},
\end{aligned}
\end{equation}
where we use the fact that $\mathcal{M}(p_{0,2})(Q)\leq (p_{0,2})^*(Q)$ for every $Q\in\partial\Omega$,
and the estimate $\eqref{f:3.8}$ is employed in the last step. The remaining thing is to estimate
$\sup_{0\leq t\leq c_0}\|p_{0,1}\|_{L^2(S_t)}$. By a similar computation as $\eqref{f:3.5}$, we
have
\begin{equation*}
\begin{aligned}
\|p_{0,1}\|_{L^2(S_t)}
&\leq C\Big\{\|p_{0,1}\|_{L^2(\Sigma_t)}
+ \|\nabla p_{0,1}\|_{L^{q}(\Sigma_{t})}^{\frac{1}{2}}
\|p_{0,1}\|_{L^{q^\prime}(\Sigma_t)}^{\frac{1}{2}}\Big\}\\
&\leq C\|\nabla p_{0,1}\|_{L^q(\mathbb{R}^d)}
\leq C\Big\{\|F\|_{L^q(\Omega)}+\|h\|_{W^{1,q}(\Omega)}\Big\}
\end{aligned}
\end{equation*}
for any $t\in[0,p_1\varepsilon]$, where we use H\"older's inequality and Sobolev's inequality
in the second step,
and the estimate $\eqref{f:3.3}$ in the last one. This gives
\begin{equation}\label{f:3.16}
\sup_{0\leq t\leq p_1\varepsilon}\|p_{0,1}\|_{L^2(S_t)}
\leq  C\Big\{\|F\|_{L^q(\Omega)}+\|h\|_{W^{1,q}(\Omega)}\Big\}.
\end{equation}
Plugging $\eqref{f:3.16}$ into $\eqref{f:3.17}$ leads to
\begin{equation*}
\|p_0\|_{L^2(\Omega\setminus\Sigma_{p_1\varepsilon})}
\leq C\varepsilon^{\frac{1}{2}}
\Big\{\|F\|_{L^q(\Omega)}+\|h\|_{W^{1,q}(\Omega)}+\|g\|_{H^1(\partial\Omega)}\Big\},
\end{equation*}
and this coupled with $\eqref{f:3.18}$ proves the desired estimate $\eqref{pri:3.6}$.

We now proceed to prove the estimate $\eqref{pri:3.7}$. It is directly to see that
\begin{equation}\label{f:3.10}
\begin{aligned}
\|\nabla^2u_0\|_{L^2(\Sigma_{p_2\varepsilon})}
&\leq \|\nabla^2 v\|_{L^2(\Sigma_{p_2\varepsilon})} + \|\nabla^2 w\|_{L^2(\Sigma_{p_2\varepsilon})} \\
&\leq C\Big\{\|F\|_{L^2(\Omega)}+\|h\|_{H^1(\Omega)}\Big\}
+ \|\nabla^2 w\|_{L^2(\Sigma_{p_2\varepsilon})},
\end{aligned}
\end{equation}
where the second inequality is due to the estimate $\eqref{pri:3.3}$,
and remaining thing is to handle the last term in the second line of $\eqref{f:3.10}$.
Observing the equation (BVP) in $\eqref{pde:3.4}$, it is not hard to derive
the following interior estimate
\begin{equation}\label{f:3.12}
|\nabla^2 w(x)| \leq \frac{C}{\delta(x)}\Big(\dashint_{B(x,\delta(x)/8)}|\nabla w|^2 dy\Big)^{\frac{1}{2}},
\end{equation}
and this result in fact follows from the Sobolev embedding theorem and the $H^k$-regularity theory
(see \cite[Theorem 1.4]{MGMG}). Then we have
\begin{equation*}
|\nabla^2 w(x)|^2[\delta(x)]^{d+2}\leq C\int_\Omega|\nabla w|^2 dy.
\end{equation*}
Integrating by parts with respect to $x$ in $\Sigma_{c_0}$ (where $c_0$ is the layer constant), we have
\begin{equation*}
c_0^{d+2}\int_{\Sigma_{c_0}}|\nabla^2 w|^2dx\leq \int_{\Sigma_{c_0}}|\nabla^2 w(x)|^2[\delta(x)]^{d+2}dx
\leq C|\Omega|\int_\Omega|\nabla w|^2 dx.
\end{equation*}
Thus we derive the following interior estimate
\begin{equation}\label{f:3.27}
  \|\nabla^2 w\|_{L^2(\Sigma_{c_0})}
  \leq C(\mu,m,d,c_0,\Omega)\|\nabla w\|_{L^2(\Omega)},
\end{equation}
and this implies
\begin{equation}\label{f:3.11}
\int_{\Sigma_{p_2\varepsilon}} |\nabla^2 w|^2 dx
= \int_{\Sigma_{p_2\varepsilon}\setminus\Sigma_{c_0}}|\nabla^2 w|^2 dx
+\int_{\Sigma_{c_0}}|\nabla^2 w|^2 dx
\leq \int_{\Sigma_{p_2\varepsilon}\setminus\Sigma_{c_0}}|\nabla^2 w|^2 dx
+ C\|\nabla w\|_{L^2(\Omega)}^2.
\end{equation}
Clearly, we only need to estimate the first term in the right-hand side of $\eqref{f:3.11}$.
In view of the estimate $\eqref{f:3.12}$, we have
\begin{equation}\label{f:3.13}
\begin{aligned}
\int_{\Sigma_{p_2\varepsilon}\setminus\Sigma_{c_0}}|\nabla^2 w|^2 dx
&\leq C\int_{\Sigma_{p_2\varepsilon}\setminus\Sigma_{c_0}}\dashint_{B(x,\delta(x)/8)}
\frac{|\nabla w(y)|^2}{[\delta(x)]^2} dy dx \\
&\leq C\int_{p_2\varepsilon}^{c_0}\int_{S_t}\frac{|(\nabla w)^*(x^\prime)|^2}{t^2}dS_t(\Lambda_t(x^\prime))dt\\
&\leq C\int_{\partial\Omega}|(\nabla w)^*(x^\prime)|^2dS(x^\prime)\int_{p_2\varepsilon}^\infty\frac{dt}{t^2}\\
&\leq C\varepsilon^{-1}\|(\nabla w)^*\|_{L^2(\partial\Omega)}^2,
\end{aligned}
\end{equation}
where $x^\prime\in\partial\Omega$
such that $\delta(x) = |x-x^\prime|$,
and in the second step, we use the co-area formula $\eqref{eq:2.3}$ and the fact that $|\nabla w(y)|\leq |(\nabla w)^*(x^\prime)|$
for all $y\in B(x,\delta(x)/8)$. Also, we have
\begin{equation}\label{f:3.15}
\begin{aligned}
\big\|\nabla w\big\|_{L^2(\Omega)}
\leq  C\|g-v\|_{H^{\frac{1}{2}}(\partial\Omega)}
&\leq  C\Big\{\|g\|_{H^{1}(\partial\Omega)}+\|v\|_{H^{1}(\partial\Omega)}\Big\} \\
&\leq C\Big\{\|F\|_{L^q(\Omega)}+\|h\|_{W^{1,q}(\Omega)}
+\|g\|_{H^1(\partial\Omega)}\Big\},
\end{aligned}
\end{equation}
where we use the estimate $\eqref{pri:2.6}$ in the first inequality, and the estimates
$\eqref{f:3.5}$ and $\eqref{f:3.14}$ in the last one.

By inserting $\eqref{f:3.13}$ and $\eqref{f:3.15}$ into $\eqref{f:3.11}$ we derive
\begin{equation}\label{f:3.20}
\begin{aligned}
\|\nabla^2 w\|_{L^2(\Sigma_{p_2\varepsilon})}
&\leq C\varepsilon^{-\frac{1}{2}}\big\|(\nabla w)^*\big\|_{L^2(\partial\Omega)}
+ C\big\|\nabla w\big\|_{L^2(\Omega)} \\
&\leq C\varepsilon^{-\frac{1}{2}}\Big\{\|F\|_{L^q(\Omega)}+\|h\|_{W^{1,q}(\Omega)}
+\|g\|_{H^1(\partial\Omega)}\Big\}.
\end{aligned}
\end{equation}
This together with $\eqref{f:3.10}$ gives the estimate $\eqref{pri:3.7}$, and we have completed the proof.
\end{proof}

\begin{thm}\label{thm:3.1}
Suppose that $A$ satisfies $\eqref{a:1}-\eqref{a:3}$.
Given $F\in L^2(\Omega;\mathbb{R}^d)$, $h\in H^1(\Omega)$ and
$g\in H^1(\partial\Omega;\mathbb{R}^d)$ with the compatibility condition
$\eqref{a:4}$,
assume that $(u_\varepsilon,p_\varepsilon)$, $(u_0,p_0)$ in
$H^1(\Omega;\mathbb{R}^d)\times L^2(\Omega)$ are two weak solutions to $(\textbf{DS}_\varepsilon)$
and $(\textbf{DS}_0)$, respectively. We have two results:
\begin{itemize}
  \item [\emph{(i)}] set $\varphi_k^\gamma = S_\varepsilon(\psi_{2\varepsilon}\nabla_ku_0^\gamma)$ in $\eqref{eq:3.4}$,
  then we have
\begin{equation}\label{pri:3.12}
\|w_\varepsilon\|_{H_0^1(\Omega)} + \|z_\varepsilon\|_{L^2(\Omega)/\mathbb{R}}
\leq C\varepsilon^{\frac{1}{2}}\Big\{\|F\|_{L^2(\Omega)}+\|h\|_{H^1(\Omega)}
+\|g\|_{H^1(\partial\Omega)}\Big\};
\end{equation}
  \item [\emph{(ii)}] let $\varphi_k^\gamma = S_\varepsilon^2(\psi_{2\varepsilon}\nabla_ku_0^\gamma)$
  in $\eqref{eq:3.4}$, then it admits
\begin{equation}\label{pri:3.9}
\|w_\varepsilon\|_{H_0^1(\Omega)} + \|z_\varepsilon\|_{L^2(\Omega)/\mathbb{R}}
\leq C\varepsilon^{\frac{1}{2}}\Big\{\|F\|_{L^q(\Omega)}+\|h\|_{W^{1,q}(\Omega)}
+\|g\|_{H^1(\partial\Omega)}\Big\}
\end{equation}
\end{itemize}
with $q=2d/(d+1)$,
where $C$ depends on $\mu,d,$ and $\Omega$.
\end{thm}

\begin{proof}
(i) It follows from the estimate $\eqref{pri:3.1}$ that
\begin{equation*}
\begin{aligned}
\|w_\varepsilon\|_{H_0^1(\Omega)} + \|z_\varepsilon\|_{L^2(\Omega)/\mathbb{R}}
&\leq C\Big\{\varepsilon\big\|\varpi(\cdot/\varepsilon)
\nabla S_\varepsilon(\psi_{2\varepsilon}\nabla u_0)\big\|_{L^2(\Omega)}
+ \big\|\nabla u_0 - S_\varepsilon(\psi_{2\varepsilon}\nabla u_0)\big\|_{L^2(\Omega)}\Big\} \\
&\leq C\varepsilon\big\|\nabla(\psi_{2\varepsilon}\nabla u_0)\big\|_{L^2(\mathbb{R}^d)}
+\big\|(1-\psi_{2\varepsilon})\nabla u_0\big\|_{L^2(\Omega)}
+\big\|\psi_{2\varepsilon}\nabla u_0 - S_\varepsilon(\psi_{2\varepsilon}\nabla u_0)\big\|_{L^2(\Omega)}\\
&\leq  C\varepsilon\big\|\nabla(\psi_{2\varepsilon}\nabla u_0)\big\|_{L^2(\mathbb{R}^d)}
+ \|\nabla u_0\|_{L^2(\Omega\setminus\Sigma_{4\varepsilon})}\\
&\leq C\|\nabla u_0\|_{L^2(\Omega\setminus\Sigma_{4\varepsilon})}
+ C\varepsilon\|\nabla^2 u_0\|_{L^2(\Sigma_{2\varepsilon})}\\
&\leq C\varepsilon^{\frac{1}{2}}\Big\{\|F\|_{L^2(\Omega)}+\|h\|_{H^1(\Omega)}
+\|g\|_{H^1(\partial\Omega)}\Big\},
\end{aligned}
\end{equation*}
where we employ the estimate $\eqref{pri:2.7}$ in the second step, and the estimate $\eqref{pri:2.8}$ in
the third one. In the last inequality above, we use the estimates $\eqref{pri:3.6}$ and $\eqref{pri:3.7}$.

(ii) Also, it follows from the estimate $\eqref{pri:3.1}$ that
\begin{equation}\label{f:3.19}
\|w_\varepsilon\|_{H_0^1(\Omega)} + \|z_\varepsilon\|_{L^2(\Omega)/\mathbb{R}}
\leq C\Big\{\varepsilon\big\|\varpi(\cdot/\varepsilon)
\nabla S_\varepsilon^2(\psi_{2\varepsilon}\nabla u_0)\big\|_{L^2(\Omega)}
+ \big\|\nabla u_0 - S_\varepsilon^2(\psi_{2\varepsilon}\nabla u_0)\big\|_{L^2(\Omega)}\Big\}.
\end{equation}
We now handle the first term in the right-hand side of $\eqref{f:3.19}$ as follows:
\begin{equation*}
\big\|\varpi(\cdot/\varepsilon)
\nabla S_\varepsilon^2(\psi_{2\varepsilon}\nabla u_0)\big\|_{L^2(\Omega)}
\leq C\|\nabla S_\varepsilon(\psi_{2\varepsilon}\nabla u_0)\|_{L^2(\mathbb{R}^d)}
\leq C\varepsilon^{-1}\|\nabla u_0\|_{L^2(\Omega\setminus\Sigma_{4\varepsilon})}
+ C\|S_\varepsilon(\psi_{2\varepsilon}\nabla^2 u_0)\|_{L^2(\Omega)},
\end{equation*}
where we employ the estimate $\eqref{pri:2.7}$ above.
Then we focus on studying the term $\|S_\varepsilon(\psi_{2\varepsilon}\nabla^2 u_0)\|_{L^2(\Omega)}$.
By noting that $u_0 = v + w$ and $v,w$ satisfy (HP) and (BVP) in $\eqref{pde:3.4}$, respectively,
it is controlled by
\begin{equation*}
\begin{aligned}
\|S_\varepsilon(\psi_{2\varepsilon}\nabla^2 v)\|_{L^2(\Omega)}
+ \|S_\varepsilon(\psi_{2\varepsilon}\nabla^2 w)\|_{L^2(\Omega)}
&\leq C\varepsilon^{-\frac{1}{2}}\|\psi_{2\varepsilon}\nabla^2v\|_{L^q(\mathbb{R}^d)}
+C\|\nabla^2 w\|_{L^2(\Sigma_{2\varepsilon})} \\
&\leq C\varepsilon^{-\frac{1}{2}}\Big\{
\|F\|_{L^q(\Omega)}+\|h\|_{W^{1,q}(\Omega)}
+\|g\|_{H^1(\partial\Omega)}\Big\},
\end{aligned}
\end{equation*}
where we use the estimates $\eqref{f:3.3}$ and $\eqref{f:3.20}$ in the second inequality.
Thus we have
\begin{equation}\label{f:3.23}
\begin{aligned}
\big\|\varpi(\cdot/\varepsilon)
\nabla S_\varepsilon^2(\psi_{2\varepsilon}\nabla u_0)\big\|_{L^2(\Omega)}
&\leq C\varepsilon^{-1}\|\nabla u_0\|_{L^2(\Omega\setminus\Sigma_{4\varepsilon})}
+ C\varepsilon^{-\frac{1}{2}}\Big\{
\|F\|_{L^q(\Omega)}+\|h\|_{W^{1,q}(\Omega)}
+\|g\|_{H^1(\partial\Omega)}\Big\}\\
&\leq C\varepsilon^{-\frac{1}{2}}\Big\{
\|F\|_{L^q(\Omega)}+\|h\|_{W^{1,q}(\Omega)}
+\|g\|_{H^1(\partial\Omega)}\Big\},
\end{aligned}
\end{equation}
where we use the estimate $\eqref{pri:3.6}$ in the last step.

We proceed to address the second term in the right-hand side of $\eqref{f:3.19}$. It is not hard to derive
\begin{equation}\label{f:3.21}
\begin{aligned}
\big\|\nabla u_0 - S_\varepsilon^2(\psi_{2\varepsilon}\nabla u_0)\big\|_{L^2(\Omega)}
&\leq \big\|(1-\psi_{2\varepsilon})\nabla u_0\big\|_{L^2(\Omega)}
 + \big\|\psi_{2\varepsilon}\nabla u_0 - S_\varepsilon(\psi_{2\varepsilon}\nabla u_0)\big\|_{L^2(\Omega)}\\
& + \big\|S_\varepsilon\big(\psi_{2\varepsilon}\nabla u_0
 - S_\varepsilon(\psi_{2\varepsilon}\nabla u_0)\big)\big\|_{L^2(\Omega)} \\
&\leq \|\nabla u_0\|_{L^2(\Omega\setminus\Sigma_{4\varepsilon})}
+ \big\|\psi_{2\varepsilon}\nabla u_0 - S_\varepsilon(\psi_{2\varepsilon}\nabla u_0)\big\|_{L^2(\Omega)},
 \end{aligned}
\end{equation}
where we use the estimate $\eqref{pri:2.7}$ in the last step.
Then we turn to estimate the second term in the third line of $\eqref{f:3.21}$.
By noting that $u_0 = v+ w$ and $\psi_{2\varepsilon}\nabla u_0 - S_\varepsilon(\psi_{2\varepsilon}\nabla u_0)$
is supported in $\Sigma_\varepsilon$, it can be separated into
\begin{equation*}
\begin{aligned}
& \|\nabla v-S_\varepsilon(\nabla v)\|_{L^2(\Sigma_\varepsilon)}
+\|(\psi_{2\varepsilon}-1)\nabla v\|_{L^2(\Sigma_\varepsilon)}
+\|S_\varepsilon\big((\psi_{2\varepsilon}-1)\nabla v\big)\|_{L^2(\Sigma_\varepsilon)}
+\|\psi_{2\varepsilon}\nabla w-S_\varepsilon(\psi_{2\varepsilon}\nabla w)\|_{L^2(\Sigma_\varepsilon)} \\
& \leq \|\nabla v-S_\varepsilon(\nabla v)\|_{L^2(\Sigma_\varepsilon)}
+ 2\|(\psi_{2\varepsilon}-1)\nabla v\|_{L^2(\Sigma_\varepsilon)}
+ \|\psi_{2\varepsilon}\nabla w-S_\varepsilon(\psi_{2\varepsilon}\nabla w)\|_{L^2(\Sigma_\varepsilon)}
\end{aligned}
\end{equation*}
where we use the fact that
$\|S_{\varepsilon}((\psi_{2\varepsilon}-1)\nabla v)\|_{L^2(\Sigma_\varepsilon)}
\leq \|(\psi_{2\varepsilon}-1)\nabla v\|_{L^2(\Omega)}$.
Applying the estimates $\eqref{pri:2.9}$ and $\eqref{pri:2.8}$ to the second line above, we then have
\begin{equation}\label{f:3.24}
\begin{aligned}
\big\|\nabla u_0 - S_\varepsilon^2(\psi_{2\varepsilon}\nabla u_0)\big\|_{L^2(\Omega)}
&\leq C\varepsilon^{\frac{1}{2}}\|\nabla^2 v\|_{L^q(\mathbb{R}^d)}
+ C\|\nabla v\|_{L^2(\Omega\setminus\Sigma_{4\varepsilon})}
+ C\varepsilon\|\nabla(\nabla w\psi_{2\varepsilon})\|_{L^2(\mathbb{R}^d)}\\
&\leq C\Big\{\varepsilon^{\frac{1}{2}}\|\nabla^2 v\|_{L^q(\mathbb{R}^d)}
+\varepsilon\|\nabla^2 w\|_{L^2(\Sigma_{2\varepsilon})}
+\|\nabla v\|_{L^2(\Omega\setminus\Sigma_{2\varepsilon})}
+\|\nabla w\|_{L^2(\Omega\setminus\Sigma_{2\varepsilon})}\Big\}\\
&\leq C\varepsilon^{\frac{1}{2}}\Big\{\|F\|_{L^q(\Omega)}+\|h\|_{W^{1,q}(\Omega)}
+\|g\|_{H^1(\partial\Omega)}\Big\},
\end{aligned}
\end{equation}
where we use the estimates $\eqref{f:3.3}$, $\eqref{f:3.20}$, $\eqref{f:3.9}$ and $\eqref{f:3.22}$ in the
last step. Collecting the estimates $\eqref{f:3.19}$, $\eqref{f:3.23}$ and $\eqref{f:3.24}$ gives the desired
result $\eqref{pri:3.9}$, and the proof is all completed.
\end{proof}

\begin{cor}\label{cor:3.1}
Assume the same conditions as that in Theorem $\ref{thm:3.1}$.
Set $w_\varepsilon,z_\varepsilon$ in $\eqref{eq:3.4}$ by choosing
$\varphi_k^\gamma = S_\varepsilon(\psi_{2\varepsilon}\nabla_ku_0^\gamma)$, and we then have
\begin{equation}\label{pri:3.10}
\|u_\varepsilon-u_0\|_{L^{2^*}(\Omega)}+\|p_\varepsilon-p_0-\pi_k^\gamma(\cdot/\varepsilon)
S_\varepsilon(\psi_{2\varepsilon}\nabla_ku_0^\gamma)\|_{L^2(\Omega)/\mathbb{R}}
\leq C\varepsilon^{\frac{1}{2}}\Big\{\|F\|_{L^2(\Omega)}+\|h\|_{H^1(\Omega)}
+\|g\|_{H^1(\partial\Omega)}\Big\},
\end{equation}
where $2^* = 2d/(d-2)$ when $d\geq 3$ and $2^*\in[1,\infty)$ when $d=2$.
Moreover, if we replace $S_\varepsilon(\psi_{2\varepsilon}\nabla_ku_0^\gamma)$ above
by $S_\varepsilon^2(\psi_{2\varepsilon}\nabla_ku_0^\gamma)$, then we obtain
\begin{equation}\label{pri:3.11}
\|p_\varepsilon-p_0-\pi_k^\gamma(\cdot/\varepsilon)
S_\varepsilon^2(\psi_{2\varepsilon}\nabla_ku_0^\gamma)\|_{L^2(\Omega)/\mathbb{R}}
\leq C\varepsilon^{\frac{1}{2}}\Big\{\|F\|_{L^q(\Omega)}+\|h\|_{W^{1,q}(\Omega)}
+\|g\|_{H^1(\partial\Omega)}\Big\}
\end{equation}
with $q=2d/(d+1)$, where $C$ depends on $\mu,d$ and $\Omega$.
\end{cor}

\begin{proof}
We first verify the estimate $\eqref{pri:3.10}$.
Let $2^* =2d/(d-2)$ in the case of $d\geq 3$ and $2^*\in[1,\infty)$ when $d=2$.
Due to the Sobolev imbedding theorem we have
\begin{equation}\label{f:3.28}
\begin{aligned}
\|u_\varepsilon - u_0\|_{L^{2^*}(\Omega)}
&\leq \|w_\varepsilon\|_{L^{2^*}(\Omega)} +
\varepsilon\|\chi_k(\cdot/\varepsilon)S_\varepsilon(\psi_{2\varepsilon}\nabla_k u_0)\|_{L^{2^*}(\Omega)} \\
&\leq C\|w_\varepsilon\|_{H^1_0(\Omega)}
+\varepsilon\|\nabla(\psi_{2\varepsilon}\nabla u_0)\|_{L^{2}(\mathbb{R}^d)}
\leq C\varepsilon^{\frac{1}{2}}\Big\{\|F\|_{L^2(\Omega)}+\|h\|_{H^1(\Omega)}
+\|g\|_{H^1(\partial\Omega)}\Big\},
\end{aligned}
\end{equation}
where we still employ the estimate $\eqref{pri:2.7}$ in the second inequality,
and the estimates $\eqref{pri:3.12}$, $\eqref{pri:3.6}$ and $\eqref{pri:3.7}$ in the last one.
Then we handle the pressure term. It follows from $\eqref{pri:3.12}$ that
\begin{equation}\label{f:3.25}
\|p_\varepsilon-p_0-\pi_k^\gamma(\cdot/\varepsilon)
S_\varepsilon(\psi_{2\varepsilon}\nabla_ku_0^\gamma)\|_{L^2(\Omega)/\mathbb{R}}
\leq \|z_\varepsilon\|_{L^2(\Omega)/\mathbb{R}}
+ \varepsilon\|q_{ik}^\gamma(\cdot/\varepsilon)\nabla_i S_\varepsilon(\psi_{2\varepsilon}
\nabla_k u_0^\gamma)\|_{L^2(\Omega)/\mathbb{R}},
\end{equation}
and it is clear to see that we only need to handle the last term in
the right-hand side of $\eqref{f:3.25}$. We have
\begin{equation}\label{f:3.26}
\begin{aligned}
\|q_{ik}^\gamma(\cdot/\varepsilon)\nabla_i S_\varepsilon(\psi_{2\varepsilon}
\nabla_k u_0^\gamma)\|_{L^2(\Omega)/\mathbb{R}}
& \leq \|q_{ik}^\gamma(\cdot/\varepsilon)\nabla_i S_\varepsilon(\psi_{2\varepsilon}
\nabla_k u_0^\gamma)\|_{L^2(\Omega)} \\
&\leq C\|\nabla(\psi_{2\varepsilon}\nabla u_0)\|_{L^2(\mathbb{R}^d)}
\leq C\Big\{\varepsilon^{-1}\|\nabla u_0\|_{L^2(\Omega\setminus\Sigma_{4\varepsilon})}
+\|\nabla^2u_0\|_{L^2(\Sigma_{2\varepsilon})}\Big\}\\
&\leq C\varepsilon^{-\frac{1}{2}}\Big\{
\|F\|_{L^2(\Omega)}+\|h\|_{H^1(\Omega)}
+\|g\|_{H^1(\partial\Omega)}\Big\},
\end{aligned}
\end{equation}
where we use the estimate $\eqref{pri:2.7}$ in the second inequality,
and the estimates $\eqref{pri:3.6}$ and $\eqref{pri:3.7}$ in the last one. Then
collecting the estimates $\eqref{f:3.25}$, $\eqref{f:3.26}$, $\eqref{f:3.28}$ and $\eqref{pri:3.12}$ leads to
the desired estimate $\eqref{pri:3.10}$.

To obtain the estimate $\eqref{pri:3.11}$, we only need to compute the term
$\|q_{ik}^\gamma(\cdot/\varepsilon)\nabla_i S_\varepsilon^2(\psi_{2\varepsilon}
\nabla_k u_0^\gamma)\|_{L^2(\Omega)}$ according to the procedure above.
In fact the computation is as the same as we did in $\eqref{f:3.23}$.
So we provide the result without details, i.e.,
\begin{equation*}
\big\|q_{ik}^\gamma(\cdot/\varepsilon)\nabla_i S_\varepsilon^2(\psi_{2\varepsilon}
\nabla_k u_0^\gamma)\big\|_{L^2(\Omega)}\leq
C\varepsilon^{-\frac{1}{2}}\Big\{
\|F\|_{L^q(\Omega)}+\|h\|_{W^{1,q}(\Omega)}
+\|g\|_{H^1(\partial\Omega)}\Big\}
\end{equation*}
with $q=2d/(d+1)$. This together with $\eqref{pri:3.9}$ implies the estimate $\eqref{pri:3.11}$ and
we are done.
\end{proof}

\section{$O(\varepsilon\ln(r_0/\varepsilon))$ convergence rate in $L^2(\Omega)$}

In the section, the main idea is the so-called duality method. So we need to consider the adjoint
problems: for any $\Phi\in L^2(\Omega;\mathbb{R}^d)$,
there exist $(\phi_\varepsilon,\theta_\varepsilon),(\phi_0,\theta_0)
\in H_0^1(\Omega;\mathbb{R}^d)\times L^2(\Omega)/\mathbb{R}$ respectively solving
\begin{equation}\label{pde:4.1}
(\mathbf{DS_\varepsilon})^*\left\{
\begin{aligned}
\mathcal{L}_\varepsilon^*(\phi_\varepsilon) + \nabla \theta_\varepsilon &= \Phi &\quad &\text{in}~~\Omega, \\
 \text{div} (\phi_\varepsilon) &= 0 &\quad&\text{in} ~~\Omega,\\
 \phi_\varepsilon &= 0 &\quad&\text{on} ~\partial\Omega,
\end{aligned}\right.
\qquad\text{and}\qquad
(\mathbf{DS_0})^*\left\{
\begin{aligned}
\mathcal{L}_0^*(\phi_0) + \nabla \theta_0 &= \Phi &\quad &\text{in}~~\Omega, \\
 \text{div} (\phi_0) &= 0 &\quad&\text{in} ~~\Omega,\\
 \phi_0 &= 0 &\quad&\text{on} ~\partial\Omega,
\end{aligned}\right.
\end{equation}
where $\mathcal{L}_\varepsilon^*, \mathcal{L}_0^*$ are the adjoint operators associated with
$\mathcal{L}_\varepsilon$ and $\mathcal{L}_0$, respectively, and given by
\begin{equation*}
\mathcal{L}_\varepsilon^* = -\text{div}(A^*(\cdot/\varepsilon)\nabla)
=-\frac{\partial}{\partial x_i}\Big[a_{ji}^{\beta\alpha}
\left(\frac{x}{\varepsilon}\right)\frac{\partial}{\partial x_j}\Big]
\qquad\text{and}\qquad
\mathcal{L}_0^* = -\text{div}(\widehat{A}^*\nabla)
=-\frac{\partial}{\partial x_i}\Big[\hat{a}_{ji}^{\beta\alpha}\frac{\partial}{\partial x_j}\Big].
\end{equation*}

\begin{lemma}[Duality lemma I]\label{lemma:4.1}
Let $(w_\varepsilon,z_\varepsilon)$ be given in $\eqref{eq:3.4}$ by choosing
$\varphi_k^\gamma= S_\varepsilon(\psi_{4\varepsilon}\nabla_k u_0^\gamma)$,
where the weak solutions $(u_\varepsilon,p_\varepsilon)$ and $(u_0,p_0)$ in
$H^1(\Omega;\mathbb{R}^d)\times L^2(\Omega)/\mathbb{R}$ satisfy
$(\mathbf{DS_\varepsilon})$ and $(\mathbf{DS_0})$, respectively.
For any $\Phi\in L^2(\Omega;\mathbb{R}^d)$, we assume that $(\phi_\varepsilon,\theta_\varepsilon)$ and
$(\phi_0,\theta_0)$ in $H_0^1(\Omega;\mathbb{R}^d)\times L^2(\Omega)$ are the solutions
to the related adjoint problems $(\mathbf{DS_\varepsilon})^*$ and $(\mathbf{DS_0})^*$,
respectively. Then we have
\begin{equation}\label{eq:4.1}
\int_\Omega w_\varepsilon\Phi dx
= -\int_\Omega \tilde{f}\cdot\nabla\phi_\varepsilon dx
-\int_\Omega(\theta_\varepsilon-\sigma)\eta dx
\end{equation}
for every $\sigma\in\mathbb{R}$, where
$\tilde{f}$ and $\eta$ are given in Lemma $\ref{lemma:3.2}$. Moreover, if we assume
\begin{equation}\label{eq:4.2}
\breve{w}_\varepsilon = \phi_\varepsilon - \phi_0
- \varepsilon\chi_k^*(\cdot/\varepsilon)S_\varepsilon(\psi_{10\varepsilon}\nabla_k\phi_0)
\qquad
\breve{z}_\varepsilon = \theta_\varepsilon - \theta_0
- \pi_k^*(\cdot/\varepsilon)S_\varepsilon(\psi_{10\varepsilon}\nabla_k\phi_0)
\end{equation}
where $(\chi_k^*,\pi_k^*)$ with $k=1,\cdots,d$ is the corrector associated with the adjoint problem
$(\mathbf{DS_\varepsilon})^*$. Then we obtain the estimate
\begin{equation}\label{pri:4.1}
\begin{aligned}
\bigg|\int_\Omega w_\varepsilon\Phi dx\bigg|
&\leq C\Big\{\big\|\nabla u_0\big\|_{L^2(\Omega\setminus\Sigma_{8\varepsilon};\delta)}
+\varepsilon\big\|\nabla^2u_0\big\|_{L^2(\Sigma_{3\varepsilon};\delta)}\Big\}
\cdot\Big\{\big\|\nabla\phi_0\big\|_{L^2(\Sigma_{3\varepsilon};\delta^{-1})}
+\big\|\theta_0\big\|_{L^2(\Sigma_{3\varepsilon};\delta^{-1})}\Big\}\\
& \qquad + C\Big\{\big\|\nabla u_0\big\|_{L^2(\Omega\setminus\Sigma_{8\varepsilon})}
+\varepsilon\|\nabla^2 u_0\|_{L^2(\Sigma_{4\varepsilon})}\Big\}\\
&\qquad \qquad\cdot\Big\{\|\nabla\breve{w}_\varepsilon\|_{L^2(\Omega)}
+\|\breve{z}_\varepsilon\|_{L^2(\Omega)/\mathbb{R}}
+\|\nabla\phi_0\|_{L^2(\Omega\setminus\Sigma_{20\varepsilon})}
+\|\theta_0\|_{L^2(\Omega\setminus\Sigma_{9\varepsilon})}
+ \varepsilon\|\nabla^2\phi_0\|_{L^2(\Sigma_{10\varepsilon})}\Big\}.
\end{aligned}
\end{equation}
\end{lemma}

\begin{proof}
We first prove the equality $\eqref{eq:4.1}$.
By noting that both of $w_\varepsilon$ and $\phi_\varepsilon$ vanish near $\partial\Omega$,
it is not hard from integrating by parts to see
$\big<\mathcal{L}_\varepsilon(w_\varepsilon),\phi_\varepsilon\big>
= \big<w_\varepsilon,\mathcal{L}_\varepsilon^*(\phi_\varepsilon)\big>$.
Thus in view of $\eqref{pde:3.2}$ and $(\mathbf{DS_\varepsilon})^*$ we have
\begin{equation}\label{f:4.9}
\begin{aligned}
\int_\Omega w_\varepsilon\Phi dx
&= \big<\mathcal{L}_\varepsilon(w_\varepsilon),\phi_\varepsilon\big>
+ \int_\Omega w_\varepsilon\nabla\theta_\varepsilon dx\\
&= \big<\text{div}(\tilde{f}) + \nabla z_\varepsilon, \phi_\varepsilon\big>
-\int_\Omega(\theta_\varepsilon-\sigma)\text{div}(w_\varepsilon)dx
=-\int_\Omega\tilde{f}\cdot\nabla\phi_\varepsilon
-\int_\Omega(\theta_\varepsilon-\sigma)\eta dx,
\end{aligned}
\end{equation}
where $\sigma\in \mathbb{R}$ is arbitrary, and
in the last step we use the fact that $\text{div}(\phi_\varepsilon) = 0$ in $\Omega$.

In view of $\eqref{f:4.9}$, we have
\begin{equation}\label{f:4.10}
\begin{aligned}
\bigg|\int_\Omega w_\varepsilon\Phi dx\bigg|
\leq \bigg|\int_\Omega \tilde{f}\cdot\nabla\phi_\varepsilon dx\bigg|
+ \bigg|\int_\Omega (\theta_\varepsilon-\sigma)\eta dx\bigg| =: I_1 + I_2.
\end{aligned}
\end{equation}
Then we calculate $I_1$ and $I_2$ one by one, and the first term is
\begin{equation}\label{f:4.11}
\begin{aligned}
I_1 \leq \varepsilon\bigg|\int_\Omega[E_{jik}^{\alpha\gamma}(y)
&+a_{ij}^{\alpha\beta}(y)\chi_k^{\beta\gamma}(y)]
\nabla_jS_\varepsilon(\psi_{4\varepsilon}\nabla_ku_0^\gamma)\nabla_i\phi_\varepsilon^\alpha dx\bigg|
+ \varepsilon\bigg|\int_\Omega q_{ik}^\gamma(y)\nabla_\alpha
S_\varepsilon(\psi_{4\varepsilon}\nabla_ku_0^\gamma)\nabla_i\phi_\varepsilon^\alpha dx\bigg|\\
&+ \bigg|\int_\Omega\big[\hat{a}_{ij}^{\alpha\beta}-a_{ij}^{\alpha\beta}(y)\big]
\Big[\nabla_j u_0^\beta-S_\varepsilon(\psi_{4\varepsilon}\nabla_ju_0^\beta)\Big]
\nabla_i\phi_\varepsilon^\alpha dx\bigg| =: I_{11} + I_{12}+ I_{13},
\end{aligned}
\end{equation}
where $y=x/\varepsilon$.
We proceed to estimate the term $I_{11}$. By setting
$\varpi_{jik}^{\alpha\gamma}(y)= E_{jik}^{\alpha\gamma}(y)
+a_{ij}^{\alpha\beta}(y)\chi_k^{\beta\gamma}(y)$, and $\varpi(y) = \big[\varpi_{jik}^{\alpha\gamma}(y)]$,
we have
\begin{equation}\label{f:4.12}
\begin{aligned}
I_{11} &= \varepsilon\bigg|\int_\Omega\varpi_{jik}^{\alpha\gamma}(\cdot/\varepsilon)
\nabla_jS_\varepsilon(\psi_{4\varepsilon}\nabla_ku_0^\gamma)\nabla_i\phi_\varepsilon^\alpha dx\bigg|\\
&\leq \int_\Omega\Big|\varpi_{jik}^{\alpha\gamma}(\cdot/\varepsilon)
S_\varepsilon(\nabla_j\psi_{4\varepsilon}\nabla_ku_0^\gamma)\nabla_i\phi_\varepsilon^\alpha \Big|dx
+\varepsilon\bigg|\int_\Omega\varpi_{jik}^{\alpha\gamma}(\cdot/\varepsilon)
S_\varepsilon(\psi_{4\varepsilon}\nabla^2_{jk}u_0^\gamma)\nabla_i\phi_\varepsilon^\alpha dx\bigg|
=: J_1+J_2.
\end{aligned}
\end{equation}
We mention that $S_\varepsilon(\nabla\psi_{4\varepsilon}\nabla u_0)$ is supported in
$\Omega\setminus\Sigma_{9\varepsilon}$ while $S_\varepsilon(\psi_{4\varepsilon}\nabla^2u_0)$
is supported in $\Sigma_{3\varepsilon}$, and then
\begin{equation}\label{f:4.13}
\begin{aligned}
J_1
&\leq \big\|\varpi(\cdot/\varepsilon)S_\varepsilon(\nabla\psi_{4\varepsilon}\nabla u_0)
\big\|_{L^2(\mathbb{R}^d)}\big\|\nabla\phi_\varepsilon\big\|_{L^2(\Omega\setminus\Sigma_{9\varepsilon})}\\
&\leq C\big\|\nabla u_0\big\|_{L^2(\Omega\setminus\Sigma_{8\varepsilon})}
\big\|\nabla\phi_\varepsilon\big\|_{L^2(\Omega\setminus\Sigma_{9\varepsilon})}
\leq C\big\|\nabla u_0\big\|_{L^2(\Omega\setminus\Sigma_{8\varepsilon})}
\Big\{\big\|\nabla \breve{w}_\varepsilon\big\|_{L^2(\Omega)}
+\big\|\nabla\phi_0\big\|_{L^2(\Omega\setminus\Sigma_{9\varepsilon})}\Big\},
\end{aligned}
\end{equation}
where we use Cauchy's inequality in the first inequality,
and the estimate $\eqref{pri:2.7}$ in the second one. In the last step above, we note that
it follows from $\eqref{eq:4.2}$ that $\phi_\varepsilon
= \breve{w}_\varepsilon + \phi_0$ in $\Omega\setminus\Sigma_{9\varepsilon}$ since the term
$\varepsilon\chi_k^*(\cdot/\varepsilon)S_\varepsilon(\psi_{10\varepsilon}\nabla_k\phi_0)$ in
$\eqref{eq:4.2}$ is supported in
$\Sigma_{9\varepsilon}$.
Then we handle the term $J_2$ by a similar argument
but a little more complicated to accelerate the convergence rate. One may have
\begin{equation*}
\begin{aligned}
J_2 &= \varepsilon\bigg|\int_\Omega\varpi_{jik}^{\alpha\gamma}(\cdot/\varepsilon)
S_\varepsilon(\psi_{4\varepsilon}\nabla^2_{jk}u_0^\gamma)\nabla_i
\big[\breve{w}_\varepsilon^\alpha+\phi_0^\alpha
+\varepsilon\chi_{k}^{\gamma\alpha}(y)S_\varepsilon(\psi_{10\varepsilon}\nabla_k\phi_0^\gamma)\big] dx\bigg|\\
&\leq \varepsilon\|\varpi(\cdot/\varepsilon)S_\varepsilon(\psi_{4\varepsilon}\nabla^2 u_0)\|_{L^2(\mathbb{R}^d)}
\Big\{\|\nabla\breve{w}_\varepsilon\|_{L^2(\Omega)}
+\varepsilon\|\chi(\cdot/\varepsilon)\nabla S_\varepsilon(\psi_{10\varepsilon}\nabla\phi_0)\|_{L^2(\mathbb{R}^d)}\Big\}\\
&\quad +\varepsilon\|\varpi(\cdot/\varepsilon)S_\varepsilon(\psi_{4\varepsilon}\nabla^2 u_0)
\|_{L^2(\Sigma_{3\varepsilon};\delta)}
\Big\{\|\nabla\phi_0\|_{L^2(\Sigma_{3\varepsilon};\delta^{-1})}
+\|(\nabla\chi)(\cdot/\varepsilon)S_\varepsilon(\psi_{10\varepsilon}\nabla \phi_0)\|_{L^2(\Sigma_{3\varepsilon};\delta^{-1})}
\Big\}\\
&\leq C\varepsilon\|\nabla^2 u_0\|_{L^2(\Sigma_{4\varepsilon})}
\Big\{\|\nabla\breve{w}_\varepsilon\|_{L^2(\Omega)}
+ \varepsilon\|\nabla(\psi_{10\varepsilon}\nabla\phi_0)\|_{L^2(\mathbb{R}^d)}\Big\} \\
&\quad + C\varepsilon \|\psi_{4\varepsilon}\nabla^2 u_0\|_{L^2(\Sigma_{3\varepsilon};\delta)}
\Big\{\|\nabla\phi_0\|_{L^2(\Sigma_{3\varepsilon};\delta^{-1})}
+\|\psi_{10\varepsilon}\nabla\phi_0\|_{L^2(\Sigma_{3\varepsilon};\delta^{-1})}\Big\},
\end{aligned}
\end{equation*}
where we use the Cauchy's inequality and the estimate $\eqref{pri:2.14}$ in the first inequality,
and the estimates $\eqref{pri:2.7}$, $\eqref{pri:2.10}$ and $\eqref{pri:2.11}$ in the last one.
Then it is not hard to see that
\begin{equation}\label{f:4.14}
\begin{aligned}
J_2
&\leq C\varepsilon \|\nabla^2 u_0\|_{L^2(\Sigma_{3\varepsilon};\delta)}
\|\nabla\phi_0\|_{L^2(\Sigma_{3\varepsilon};\delta^{-1})}\\
&\quad + C\varepsilon\|\nabla^2 u_0\|_{L^2(\Sigma_{4\varepsilon})}
\Big\{\|\nabla\breve{w}_\varepsilon\|_{L^2(\Omega)}
+\|\nabla\phi_0\|_{L^2(\Omega\setminus\Sigma_{20\varepsilon})}
+ \varepsilon\|\nabla^2\phi_0\|_{L^2(\Sigma_{10\varepsilon})}\Big\}.
\end{aligned}
\end{equation}
Combining the estimates $\eqref{f:4.13}$ and $\eqref{f:4.14}$ gives
\begin{equation}\label{f:4.15}
\begin{aligned}
I_{11}
&\leq C\big\|\nabla u_0\big\|_{L^2(\Omega\setminus\Sigma_{8\varepsilon})}\Big\{
\big\|\nabla\breve{w}_\varepsilon\big\|_{L^2(\Omega)}
+\big\|\nabla\phi_0\big\|_{L^2(\Omega\setminus\Sigma_{9\varepsilon})}\Big\}
+C\varepsilon \|\nabla^2 u_0\|_{L^2(\Sigma_{3\varepsilon};\delta)}
\|\nabla\phi_0\|_{L^2(\Sigma_{3\varepsilon};\delta^{-1})}\\
& + C\varepsilon\|\nabla^2 u_0\|_{L^2(\Sigma_{4\varepsilon})}
\Big\{\|\nabla\breve{w}_\varepsilon\|_{L^2(\Omega)}
+\|\nabla\phi_0\|_{L^2(\Omega\setminus\Sigma_{20\varepsilon})}
+ \varepsilon\|\nabla^2\phi_0\|_{L^2(\Sigma_{10\varepsilon})}\Big\}.
\end{aligned}
\end{equation}

By the same token, we have
\begin{equation}\label{f:4.16}
\begin{aligned}
I_{12}
&\leq C\big\|\nabla u_0\big\|_{L^2(\Omega\setminus\Sigma_{8\varepsilon})}\Big\{
\big\|\nabla\breve{w}_\varepsilon\big\|_{L^2(\Omega)}
+\big\|\nabla\phi_0\big\|_{L^2(\Omega\setminus\Sigma_{9\varepsilon})}\Big\}
+C\varepsilon \|\nabla^2 u_0\|_{L^2(\Sigma_{3\varepsilon};\delta)}
\|\nabla\phi_0\|_{L^2(\Sigma_{3\varepsilon};\delta^{-1})}\\
& + C\varepsilon\|\nabla^2 u_0\|_{L^2(\Sigma_{4\varepsilon})}
\Big\{\|\nabla\breve{w}_\varepsilon\|_{L^2(\Omega)}
+\|\nabla\phi_0\|_{L^2(\Omega\setminus\Sigma_{20\varepsilon})}
+ \varepsilon\|\nabla^2\phi_0\|_{L^2(\Sigma_{10\varepsilon})}\Big\}
\end{aligned}
\end{equation}
by noting that $q_{ik}^\gamma(y)$ plays a similar role as
the term $E_{jik}^{\alpha\gamma}(y)+a_{ij}^{\alpha\beta}(y)\chi_k^{\beta\gamma}(y)$ in $I_{11}$, and
the remaining things are exactly the same as we did in $I_{11}$.

Now we turn to estimate the term $I_{13}$, and
\begin{equation}\label{f:4.18}
\begin{aligned}
I_{13} &= \bigg|\int_\Omega\big[\hat{a}_{ij}^{\alpha\beta}-a_{ij}^{\alpha\beta}(y)\big]
\Big[\nabla_j u_0^\beta-S_\varepsilon(\psi_{4\varepsilon}\nabla_ju_0^\beta)\Big]
\nabla_i\phi_\varepsilon^\alpha dx\bigg| \\
&\leq C\int_\Omega\Big(\big|(1-\psi_{4\varepsilon})\nabla u_0\big|
+ \big|\psi_{4\varepsilon}\nabla u_0 - S_\varepsilon(\psi_{4\varepsilon}\nabla u_0)\big|\Big)
\big|\nabla\phi_\varepsilon\big| dx \\
&\leq C\Big\{\|\nabla u_0\|_{L^2(\Omega\setminus\Sigma_{8\varepsilon})}
\|\nabla\phi_\varepsilon\|_{L^2(\Omega\setminus\Sigma_{8\varepsilon})} + J_3\Big\}\\
&\leq C\bigg\{\|\nabla u_0\|_{L^2(\Omega\setminus\Sigma_{8\varepsilon})}
\Big(\big\|\nabla\breve{w}_\varepsilon\big\|_{L^2(\Omega)}
+\big\|\nabla\phi_0\big\|_{L^2(\Omega\setminus\Sigma_{9\varepsilon})}\Big)+ J_3
\bigg\},
\end{aligned}
\end{equation}
where $J_3 = \int_\Omega |\psi_{4\varepsilon}\nabla u_0 -
S_\varepsilon(\psi_{4\varepsilon}\nabla u_0)||\nabla\phi_\varepsilon| dx$, and in the last step
we use the same observation as in $J_1$. Also, we find that
$S_\varepsilon(\psi_{4\varepsilon}\nabla u_0)$ is supported in $\Sigma_{3\varepsilon}$,
and employ the same arguments as in $J_2$.
Hence we have
\begin{equation*}
\begin{aligned}
J_3 &= \int_\Omega\big|\psi_{4\varepsilon}\nabla u_0 - S_\varepsilon(\psi_{4\varepsilon}\nabla u_0)\big|
\big|\nabla\big[\breve{w}_\varepsilon + \phi_0 + \varepsilon\chi_{k}^{*}(y)
S_\varepsilon(\psi_{10\varepsilon}\nabla_k \phi_0)\big]\big|dx \\
&\leq \|\psi_{4\varepsilon}\nabla u_0 - S_\varepsilon(\psi_{4\varepsilon}\nabla u_0)\|_{L^2(\mathbb{R}^d)}
\Big\{\|\nabla\breve{w}\|_{L^2(\Omega)}
+\varepsilon\|\chi(\cdot/\varepsilon)
\nabla S_\varepsilon(\psi_{10\varepsilon}\nabla \phi_0)\|_{L^2(\mathbb{R}^d)}\Big\}\\
&\quad +\|\psi_{4\varepsilon}\nabla u_0
- S_\varepsilon(\psi_{4\varepsilon}\nabla u_0)\|_{L^2(\Sigma_{3\varepsilon};\delta)}
\Big\{\|\nabla\phi_0\|_{L^2(\Sigma_{3\varepsilon};\delta^{-1})}
+ \|(\nabla\chi)(\cdot/\varepsilon)
S_\varepsilon(\psi_{10\varepsilon}\nabla \phi_0)\|_{L^2(\Sigma_{3\varepsilon};\delta^{-1})}\Big\} \\
&\leq C\varepsilon\|\nabla(\psi_{4\varepsilon}\nabla u_0)\|_{L^2(\mathbb{R}^d)}
\Big\{\|\breve{w}_\varepsilon\|_{L^2(\Omega)}
+\varepsilon\|\nabla(\psi_{10\varepsilon}\nabla \phi_0)\|_{L^2(\mathbb{R}^d)}\Big\}\\
&\quad + C\varepsilon\|\nabla(\psi_{4\varepsilon}\nabla u_0)\|_{L^2(\Sigma_{3\varepsilon};\delta)}
\Big\{\|\nabla\phi_0\|_{L^2(\Sigma_{3\varepsilon};\delta^{-1})}
+\|\psi_{10\varepsilon}\nabla\phi_0\|_{L^2(\Sigma_{3\varepsilon};\delta^{-1})}\Big\},
\end{aligned}
\end{equation*}
where we use Cauchy's inequality and the estimate $\eqref{pri:2.14}$ in the first inequality,
and the estimates $\eqref{pri:2.7}$, $\eqref{pri:2.8}$, $\eqref{pri:2.11}$ and $\eqref{pri:2.12}$ in the
second one. Thus we have
\begin{equation*}
\begin{aligned}
J_3 &\leq C\Big\{\big\|\nabla u_0\big\|_{L^2(\Omega\setminus\Sigma_{8\varepsilon};\delta)}
+\varepsilon\big\|\nabla^2u_0\big\|_{L^2(\Sigma_{3\varepsilon};\delta)}\Big\}
\big\|\nabla\phi_0\big\|_{L^2(\Sigma_{3\varepsilon};\delta^{-1})}\\
& + C\Big\{\big\|\nabla u_0\big\|_{L^2(\Omega\setminus\Sigma_{8\varepsilon})}
+\varepsilon\|\nabla^2 u_0\|_{L^2(\Sigma_{4\varepsilon})}\Big\}
\cdot\Big\{\|\nabla\breve{w}_\varepsilon\|_{L^2(\Omega)}
+\|\nabla\phi_0\|_{L^2(\Omega\setminus\Sigma_{20\varepsilon})}
+ \varepsilon\|\nabla^2\phi_0\|_{L^2(\Sigma_{10\varepsilon})}\Big\}.
\end{aligned}
\end{equation*}
This together with $\eqref{f:4.15}$, $\eqref{f:4.16}$ and $\eqref{f:4.18}$ also leads to
\begin{equation}\label{f:4.17}
\begin{aligned}
I_1 &\leq C\Big\{\big\|\nabla u_0\big\|_{L^2(\Omega\setminus\Sigma_{8\varepsilon};\delta)}
+\varepsilon\big\|\nabla^2u_0\big\|_{L^2(\Sigma_{3\varepsilon};\delta)}\Big\}
\big\|\nabla\phi_0\big\|_{L^2(\Sigma_{3\varepsilon};\delta^{-1})}\\
& + C\Big\{\big\|\nabla u_0\big\|_{L^2(\Omega\setminus\Sigma_{8\varepsilon})}
+\varepsilon\|\nabla^2 u_0\|_{L^2(\Sigma_{4\varepsilon})}\Big\}
\cdot\Big\{\|\nabla\breve{w}_\varepsilon\|_{L^2(\Omega)}
+\|\nabla\phi_0\|_{L^2(\Omega\setminus\Sigma_{20\varepsilon})}
+ \varepsilon\|\nabla^2\phi_0\|_{L^2(\Sigma_{10\varepsilon})}\Big\}.
\end{aligned}
\end{equation}

Then we continue to study the term $I_2$ in $\eqref{f:4.10}$, and the trick of the proof to
$I_1$ also works here. The main difference is that
the auxiliary function $\breve{z}_\varepsilon$ in $\eqref{eq:4.2}$ is employed to
accelerate the convergence rate, and we therefore remind the reader to pay attention to
the role of the constant $\sigma$ in the proof. In view of $\eqref{eq:4.2}$, we have
$\theta_\varepsilon = \breve{z} + \theta_0 + \pi^*(\cdot/\varepsilon)
S_\varepsilon(\psi_{10\varepsilon}\nabla\phi_0)$ in $\Omega$, and then
\begin{equation*}
\begin{aligned}
I_2 & = \varepsilon\bigg|\int_\Omega (\theta_\varepsilon-\sigma)
\chi_{k}^{\beta\gamma}(\cdot/\varepsilon)
\nabla_\beta S_{\varepsilon}(\psi_{4\varepsilon}\nabla_k u_0^\gamma)dx\bigg| \\
&\leq \int_\Omega\Big|(\theta_\varepsilon -\sigma)
\chi_k^{\beta\gamma}(\cdot/\varepsilon)S_\varepsilon(\nabla_\beta\psi_{4\varepsilon}\nabla_k u_0^\gamma)\Big|dx
+ \varepsilon\bigg|\int_\Omega (\theta_\varepsilon-\sigma)
\chi_{k}^{\beta\gamma}(\cdot/\varepsilon)
 S_{\varepsilon}(\psi_{4\varepsilon}\nabla_{k\beta}^2 u_0^\gamma)dx\bigg|\\
&\leq \|\chi(\cdot/\varepsilon)S_\varepsilon(\nabla\psi_{4\varepsilon}\nabla u_0)\|_{L^2(\mathbb{R}^d)}
\|\breve{z}_\varepsilon-\sigma+\theta_0\|_{L^2(\Omega\setminus\Sigma_{9\varepsilon})} \\
& + \varepsilon\bigg|\int_\Omega \big[\breve{z}_\varepsilon-\sigma + \theta_0 +
\pi_k^{\gamma}(\cdot/\varepsilon)S_\varepsilon(\psi_{10\varepsilon}\nabla_k\phi_0^\gamma)\big]
\chi_{k}^{\beta\gamma}(\cdot/\varepsilon)
 S_{\varepsilon}(\psi_{4\varepsilon}\nabla_{k\beta}^2 u_0^\gamma)dx\bigg|,
\end{aligned}
\end{equation*}
where we note that $S_\varepsilon(\psi_{10\varepsilon}\nabla\phi_0)$ is supported in $\Sigma_{9\varepsilon}$.
Moreover, the right-hand side of the inequality above is controlled by
\begin{equation*}
\begin{aligned}
&\|\chi(\cdot/\varepsilon)S_\varepsilon(\nabla\psi_{4\varepsilon}\nabla u_0)\|_{L^d(\mathbb{R}^d)}
\Big\{\|\breve{z}_\varepsilon-\sigma\|_{L^2(\Omega)}
+\|\theta_0\|_{L^2(\Omega\setminus\Sigma_{9\varepsilon})}\Big\}
+\varepsilon\|\chi(\cdot/\varepsilon)S_\varepsilon(\psi_{4\varepsilon}\nabla^2u_0)\|_{L^2(\mathbb{R}^d)}
\|\breve{z}_\varepsilon-\sigma\|_{L^2(\Omega)}\\
&\qquad\quad+\varepsilon\|\chi(\cdot/\varepsilon)S_\varepsilon(\psi_{4\varepsilon}\nabla^2u_0)\|_{L^2(\Sigma_{3\varepsilon};\delta)}
\Big\{\|\theta_0\|_{L^2(\Sigma_{3\varepsilon};\delta^{-1})}
+\|\pi(\cdot/\varepsilon)S_\varepsilon(\psi_{10\varepsilon}
\nabla\phi_0)\|_{L^2(\Sigma_{3\varepsilon};\delta^{-1})}\Big\}.
\end{aligned}
\end{equation*}
Then we apply the estimates $\eqref{pri:2.7}$, $\eqref{pri:2.10}$ and $\eqref{pri:2.11}$
to the above expression, and consequently obtain
\begin{equation}\label{f:4.19}
\begin{aligned}
I_2 &\leq C\|\nabla u_0\|_{L^2(\Omega\setminus\Sigma_{8\varepsilon})}
\Big\{\|\breve{z}_\varepsilon\|_{L^2(\Omega)/\mathbb{R}}
+ \|\theta_0\|_{L^2(\Omega\setminus\Sigma_{9\varepsilon})}\Big\}
+ C\varepsilon\|\nabla^2 u_0\|_{L^2(\Sigma_{4\varepsilon})}
\|\breve{z}_\varepsilon\|_{L^2(\Omega)/\mathbb{R}}\\
&+ C\varepsilon\|\nabla^2u_0\|_{L^2(\Sigma_{3\varepsilon};\delta)}
\Big\{\|\theta_0\|_{L^2(\Sigma_{3\varepsilon};\delta^{-1})}
+\|\nabla\phi_0\|_{L^2(\Sigma_{3\varepsilon};\delta^{-1})}\Big\}
\end{aligned}
\end{equation}
by noting that $\sigma\in \mathbb{R}$ is arbitrary.
Combining $\eqref{f:4.10}$, $\eqref{f:4.17}$ and $\eqref{f:4.19}$ finally leads to
the desired estimate $\eqref{pri:4.1}$, and we have completed the proof.
\end{proof}

In fact, the following lemma is designed for smooth domains,
since we make a litter stronger assumption $u_0\in H^2(\Omega;\mathbb{R}^d)$ there.
This assumption is not very natural in our setting unless $\partial\Omega\in C^{1,1}$ at least
(see \cite[pp.1012]{SZW2}).
Nevertheless, it will release us from a complex calculation more or little, and
will lead to a better result than that established in \cite{G} concerning smooth domains.

\begin{lemma}[Duality lemma II]\label{lemma:4.2}
Assume $u_0\in H^2(\Omega;\mathbb{R}^d)$ coupled with $p_0\in L^2(\Omega)/\mathbb{R}$
satisfies $(\mathbf{DS_0})$.
Let $w_\varepsilon$ be given in $\eqref{eq:3.4}$ by setting
$\varphi_k^\gamma = S_\varepsilon(\psi_{2\varepsilon}\nabla_k \tilde{u}_0^\gamma)$, where
$\tilde{u}_0$ is the extension of $u_0$ such that $\tilde{u}_0 = u_0$ on $\Omega$ and
$\|\tilde{u}_0\|_{H^2(\mathbb{R}^d)}\leq C\|u_0\|_{H^2(\Omega)}$. For any $\Phi\in L^q(\Omega;\mathbb{R}^d)$,
let $(\phi_\varepsilon,\theta_\varepsilon)$ be the weak solution to $(\mathbf{DS_\varepsilon})^*$. Then we have
\begin{equation}\label{pri:4.2}
\begin{aligned}
\bigg|\int_\Omega w_\varepsilon\Phi dx\bigg|
&\leq C\|\nabla u_0\|_{L^2(\Omega\setminus\Sigma_{4\varepsilon})}
\Big(\|\nabla \phi_\varepsilon\|_{L^2(\Omega\setminus\Sigma_{5\varepsilon})}
+\|\theta_\varepsilon\|_{L^2(\Omega\setminus\Sigma_{5\varepsilon})/\mathbb{R}}\Big) \\
& \qquad + C\varepsilon\|u_0\|_{H^2(\Omega)}
\Big(\|\nabla \phi_\varepsilon\|_{L^2(\Omega)}
+\|\theta_\varepsilon\|_{L^2(\Omega)/\mathbb{R}}\Big),
\end{aligned}
\end{equation}
where $C$ depends on $\mu,d$ and $\Omega$.
\end{lemma}

\begin{proof}
Compared to the proof given in Lemma $\eqref{lemma:4.1}$, the following one will be
straightforward and simple.
In view of $\eqref{eq:4.1}$, we have
\begin{equation}\label{f:4.6}
\begin{aligned}
\bigg|\int_\Omega w_\varepsilon\Phi dx\bigg|
\leq \bigg|\int_\Omega \tilde{f}\cdot\nabla\phi_\varepsilon dx\bigg|
+ \bigg|\int_\Omega (\theta_\varepsilon-\sigma)\eta dx\bigg| =: I_1 + I_2,
\end{aligned}
\end{equation}
and then calculate $I_1$ and $I_2$, respectively.
\begin{equation}\label{f:4.5}
\begin{aligned}
I_1 \leq \varepsilon\bigg|\int_\Omega[E_{jik}^{\alpha\gamma}(y)
&+a_{ij}^{\alpha\beta}(y)\chi_k^{\beta\gamma}(y)]
\nabla_jS_\varepsilon(\psi_{2\varepsilon}\nabla_k\tilde{u}_0^\gamma)\nabla_i\phi_\varepsilon^\alpha dx\bigg|
+ \varepsilon\bigg|\int_\Omega q_{ik}^\gamma(y)\nabla_\alpha
S_\varepsilon(\psi_{2\varepsilon}\nabla_k\tilde{u}_0^\gamma)\nabla_i\phi_\varepsilon^\alpha dx\bigg|\\
&+ \bigg|\int_\Omega\big[\hat{a}_{ij}^{\alpha\beta}-a_{ij}^{\alpha\beta}(y)\big]
\Big[\nabla_j u_0^\beta-S_\varepsilon(\psi_{2\varepsilon}\nabla_j\tilde{u}_0^\beta)\Big]
\nabla_i\phi_\varepsilon^\alpha dx\bigg| =: I_{11} + I_{12}+ I_{13},
\end{aligned}
\end{equation}
where $y=x/\varepsilon$.
We proceed to estimate the term $I_{11}$. By setting
$\varpi_{jik}^{\alpha\gamma}(y)= E_{jik}^{\alpha\gamma}(y)
+a_{ij}^{\alpha\beta}(y)\chi_k^{\beta\gamma}(y)$, and $\varpi(y) = \big[\varpi_{jik}^{\alpha\gamma}(y)]$ again,
we have
\begin{equation}\label{f:4.2}
\begin{aligned}
I_{11} &= \varepsilon\bigg|\int_\Omega\varpi_{jik}^{\alpha\gamma}(\cdot/\varepsilon)
\nabla_jS_\varepsilon(\psi_{2\varepsilon}\nabla_k\tilde{u}_0^\gamma)\nabla_i\phi_\varepsilon^\alpha dx\bigg|\\
&\leq \bigg|\int_\Omega\varpi_{jik}^{\alpha\gamma}(\cdot/\varepsilon)
S_\varepsilon(\nabla_j\psi_{2\varepsilon}\nabla_k\tilde{u}_0^\gamma)\nabla_i\phi_\varepsilon^\alpha dx\bigg|
+\varepsilon\bigg|\int_\Omega\varpi_{jik}^{\alpha\gamma}(\cdot/\varepsilon)
S_\varepsilon(\psi_{2\varepsilon}\nabla^2_{jk}\tilde{u}_0^\gamma)\nabla_i\phi_\varepsilon^\alpha dx\bigg|\\
&\leq \big\|\varpi(\cdot/\varepsilon)S_\varepsilon(\nabla\psi_{2\varepsilon}\nabla\tilde{u}_0)
\big\|_{L^2(\mathbb{R}^d)}\big\|\nabla\phi_\varepsilon\big\|_{L^2(\Omega\setminus\Sigma_{5\varepsilon})}
+\varepsilon\big\|\varpi(\cdot/\varepsilon)S_\varepsilon(\psi_{2\varepsilon}\nabla^2\tilde{u}_0)
\big\|_{L^2(\mathbb{R}^d)}\big\|\nabla\phi_\varepsilon\big\|_{L^2(\Sigma_{\varepsilon})}\\
&\leq C\bigg\{\big\|\nabla u_0\big\|_{L^2(\Omega\setminus\Sigma_{4\varepsilon})}
\big\|\nabla\phi_\varepsilon\big\|_{L^2(\Omega\setminus\Sigma_{5\varepsilon})}
+\varepsilon\big\|\nabla^2 u_0\big\|_{L^2(\Sigma_{2\varepsilon})}
\big\|\nabla\phi_\varepsilon\big\|_{L^2(\Sigma_{\varepsilon})}\bigg\},
\end{aligned}
\end{equation}
where we use Cauchy's inequality in the second inequality, and the estimate $\eqref{pri:2.7}$ in the last one.
We mention that $S_\varepsilon(\nabla\psi_{2\varepsilon}\nabla\tilde{u}_0)$ is supported in
$\Omega\setminus\Sigma_{5\varepsilon}$ while $S_\varepsilon(\psi_{2\varepsilon}\nabla^2\tilde{u}_0)$
is supported in $\Sigma_{\varepsilon}$. By the same token, we have
\begin{equation}\label{f:4.3}
\begin{aligned}
I_{12} &= \varepsilon\bigg|\int_\Omega q_{ik}^\gamma(y)\nabla_\alpha
S_\varepsilon(\psi_{2\varepsilon}\nabla_k\tilde{u}_0^\gamma)\nabla_i\phi_\varepsilon^\alpha dx\bigg| \\
& \leq C\bigg\{\big\|\nabla u_0\big\|_{L^2(\Omega\setminus\Sigma_{4\varepsilon})}
\big\|\nabla\phi_\varepsilon\big\|_{L^2(\Omega\setminus\Sigma_{5\varepsilon})}
+\varepsilon\big\|\nabla^2 u_0\big\|_{L^2(\Sigma_{2\varepsilon})}
\big\|\nabla\phi_\varepsilon\big\|_{L^2(\Sigma_{\varepsilon})}\bigg\}
\end{aligned}
\end{equation}
by noting that $q_{ik}^\gamma(y)$ plays a similar role as
the term $E_{jik}^{\alpha\gamma}(y)+a_{ij}^{\alpha\beta}(y)\chi_k^{\beta\gamma}(y)$ did in $I_{11}$.
By the fact that $\tilde{u}_0$ is the extension of $u_0$, we have
$(1-\psi_{2\varepsilon})\nabla\tilde{u}_0 = (1-\psi_{2\varepsilon})\nabla u_0$ in $\Omega$, and then
\begin{equation*}
\nabla u_0 - S_\varepsilon(\psi_{2\varepsilon}\nabla \tilde{u}_0)
=\big[\nabla \tilde{u}_0 - S_\varepsilon(\nabla \tilde{u}_0)\big]
+\big[S_\varepsilon((1-\psi_{2\varepsilon})\nabla\tilde{u}_0) - (1-\psi_{2\varepsilon})\nabla\tilde{u}_0\big]
+(1-\psi_{2\varepsilon})\nabla u_0 \qquad\text{in}\quad\Omega.
\end{equation*}
Hence we calculate $I_{13}$ as follows
\begin{equation}\label{f:4.1}
\begin{aligned}
I_{13} &= \bigg|\int_\Omega\big[\hat{a}_{ij}^{\alpha\beta}-a_{ij}^{\alpha\beta}(y)\big]
\Big[\nabla_j u_0^\beta-S_\varepsilon(\psi_{2\varepsilon}\nabla_j\tilde{u}_0^\beta)\Big]
\nabla_i\phi_\varepsilon^\alpha dx\bigg| \\
&\leq C\int_\Omega\Big(\big|\nabla \tilde{u}_0 - S_\varepsilon(\nabla \tilde{u}_0) \big|
+ \big|S_\varepsilon((1-\psi_{2\varepsilon})\nabla\tilde{u}_0)
- (1-\psi_{2\varepsilon})\nabla\tilde{u}_0\big|
+ \big|(1-\psi_{2\varepsilon})\nabla u_0\big|\Big)\big|\nabla\phi_\varepsilon\big| dx \\
&\leq C\bigg\{
\big\|\nabla \tilde{u}_0 - S_\varepsilon(\nabla \tilde{u}_0)\big\|_{L^2(\mathbb{R}^d)}
\|\nabla\phi_\varepsilon\|_{L^2(\Omega)}
+\big\|S_\varepsilon((1-\psi_{2\varepsilon})\nabla\tilde{u}_0)
- (1-\psi_{2\varepsilon})\nabla\tilde{u}_0\big\|_{L^2(\mathbb{R}^d)}
\big\|\nabla\phi_\varepsilon\big\|_{L^2(\Omega\setminus\Sigma_{5\varepsilon})} \\
&+ \big\|\nabla u_0\big\|_{L^2(\Omega\setminus\Sigma_{4\varepsilon})}
\big\|\nabla\phi_\varepsilon\big\|_{L^2(\Omega\setminus\Sigma_{4\varepsilon})}\bigg\},
\end{aligned}
\end{equation}
where we use Cauchy's inequality in the last step. Also, we find that
$S_\varepsilon((1-\psi_{2\varepsilon})\nabla\tilde{u}_0)$ is supported in $\Omega\setminus\Sigma_{5\varepsilon}$
The right-hand side of $\eqref{f:4.1}$ is controlled by
\begin{equation*}
C\bigg\{
\varepsilon\|\nabla^2 \tilde{u}_0\|_{L^2(\mathbb{R}^d)}\|\nabla\phi_\varepsilon\|_{L^2(\Omega)}
+\varepsilon\|\nabla\big((1-\psi_{2\varepsilon})\nabla\tilde{u}_0\big)\|_{L^2(\mathbb{R}^d)}
\|\nabla\phi_\varepsilon\|_{L^2(\Omega\setminus\Sigma_{5\varepsilon})}
+ \big\|\nabla u_0\big\|_{L^2(\Omega\setminus\Sigma_{4\varepsilon})}
\big\|\nabla\phi_\varepsilon\big\|_{L^2(\Omega\setminus\Sigma_{4\varepsilon})}\bigg\}
\end{equation*}
due to the estimate $\eqref{pri:2.8}$. Reorganizing the above formula we consequently obtain
\begin{equation}\label{f:4.4}
I_{13}
\leq C\bigg\{
\varepsilon\|u_0\|_{H^2(\Omega)}\|\nabla\phi_\varepsilon\|_{L^2(\Omega)}
+ \big\|\nabla u_0\big\|_{L^2(\Omega\setminus\Sigma_{4\varepsilon})}
\big\|\nabla\phi_\varepsilon\big\|_{L^2(\Omega\setminus\Sigma_{5\varepsilon})}\bigg\}.
\end{equation}
Then plugging $\eqref{f:4.2}$, $\eqref{f:4.3}$ and $\eqref{f:4.4}$ into $\eqref{f:4.5}$, we obtain
\begin{equation}\label{f:4.7}
I_1\leq C\bigg\{\big\|\nabla u_0\big\|_{L^2(\Omega\setminus\Sigma_{4\varepsilon})}
\big\|\nabla\phi_\varepsilon\big\|_{L^2(\Omega\setminus\Sigma_{5\varepsilon})}
+\varepsilon\|u_0\|_{H^2(\Omega)}\|\nabla\phi_\varepsilon\|_{L^2(\Omega)}\bigg\}.
\end{equation}

We now turn to estimate the term $I_2$ in the right-hand side of $\eqref{f:4.6}$.
For any $\sigma\in\mathbb{R}$, we have
\begin{equation*}
\begin{aligned}
I_2 &= \varepsilon\bigg|\int_\Omega\chi_{k}^{\beta\gamma}(\cdot/\varepsilon)\nabla_\beta
S_\varepsilon(\psi_{2\varepsilon}\nabla_k\tilde{u}_0^\gamma)(\theta_\varepsilon-\sigma)dx\bigg|\\
&\leq \int_\Omega\Big|\chi_k^{\beta\gamma}(\cdot/\varepsilon)
S_\varepsilon(\nabla_\beta\psi_{2\varepsilon}\nabla_k\tilde{u}_0^\gamma)
(\theta_\varepsilon-\sigma)\Big|dx
+\varepsilon\int_\Omega\Big|\chi_k^{\beta\gamma}(\cdot/\varepsilon)
S_\varepsilon(\psi_{2\varepsilon}\nabla_{\beta k}^2\tilde{u}_0^\gamma)(\theta_\varepsilon-\sigma)\Big|dx\\
&\leq \big\|\chi(\cdot/\varepsilon)S_\varepsilon(\nabla\psi_{2\varepsilon}\nabla\tilde{u}_0)\big\|_{L^2(\mathbb{R}^d)}
\big\|\theta_\varepsilon-\sigma\big\|_{L^2(\Omega\setminus\Sigma_{5\varepsilon})}
+\varepsilon\big\|\chi(\cdot/\varepsilon)S_\varepsilon(\psi_{2\varepsilon}\nabla^2\tilde{u}_0)\big\|_{L^2(\mathbb{R}^d)}
\big\|\theta_\varepsilon-\sigma\big\|_{L^2(\Sigma_\varepsilon)}\\
&\leq C\bigg\{\|\nabla u_0\|_{L^2(\Omega\setminus\Sigma_{4\varepsilon})}
\|\theta_\varepsilon-\sigma\|_{L^2(\Omega\setminus\Sigma_{5\varepsilon})}
+\varepsilon\|\nabla^2 \tilde{u}_0\|_{L^2(\mathbb{R}^d)}
\|\theta_\varepsilon-\sigma\|_{L^2(\Sigma_\varepsilon)}\bigg\},
\end{aligned}
\end{equation*}
where we mention that $S_\varepsilon(\nabla_\beta\psi_{2\varepsilon}\nabla_k\tilde{u}_0^\gamma)$ is supported
in $\Omega\setminus\Sigma_{5\varepsilon}$ in the second inequality.
In the last step, we use the estimate $\eqref{pri:2.7}$. Since $\sigma\in\mathbb{R}$ is arbitrary, we have
\begin{equation}\label{f:4.8}
 I_2 \leq C\bigg\{\|\nabla u_0\|_{L^2(\Omega\setminus\Sigma_{4\varepsilon})}
\|\theta_\varepsilon\|_{L^2(\Omega\setminus\Sigma_{5\varepsilon})/\mathbb{R}}
+\varepsilon\|u_0\|_{H^2(\Omega)}
\|\theta_\varepsilon\|_{L^2(\Omega)/\mathbb{R}}\bigg\},
\end{equation}
where we use the fact that $\|\nabla^2\tilde{u}_0\|_{L^2(\mathbb{R}^d)}\leq C\|u_0\|_{H^2(\Omega)}$.
Consequently, combining $\eqref{f:4.6}$, $\eqref{f:4.7}$ and $\eqref{f:4.8}$ leads to the estimate
$\eqref{pri:4.2}$, and we have completed the proof.
\end{proof}

Compared with the results of Lemma $\ref{lemma:3.3}$, it is clear to see that the weighted-type norms
can notably
improve the $\varepsilon$'s power in the ``layer type'' estimate as well as in the ``co-layer type'' one.

\begin{lemma}[Improved lemma]\label{lemma:4.3}
Assume the same conditions as in Lemma $\ref{lemma:3.3}$.
Let $(u_0,p_0)$ be the solution to $(\textbf{DS}_0)$ with $F\in L^q(\Omega;\mathbb{R}^d)$,
$h\in W^{1,q}(\Omega)$ and $g\in H^1(\partial\Omega;\mathbb{R}^d)$ satisfying the compatibility condition
$\eqref{a:4}$, where $q=2d/(d+1)$. Then we have
\begin{equation}\label{pri:4.3}
\|\nabla u_0\|_{L^2(\Omega\setminus\Sigma_{p_1\varepsilon};\delta)}
+\|p_0\|_{L^2(\Omega\setminus\Sigma_{p_1\varepsilon};\delta)}
\leq C\varepsilon\Big\{\|F\|_{L^q(\Omega)} + \|h\|_{W^{1,q}(\Omega)} + \|g\|_{H^1(\partial\Omega)} \Big\},
\end{equation}
and further assuming $F\in L^2(\Omega;\mathbb{R}^d)$ and $h\in H^1(\Omega)$, we obtain
\begin{equation}\label{pri:4.4}
\begin{aligned}
\max\Big\{\|\nabla^2u_0\|_{L^2(\Sigma_{p_2\varepsilon};\delta)},
\|\nabla u_0\|_{L^2(\Sigma_{p_2\varepsilon};\delta^{-1})},
&\|p_0\|_{L^2(\Sigma_{p_2\varepsilon};\delta^{-1})}\Big\}\\
&\qquad\quad\leq C\big[\ln(c_0/\varepsilon)\big]^{\frac{1}{2}}
\Big\{\|F\|_{L^2(\Omega)} + \|h\|_{H^{1}(\Omega)}+\|g\|_{H^1(\partial\Omega)}\Big\},
\end{aligned}
\end{equation}
where $p_1,p_2>0$ are fixed real number, and $C$ depends on $\mu,d,p_1,p_2$ and $\Omega$.
\end{lemma}

\begin{proof}
We first address the estimate $\eqref{pri:4.3}$. It follows the definition of weighted-type norm $\eqref{def:2.2}$ and the estimate
$\eqref{pri:3.6}$ that
\begin{equation*}
\begin{aligned}
\|\nabla u_0\|_{L^2(\Omega\setminus\Sigma_{p_1\varepsilon};\delta)}
+\|p_0\|_{L^2(\Omega\setminus\Sigma_{p_1\varepsilon};\delta)}
&\leq C\varepsilon^{\frac{1}{2}}\Big\{
\|\nabla u_0\|_{L^2(\Omega\setminus\Sigma_{p_1\varepsilon})}
+\|p_0\|_{L^2(\Omega\setminus\Sigma_{p_1\varepsilon})}\Big\} \\
&\leq C\varepsilon\Big\{\|F\|_{L^q(\Omega)} + \|h\|_{W^{1,q}(\Omega)} + \|g\|_{H^1(\partial\Omega)} \Big\}.
\end{aligned}
\end{equation*}

Now we continue to handle the estimate $\eqref{pri:4.4}$. Proceeding as in the proof of $\eqref{pri:3.7}$ in
Lemma $\ref{lemma:3.3}$,  we first arrive at
\begin{equation}\label{f:4.20}
\begin{aligned}
\|\nabla^2u_0\|_{L^2(\Sigma_{p_2\varepsilon};\delta)}
&\leq C\|\nabla^2 v\|_{L^2(\mathbb{R}^d)} + \|\nabla^2 w\|_{L^2(\Sigma_{p_2\varepsilon};\delta)} \\
&\leq C\Big\{\|F\|_{L^2(\Omega)} + \|h\|_{H^1(\Omega)}\Big\}
+ \|\nabla^2 w\|_{L^2(\Sigma_{p_2\varepsilon};\delta)},
\end{aligned}
\end{equation}
where we use the hypothesis that $\delta(x) = 0$ when $x\in\mathbb{R}^d\setminus\Omega$ in the first
step, and the estimate $\eqref{f:3.3}$ (with $q=2$) in the last one. By noting $\eqref{f:3.27}$
we actually have
\begin{equation*}
\begin{aligned}
\|\nabla^2 w\|_{L^2(\Sigma_{p_2\varepsilon};\delta)}
&\leq \|\nabla^2 w\|_{L^2(\Sigma_{p_2\varepsilon}\setminus\Sigma_{c_0};\delta)}
+ C\|\nabla w\|_{L^2(\Omega)} \\
&\leq \|\nabla^2 w\|_{L^2(\Sigma_{p_2\varepsilon}\setminus\Sigma_{c_0};\delta)}
+ C\Big\{\|F\|_{L^q(\Omega)}+\|h\|_{W^{1,q}(\Omega)}
+\|g\|_{H^1(\partial\Omega)}\Big\},
\end{aligned}
\end{equation*}
where the estimate $\eqref{f:3.15}$ is used in the last step. The remaining thing is to estimate the first
term in the right-hand side above, and the proof is very similar to that in $\eqref{f:3.13}$.
It follows from the estimate $\eqref{f:3.12}$ that
\begin{equation*}
\begin{aligned}
\int_{\Sigma_{p_2\varepsilon}\setminus\Sigma_{c_0}}|\nabla^2 w|^2 \delta(x)dx
&\leq C\int_{\Sigma_{p_2\varepsilon}\setminus\Sigma_{c_0}}\dashint_{B(x,\delta(x)/8)}
\frac{|\nabla w(y)|^2}{\delta(x)} dy dx \\
&\leq C\int_{p_2\varepsilon}^{c_0}\int_{S_t}\frac{|(\nabla w)^*(x^\prime)|^2}{t}dS_t(\Lambda_t(x^\prime))dt\\
&\leq C\int_{\partial\Omega}|(\nabla w)^*(x^\prime)|^2dS(x^\prime)\int_{p_2\varepsilon}^\infty\frac{dt}{t}\\
&\leq C\ln(c_0/\varepsilon)\|(\nabla w)^*\|_{L^2(\partial\Omega)}^2
\leq C\ln(c_0/\varepsilon)\Big\{\|F\|_{L^q(\Omega)}^2+\|h\|_{W^{1,q}(\Omega)}^2
+\|g\|_{H^1(\partial\Omega)}^2\Big\},
\end{aligned}
\end{equation*}
where we use the estimate $\eqref{f:3.8}$ in the last step. This together with $\eqref{f:4.20}$
partially gives the estimate $\eqref{pri:4.4}$, and we continue to consider how to estimate the quantities
$\|\nabla u_0\|_{L^2(\Sigma_{p_2\varepsilon};\delta^{-1})}$
and $\|p_0\|_{L^2(\Sigma_{p_2\varepsilon};\delta^{-1})}$. For the convenience, they have been
calculated together in the following:
\begin{equation}\label{f:4.21}
\begin{aligned}
\|\nabla u_0\|_{L^2(\Sigma_{p_2\varepsilon};\delta^{-1})}^2
&+ \|p_0\|_{L^2(\Sigma_{p_2\varepsilon};\delta^{-1})}^2
\leq \int_{\Sigma_{p_2\varepsilon}\setminus\Sigma_{c_0}}\big(|\nabla u_0|^2 + |p_0|^2\big)
\frac{dx}{\delta(x)}
+ \frac{1}{c_0}\int_{\Sigma_{c_0}}\big(|\nabla u_0|^2 + |p_0|^2\big)dx\\
&\leq C\int_{\partial\Omega}\big(|\mathcal{M}(\nabla u_0)|^2 + |\mathcal{M}(p_0)|^2\big)dS
\int_{p_2\varepsilon}^{c_0}\frac{dt}{t} + C\Big\{\|\nabla u_0\|_{L^2(\Omega)}^2+\|p_0\|_{L^2(\Omega)}^2\Big\}\\
&\leq C\ln(c_0/\varepsilon)
\Big\{\|\mathcal{M}(\nabla u_0)\|_{L^2(\partial\Omega)}^2
+\|\mathcal{M}(p_0)\|_{L^2(\partial\Omega)}^2
+\|\nabla u_0\|_{L^2(\Omega)}^2+\|p_0\|_{L^2(\Omega)}^2\Big\}\\
&\leq C\ln(c_0/\varepsilon)\Big\{
\|F\|_{L^2(\Omega)}^2 + \|h\|_{H^1(\Omega)}^2+\|g\|_{H^1(\partial\Omega)}^2\Big\}.
\end{aligned}
\end{equation}
In fact some explanations are needed for the last step above. Take pressure term as an example
(the term $\nabla u_0$ obeys the same computation), and since $p_0 = p_{0,1}+p_{0,2}$, where
$p_{0,1}$ and $p_{0,2}$ are given by $\eqref{pde:3.4}$, we have
\begin{equation*}
\begin{aligned}
\|\mathcal{M}(p_0)\|_{L^2(\partial\Omega)}
&\leq \|\mathcal{M}(p_{0,1})\|_{L^2(\partial\Omega)}
+ \|\mathcal{M}(p_{0,2})\|_{L^2(\partial\Omega)} \\
&\leq C\|p_{0,1}\|_{H^1(\Omega\setminus\Sigma_{c_0})} +  \|(p_{0,2})^*\|_{L^2(\partial\Omega)}
\leq C\Big\{\|F\|_{L^2(\Omega)} + \|h\|_{H^1(\Omega)}+\|g\|_{H^1(\partial\Omega)}\Big\},
\end{aligned}
\end{equation*}
where we use the estimate $\eqref{pri:2.15}$ and the fact that
$\mathcal{M}(p_{0,2})(z)\leq (p_{0,2})^*(z)$ for a.e. $z\in\partial\Omega$ in the second step,
and the estimates $\eqref{f:3.3}$ and $\eqref{f:3.8}$ in the last one.
Hence collecting the estimates $\eqref{f:4.20}$ and $\eqref{f:4.21}$ consequently leads to the desired
estimate $\eqref{pri:4.4}$, and we are done.
\end{proof}

\begin{thm}
Assume that $A$ satisfies $\eqref{a:1}-\eqref{a:3}$. Suppose that $F\in L^2(\Omega;\mathbb{R}^d)$,
$h\in H^1(\Omega)$ and $g\in H^1(\partial\Omega;\mathbb{R}^d)$ with the compatibility condition
$\eqref{a:4}$.
Let $(u_\varepsilon,p_\varepsilon)$ and $(u_0,p_0)$ in $H^1(\Omega;\mathbb{R}^d)\times L^2(\Omega)$ be the weak solutions of the Dirichlet problems
$(\textbf{DS})_\varepsilon$ and $(\textbf{DS})_0$, respectively. Then
we have
\begin{equation}\label{pri:4.5}
\big\|u_\varepsilon - u_0 - \varepsilon\chi_{k}(\cdot/\varepsilon)
S_\varepsilon(\psi_{4\varepsilon}\nabla_k u_0)\big\|_{L^2(\Omega)}
\leq C\varepsilon\ln(r_0/\varepsilon)\Big\{\|F\|_{L^2(\Omega)} + \|h\|_{H^1(\Omega)}
+ \|g\|_{H^1(\partial\Omega)}.
\Big\}
\end{equation}
where $C$ depends on $\mu,d$ and $\Omega$.
\end{thm}

\begin{proof}
It is convenient to assume $\|F\|_{L^2(\Omega)} + \|h\|_{H^1(\Omega)}
+ \|g\|_{H^1(\partial\Omega)} = 1$ on account of the linearity of
$(\textbf{DS})_\varepsilon$ and $(\textbf{DS})_0$. Thus by setting
\begin{equation*}
w_\varepsilon = u_\varepsilon - u_0 - \varepsilon\chi_k(\cdot/\varepsilon)
S_\varepsilon(\psi_{4\varepsilon}\nabla_ku_0)
\end{equation*}
it is equivalent to proving $\|w_\varepsilon\|_{L^2(\Omega)}\leq C\ln(r_0/\varepsilon)$.
Let $\Phi\in L^2(\Omega;\mathbb{R}^d)$, and $(\phi_\varepsilon,\theta_\varepsilon),
(\phi_0,\theta_0)$ be the solutions of the corresponding adjoint problems
$(\textbf{DS})_\varepsilon^*$ and $(\textbf{DS})_0^*$.
Due to the estimate $\eqref{pri:4.1}$ in Lemma $\ref{lemma:4.1}$, the remaining thing is to estimate the right-hand side of
$\eqref{pri:4.1}$ term by term. Then we first have

\begin{equation}\label{f:4.22}
\begin{aligned}
\Big\{\big\|\nabla u_0\big\|_{L^2(\Omega\setminus\Sigma_{8\varepsilon};\delta)}
&+\varepsilon\big\|\nabla^2u_0\big\|_{L^2(\Sigma_{3\varepsilon};\delta)}\Big\}\\
&\cdot\Big\{\big\|\nabla\phi_0\big\|_{L^2(\Sigma_{3\varepsilon};\delta^{-1})}
+\big\|\theta_0\big\|_{L^2(\Sigma_{3\varepsilon};\delta^{-1})}\Big\}
\leq C\Big\{\varepsilon + \varepsilon\big[\ln(r_0/\varepsilon)\big]^{\frac{1}{2}}\Big\}
\cdot\big[\ln(r_0/\varepsilon)\big]^{\frac{1}{2}}\big\|\Phi\big\|_{L^2(\Omega)},
\end{aligned}
\end{equation}
where we use the estimates $\eqref{pri:4.3}$ and $\eqref{pri:4.4}$, and then obtain
\begin{equation}\label{f:4.23}
\begin{aligned}
 \Big\{\big\|\nabla u_0\big\|_{L^2(\Omega\setminus\Sigma_{8\varepsilon})}
&+\varepsilon\|\nabla^2 u_0\|_{L^2(\Sigma_{4\varepsilon})}\Big\}\\
&\cdot\Big\{\|\nabla\breve{w}_\varepsilon\|_{L^2(\Omega)}
+\|\breve{z}_\varepsilon\|_{L^2(\Omega)/\mathbb{R}}
+\|\nabla\phi_0\|_{L^2(\Omega\setminus\Sigma_{20\varepsilon})}
+\|\theta_0\|_{L^2(\Omega\setminus\Sigma_{9\varepsilon})}
+ \varepsilon\|\nabla^2\phi_0\|_{L^2(\Sigma_{10\varepsilon})}\Big\}\\
&\leq C\Big\{\varepsilon^{\frac{1}{2}}+\varepsilon^{\frac{1}{2}}\Big\}
\cdot\Big\{\varepsilon^{\frac{1}{2}} + \varepsilon^{\frac{1}{2}} + \varepsilon^{\frac{1}{2}}
+ \varepsilon^{\frac{1}{2}}\Big\}\big\|\Phi\big\|_{L^2(\Omega)}
\leq C\varepsilon\big\|\Phi\big\|_{L^2(\Omega)}
\end{aligned}
\end{equation}
where we employ the estimates $\eqref{pri:3.6}$, $\eqref{pri:3.7}$,
$\eqref{pri:3.10}$ and $\eqref{pri:3.12}$.
Inserting $\eqref{f:4.22}$ and $\eqref{f:4.23}$ into the estimate $\eqref{pri:4.1}$,
we consequently arrive at
$\big|\int_{\Omega}w_\varepsilon\Phi dx\big|\leq C\varepsilon\ln(r_0/\varepsilon)
\|\Phi\|_{L^2(\Omega)}$, and this leads to the desired estimate $\eqref{pri:4.5}$.
We have completed the proof.
\end{proof}

\section{$O(\varepsilon)$ convergence rates for $d=2$}

\begin{lemma}\label{lemma:6.1}
Let $d\geq 2$,
and  $\big\{E_{jik}^{\alpha\gamma},T_{ik}^{\alpha\gamma},b_{ik}^{\alpha\gamma},q_{ik}^\gamma\big\}$ be given
in Lemma $\ref{lemma:2.3}$. If the correctors $\chi_k=(\chi_k^{\beta\gamma})$ with $k=1,\cdots,d$ are H\"older continuous,
then $E_{jik}^{\alpha\gamma},q_{ik}^\gamma\in L^\infty(Y)$.
\end{lemma}

\begin{proof}
We mention that the original idea of the proof belongs to \cite{SZW16}.
It is convenient to assume $\chi_k\in C^{0,\sigma}(\mathbb{R}^d;\mathbb{R}^{d^2})$ with $\sigma\in(0,1)$.
It follows from Caccioppoli's inequality $\eqref{pri:2.13}$ that
\begin{equation*}
 \int_{B(x,r)} |\nabla \chi_k|^2 dy
 \leq \frac{C}{r^2} \int_{B(x,2r)} |\chi_k(y) - \chi_k(x)|^2 dy + Cr^d
 \leq Cr^{d+2\sigma-2}
\end{equation*}
for any $r\in(0,1)$ and $x\in Y$. Recalling the expression of $b_{ik}^{\alpha\gamma}$
in Lemma $\ref{lemma:2.3}$, it follows from the above inequality that
\begin{equation}\label{f:6.1}
\dashint_{B(x,r)}|b_{ik}^{\alpha\gamma}|^2 dy \leq C\big\{1+r^{2\sigma-2}\big\}
\leq Cr^{2\sigma-2}.
\end{equation}
By translation we may assume $Y$ is centered at $0$, and construct the smooth cut-off function
(still denoted by $\psi$) such that $\psi = 1$ in $B(0,3/2)$ and $\psi = 0$ outside $B(0,2)$.
Let $w = (w_{ik}^{\alpha\gamma})$ with $w_{ik}^{\alpha\gamma} = \psi T_{ik}^{\alpha\gamma}$ and
$z=(z_{ik}^\gamma)$ with $z_{ik}^\gamma = \psi q_{ik}^\gamma$. Then
by a localization argument it follows from the equation $\eqref{pde:2.2}$ that
\begin{equation}\label{eq:6.1}
\Delta w + \nabla z = \tilde{F}
\quad\text{and}\quad \text{div}(w) = \tilde{h}
\qquad\text{in}~~\mathbb{R}^d,
\end{equation}
where $\tilde{F}_{ik}^{\alpha\gamma} = \psi b_{ik}^{\alpha\gamma}
+\nabla_\alpha\psi q_{ik}^\gamma + \nabla\psi\cdot\nabla T_{ik}^{\alpha\gamma}
+ \Delta\psi T_{ik}^{\alpha\gamma}$, and
$\tilde{h}_{ik}^\gamma = \nabla_\alpha\psi T_{ik}^{\alpha\gamma}$.
We mention that $\tilde{F}$ and $\tilde{h}$ are supported in $B(0,2)$.

Furthermore, let $\Delta v = \tilde{h}$
in $\mathbb{R}^d$, and it is clear to see that $\text{div}(w-\nabla v) = 0$ in $\mathbb{R}^d$. If we set
$\tilde{w} = w-\nabla v$, then the equation $\eqref{eq:6.1}$ becomes
$\Delta \tilde{w} + \nabla z = \tilde{F}-\Delta(\nabla v)$ and $\text{div}(\tilde{w}) = 0$ in $\mathbb{R}^d$.
Thus in view of the fundamental solution to Stokes systems (see \cite[Chapter 3]{OAL}) we have
$\tilde{w} = \textbf{U}*(\tilde{F}-\nabla\tilde{h})$ and
$z = \textbf{Q}*(\tilde{F}-\nabla\tilde{h})$, where we use the
fact that $\Delta(\nabla v) = \nabla (\Delta v) = \nabla\tilde{h}$ in $\mathbb{R}^d$.
Hence denoting the fundamental solution to $\Delta$ by $\Gamma$ it follows that
$v=\Gamma*\tilde{h}$, and
\begin{equation}
  w = \textbf{U}*\big(\tilde{F}-\nabla\tilde{h}\big) + \nabla \Gamma*\tilde{h},
  \qquad \text{and}\qquad
  \qquad z = \textbf{Q}*(\tilde{F}-\nabla\tilde{h}).
\end{equation}
From the expressions of $\textbf{U}$ and $\textbf{Q}$ (see $\eqref{eq:3.9}$), it is not hard to see that
$|\nabla \textbf{U}(x)|\leq C|x|^{1-d}$ and $|\nabla \textbf{Q}(x)|\leq C|x|^{1-d}$.
Also $|\nabla\Gamma(x)|\leq C|x|^{1-d}$. So, to estimate
$|\nabla w(x)|$ and $|\nabla z(x)|$,
an important thing is to analyze the terms $\tilde{F}$, $\tilde{h}$ and $\nabla\tilde{h}$. For any
$x\in Y$,
\begin{equation}
\begin{aligned}
|\nabla w(x)|
&\leq C\int_{\mathbb{R}^d}\frac{1}{|x-y|^{d-1}}\Big\{|\tilde{F}(y)|+|\nabla\tilde{h}(y)|\Big\}dy\\
&\leq C\sum_{i,k,\alpha,\gamma=1}^d\bigg\{\int_{B(x,3)}\frac{|b_{ik}^{\alpha\gamma}(y)|dy}{|x-y|^{d-1}}
+ \int_{B(0,2)\setminus B(0,\frac{3}{2})}\frac{1}{|x-y|^{d-1}}\Big(|q_{ik}^{\gamma}(y)|
+|\nabla T_{ik}^{\alpha\gamma}(y)|
+|T_{ik}^{\alpha\gamma}(y)|\Big)dy\bigg\}\\
&\leq C\sum_{i,k,\alpha,\gamma=1}^d\sum_{k=0}^\infty \int_{2^{-k+1}<|x-y|<2^{-k+2}}
\frac{|b_{ik}^{\alpha\gamma}(y)|}{|x-y|^{d-1}} dy
+C\sum_{i,k,\alpha,\gamma=1}^d\Big\{\|q_{ik}^\gamma\|_{L^2(Y)}
+\|T_{ik}^{\alpha\gamma}\|_{H^1(Y)}\Big\}\\
&\leq C\sum_{i,k,\alpha,\gamma=1}^d
\bigg\{\sum_{k=0}^\infty 2^{(1-k)(1-d)}\cdot 2^{(2-k)d}
\Big(\dashint_{|x-y|<2^{-k+2}}|b_{ik}^{\alpha\beta}|^2 dy\Big)^{1/2}+ 1\bigg\}\\
&\leq C\Big\{\sum_{k=0}^\infty 2^{-\sigma k} + 1\Big\}<\infty,
\end{aligned}
\end{equation}
where we use Cauchy's inequality in the third step, and the estimate $\eqref{f:2.2}$ in the fourth one. In the last
inequality we employ the estimate $\eqref{f:6.1}$. This implies
$\|\nabla T_{ik}^{\alpha\gamma}\|_{L^\infty}(Y)\leq C$.
By the same token, it is easy to derive $\|\nabla q_{ik}^{\gamma}\|_{L^\infty}(Y)\leq C$.
By recalling the expression of $E_{jik}^{\alpha\gamma}$ in the proof of Lemma $\ref{lemma:2.3}$, we have
$\|E_{jik}^{\alpha\gamma}\|_{L^\infty(Y)}\leq C$, and we have completed the proof.
\end{proof}

\begin{lemma}\label{lemma:6.2}
Let $d=2$. Suppose that $A$ satisfies $\eqref{a:1}$ and $\eqref{a:2}$. Assume the corrector
$(\chi_k^{\beta\gamma},\pi_k^\gamma)\in H^1_{per}(Y)\times L^2_{per}(Y)/\mathbb{R}$ satisfies
$\eqref{pde:2.1}$.
Then for some $\sigma\in(0,1)$ we have $\chi_k^{\beta\gamma}\in C^{0,\sigma}(\mathbb{R}^d)$
and $\pi_k^\gamma\in L_{loc}^{2,2\sigma}(\mathbb{R}^d)$
with $k,\beta,\gamma=1,2$.
\end{lemma}

\begin{proof}
We take the hole-filling technique (see \cite{MGMG} and originally developed in \cite{KOW}) to handle this estimate, and provide a
proof for the sake of the completeness. For any $B=B(x,R)\subset\mathbb{R}^d$, we may assume $x=0$ by translation,
and it follows from Caccioppoli's inequality $\eqref{pri:2.13}$ that
\begin{equation*}
\int_{B} |\nabla\chi_k^\gamma|^2 dy \leq C\Big\{\frac{1}{R^2}\int_{2B\setminus B}
|\chi_k^\gamma-c|^2dy + R^2\Big\} \qquad \forall c\in\mathbb{R}^d.
\end{equation*}
By adding the term $C\int_B|\nabla\chi_k^\gamma|dy$ in the both sides of the above inequality, we have
\begin{equation*}
(C+1)\int_{B} |\nabla\chi_k^\gamma|^2 dy \leq C\Big\{\frac{1}{R^2}\int_{2B}
|\chi_k^\gamma-c|^2dy + R^2\Big\}
\leq  C\Big\{\int_{2B}
|\nabla \chi_k^\gamma|^2dy + R^2\Big\}.
\end{equation*}
Hence set $\theta = C/(C+1)$ and $\phi(R) = \int_{B(x,R)}|\nabla\chi_k^\gamma|^2dy$, and we then have
\begin{equation*}
\phi(R) = \theta\phi(2R) + \theta R^2.
\end{equation*}
Iterating the above formula with respect to $R$, one may derive
\begin{equation*}
 \phi(2^{-k}R) \leq \theta^{k+1}\Big\{\phi(2R)+\frac{4\theta}{4\theta+1}R^2\Big\}.
\end{equation*}
We choose $\rho>0$ such that $2^{-k-1}R<\rho\leq 2^{-k}R$, and then obtain
\begin{equation*}
\phi(\rho) \leq \left(\frac{\rho}{R}\right)^{\log_2\frac{1}{\theta}}
\Big\{\phi(2R)+R^2\Big\}.
\end{equation*}
By setting $2\sigma = \log_2\frac{1}{\theta}$ and $R=1$, we consequently arrive at
$\phi(\rho) \leq C\rho^{2\sigma}$, where we employ the estimate $\ref{pri:2.4}$ to handle
the term $\phi(2)$.
Thus the desired result $\chi_k^{\beta\gamma}\in C^{0,\sigma}(\mathbb{R}^d)$
follows from the Morrey theorem (see \cite[Theorem 5.7]{MGLM}) and the periodicity of
$\chi_k^{\beta\gamma}$. Moreover, on account of Lemma $\eqref{pri:2.1}$ one may have
\begin{equation}
 \int_{B(x,\rho)}|\pi_k^\gamma -c_1|^2 dy
 \leq C\Big\{\int_{B(x,\rho)}|\nabla\chi_k^\gamma|^2 dy + \rho^2\Big\}
 \leq C\rho^{2\sigma}
\end{equation}
for any $c_1\in\mathbb{R}$.
By setting $c_1 = \dashint_{B(x,\rho)} \pi_k^\gamma dy$ it is clear to see that
$\pi_k^\gamma\in L^{2,2\sigma}_{loc}(\mathbb{R}^d)$,
where the space $L^{2,2\sigma}_{loc}(\mathbb{R}^d)$ is the Morrey space
(see \cite[Definition 5.1]{MGLM}). We have completed the proof.
\end{proof}

\begin{thm}\label{thm:6.1}
Let $d=2$ and $1\leq q<\infty$. Suppose that $A$ satisfies $\eqref{a:1}$ and $\eqref{a:2}$,
and $u_0\in H^2(\Omega;\mathbb{R}^d)$. Assume $(u_\varepsilon,p_\varepsilon)$ and
$(u_0,p_0)$ in $H^1(\Omega;\mathbb{R}^d)\times L^2(\Omega)/\mathbb{R}$ satisfy the equations
$(\textbf{DS}_\varepsilon)$ and $(\textbf{DS}_0)$, respectively.
If we set $\varphi_k^\gamma = \nabla_k u_0^\gamma$ in $\eqref{eq:3.4}$, then we have the estimate
$\eqref{pri:1.4}$ where $C$ depends only on $\mu,d$ and $\Omega$.
\end{thm}

\begin{proof}
Since we choose $\varphi_k^\gamma = \nabla_k u_0^\gamma$ in $\eqref{eq:3.4}$ this time,
the term $\tilde{f}_i^\alpha$ in $\eqref{eq:3.5}$ turns to
\begin{equation*}
\tilde{f}_i^\alpha
= \varepsilon \big[E_{jik}^{\alpha\gamma}(y)+a_{ij}^{\alpha\beta}(y)\chi_k^{\beta\gamma}(y)\big]
\frac{\partial^2 u_0^\gamma}{\partial x_j\partial x_k}
-\varepsilon q_{ik}^\gamma(y)\frac{\partial^2 u_0^\gamma}{\partial x_\alpha\partial x_k},
\end{equation*}
and we have $\eta = -\varepsilon\chi_{k}^{\beta\gamma}(\cdot/\varepsilon)\nabla_{\beta k}^2 u_0^\gamma$
and $w_\varepsilon^\beta = -\varepsilon\chi_{k}^{\beta\gamma}(\cdot/\varepsilon)\nabla_k u^\gamma_0$ on $\partial\Omega$ in $\eqref{pde:3.2}$.
In view of Lemmas $\eqref{lemma:6.1}$ and $\eqref{lemma:6.2}$, it is clear to see that
$E_{jik}^{\alpha\gamma}+a_{ij}^{\alpha\beta}\chi_k^{\beta\gamma}$ and $q_{ik}^\gamma$ belong to
$L^\infty(\mathbb{R}^d)$. Thus from the estimate $\eqref{pri:2.6}$ we have
\begin{equation*}
\|w_\varepsilon\|_{H_0^1(\Omega)} + \|z_\varepsilon\|_{L^2(\Omega)/\mathbb{R}}
\leq C\varepsilon\Big\{\|u_0\|_{L^2(\Omega)}+\|\nabla u_0\|_{H^{\frac{1}{2}}(\partial\Omega)}\Big\}
\leq C\varepsilon\|u_0\|_{H^2(\Omega)}.
\end{equation*}
where $w_\varepsilon$ and $z_\varepsilon$ is given in $\eqref{eq:3.4}$ by fixing
$\varphi_k^\gamma = \nabla_k u_0^\gamma$. Due to the Sobolev imbedding theorem, we arrive at
\begin{equation*}
\|w_\varepsilon\|_{L^q(\Omega)} + \|z_\varepsilon\|_{L^2(\Omega)/\mathbb{R}}
\leq C\varepsilon\|u_0\|_{L^2(\Omega)}
\end{equation*}
for any $1\leq q<\infty$. This implies the desired estimate $\eqref{pri:1.4}$ and we are done.
\end{proof}

\begin{flushleft}
\textbf{Proof of Theorem \ref{thm:1.1}.}
This theorem includes the estimates $\eqref{pri:1.1}$, $\eqref{pri:1.3}$, $\eqref{pri:1.2}$
and $\eqref{pri:1.4}$, where
the estimate $\eqref{pri:1.3}$ is related to the pressure term, and its proof is shown in Corollary
$\ref{cor:3.1}$. The estimate $\eqref{pri:1.4}$ is actually built for the special case $d=2$,
and we have already shown it in Theorem $\ref{lemma:6.1}$.
So the remaining thing is to estimate $\eqref{pri:1.1}$ and $\eqref{pri:1.2}$ in the proof.
We first show the estimate $\eqref{pri:1.1}$. Let
\end{flushleft}
\begin{equation*}
  w_\varepsilon = u_\varepsilon - u_0 -
  \varepsilon\chi_k(\cdot/\varepsilon)S_\varepsilon(\psi_{4\varepsilon}\nabla_k u_0).
\end{equation*}
Due to the estimate $\eqref{pri:4.5}$, it is not hard to see that
\begin{equation*}
\begin{aligned}
\|u_\varepsilon - u_0\|_{L^2(\Omega)}
&\leq \|w_\varepsilon\|_{L^2(\Omega)}
+ \varepsilon\|\chi_k(\cdot/\varepsilon)S_\varepsilon(\psi_{4\varepsilon}\nabla_ku_0)\|_{L^2(\Omega)}\\
&\leq C\varepsilon\ln(r_0/\varepsilon)\Big\{\|F\|_{L^2(\Omega)}
+\|h\|_{H^1(\Omega)}+\|g\|_{H^1(\partial\Omega)}\Big\} + C\varepsilon\|\nabla u_0\|_{L^2(\Omega)}\\
&\leq  C\varepsilon\ln(r_0/\varepsilon)\Big\{\|F\|_{L^2(\Omega)}
+\|h\|_{H^1(\Omega)}+\|g\|_{H^1(\partial\Omega)}\Big\},
\end{aligned}
\end{equation*}
where we use the estimate $\eqref{pri:2.7}$ in the second step,
and the estimate $\eqref{pri:2.6}$ in the last one.

We now turn to estimate $\eqref{pri:1.2}$. Set
\begin{equation*}
\tilde{w}_\varepsilon = u_\varepsilon - u_0  -\varepsilon\chi_{k}(\cdot/\varepsilon)
S_\varepsilon(\psi_{2\varepsilon}\tilde{u}_0),
\end{equation*}
where $\tilde{u}_0$ is the extension of $u_0$ to $\mathbb{R}^d$.
Let $\Phi_\varepsilon\in L^q(\Omega;\mathbb{R}^d)$, and then there exist the weak solutions to
$(\textbf{DS}_\varepsilon)^*$ and $(\textbf{DS}_0)^*$, still denoted
by $(\phi_\varepsilon,\theta_0)$ and $(\phi_0,\theta_0)$, respectively.
On account of $\eqref{pri:4.2}$, we have
\begin{equation}\label{f:4.24}
\begin{aligned}
\bigg|\int_\Omega \tilde{w}_\varepsilon\Phi dx\bigg|
&\leq C\|\nabla u_0\|_{L^2(\Omega\setminus\Sigma_{4\varepsilon})}
\Big(\|\nabla \phi_\varepsilon\|_{L^2(\Omega\setminus\Sigma_{5\varepsilon})}
+\|\theta_\varepsilon\|_{L^2(\Omega\setminus\Sigma_{5\varepsilon})/\mathbb{R}}\Big) \\
& \qquad + C\varepsilon\|u_0\|_{H^2(\Omega)}
\Big(\|\nabla \phi_\varepsilon\|_{L^2(\Omega)}
+\|\theta_\varepsilon\|_{L^2(\Omega)/\mathbb{R}}\Big).
\end{aligned}
\end{equation}
To accelerate the convergence rate,
we substitute the terms $\phi_\varepsilon$ and $\theta_\varepsilon$
on $\Omega\setminus\Sigma_{5\varepsilon}$ with
\begin{equation*}
\xi_\varepsilon = \phi_\varepsilon -\phi_0 -\varepsilon\chi_{k}^*(\cdot/\varepsilon)
S_\varepsilon(\psi_{6\varepsilon}\nabla_k\phi_0),
\qquad
\vartheta_\varepsilon = \theta_\varepsilon - \theta_0 -
\pi_k^*(\cdot/\varepsilon)S_\varepsilon(\psi_{6\varepsilon}\nabla\phi_0).
\end{equation*}
Hence we derive
\begin{equation}\label{f:4.25}
\begin{aligned}
\|\nabla \phi_\varepsilon\|_{L^2(\Omega\setminus\Sigma_{5\varepsilon})}
+\|\theta_\varepsilon\|_{L^2(\Omega\setminus\Sigma_{5\varepsilon})/\mathbb{R}}
& \leq \|\nabla \xi_\varepsilon\|_{L^2(\Omega)}
+\|\vartheta_\varepsilon\|_{L^2(\Omega)/\mathbb{R}}
+\|\nabla \phi_0\|_{L^2(\Omega\setminus\Sigma_{5\varepsilon})}
+\|\theta_0\|_{L^2(\Omega\setminus\Sigma_{5\varepsilon})/\mathbb{R}} \\
& \leq C\varepsilon^{\frac{1}{2}}\|\Phi\|_{L^q(\Omega)}
+\|\nabla \phi_0\|_{L^2(\Omega\setminus\Sigma_{5\varepsilon})}
+\|\theta_0\|_{L^2(\Omega\setminus\Sigma_{5\varepsilon})} \\
&\leq C\varepsilon^{\frac{1}{2}}\|\Phi\|_{L^q(\Omega)},
\end{aligned}
\end{equation}
where we mention that
$S_\varepsilon(\psi_{6\varepsilon}\nabla u_0)$ is supported in $\Sigma_{5\varepsilon}$ and
$\|\theta_0\|_{L^2(\Omega\setminus\Sigma_{5\varepsilon})/\mathbb{R}}\leq
\|\theta_0\|_{L^2(\Omega\setminus\Sigma_{5\varepsilon})}$ by definition, and we employ
the estimates $\eqref{pri:3.9}$ and $\eqref{pri:3.11}$ in the second one,
and the estimate $\eqref{pri:3.6}$ in the last one.

By the estimate $\eqref{pri:2.6}$ and the fact that $L^{q}(\Omega)\subset H^{-1}(\Omega)$, we have
\begin{equation}\label{f:4.26}
\|\nabla \phi_\varepsilon\|_{L^2(\Omega)}
+\|\theta_\varepsilon\|_{L^2(\Omega)/\mathbb{R}}
\leq C\|\Phi\|_{L^q(\Omega)},
\end{equation}
where $C$ depends on $\mu,d$ and $\Omega$.

The last thing is to estimate $\|\nabla u_0\|_{L^2(\Omega\setminus\Sigma_{4\varepsilon})}$. On account of
the estimates $\eqref{pri:2.5}$ and $\eqref{pri:2.15}$, we acquire
\begin{equation}\label{f:4.27}
\|\nabla u_0\|_{L^2(\Omega\setminus\Sigma_{4\varepsilon})}
\leq C\varepsilon^{\frac{1}{2}}\|\mathcal{M}(\nabla u_0)\|_{L^2(\partial\Omega)}
\leq C\varepsilon^{\frac{1}{2}}\|u_0\|_{H^2(\Omega)}.
\end{equation}
Plugging the estimates $\eqref{f:4.25}$, $\eqref{f:4.26}$ and $\eqref{f:4.27}$ back into $\eqref{f:4.24}$
subsequently leads to the desired estimate $\eqref{pri:1.2}$, and we have completed the proof.
\qed

\section{$W^{1,p}$ estimates}

\begin{lemma}[Caccioppoli's inequality near boundary]\label{lemma:5.5}
Suppose that $A$ satisfies $\eqref{a:1}$. Let $(u_\varepsilon,p_\varepsilon)
\in H^1(D_{5r};\mathbb{R}^d)\times L^2(D_{5r})/\mathbb{R}$ be the weak solution of
$\mathcal{L}_\varepsilon(u_\varepsilon)+\nabla p_\varepsilon = 0$ and $\emph{div}(u_\varepsilon) = 0$ in
$D_{5r}$ and $u_\varepsilon = 0$ on $\Delta_{5r}$. Then we have
\begin{equation}\label{pri:5.10}
\Big(\dashint_{D_r}|\nabla u_\varepsilon|^2 dx\Big)^{\frac{1}{2}}
\leq \frac{C}{r}\Big(\dashint_{D_{2r}}|u_\varepsilon|^2 dx\Big)^{\frac{1}{2}}
\end{equation}
where $C$ depends on $\mu,d$ and $M$.
\end{lemma}

\begin{proof}
The proof is standard and may be found in \cite[pp.203]{MGMG}, and we provide a proof for the sake of completeness.
By dilation we may assume $r=1$. Let $\psi\in C^1(D_5)$ be a cut-off function such that
$\psi =1$ in $D_1$, $\psi=0$ outside $D_5\setminus D_2$ and $|\nabla\psi|\leq C$. Then let
$\psi^2u_\varepsilon^\alpha$ be a test function, and we obtain
\begin{equation*}
\int_{D_5}\psi^2 A(\cdot/\varepsilon)\nabla u_\varepsilon \nabla u_\varepsilon dx
+ 2\int_{D_5} \psi A(\cdot/\varepsilon)\nabla u_\varepsilon\nabla\psi u_\varepsilon dx
=2\int_{D_5} (p_\varepsilon -c)\psi\nabla_\alpha\psi  u_\varepsilon^\alpha dx
\end{equation*}
for any $c\in\mathbb{R}$. It follows from Young's inequality that
\begin{equation}\label{f:5.15}
\int_{D_1} |\nabla u_\varepsilon|^2 dx
\leq C\int_{D_2}|u_\varepsilon|^2 dx + \theta\int_{D_2}|p_\varepsilon -c|^2 dx
\end{equation}
Note that the second term in the right-hand side of $\eqref{f:5.15}$ is controlled by
$\int_{D_2}|\nabla u_\varepsilon|^2 dx$ on account of Lemma $\ref{lemma:2.1}$. Thus we arrive at
\begin{equation*}
\int_{D_1} |\nabla u_\varepsilon|^2 dx
\leq C\int_{D_2}|u_\varepsilon|^2 dx + \theta^\prime\int_{D_2}|\nabla u_\varepsilon|^2 dx
\end{equation*}
where $\theta^\prime = C\theta<1$ may be very small by choosing $\theta$.
This together with \cite[Lemma 0.5]{MGMG} gives
$\int_{D_1} |\nabla u_\varepsilon|^2 dx\leq C\int_{D_2}|u_\varepsilon|^2 dx$. Then by rescaling arguments
it is not hard to see the estimate $\eqref{pri:5.10}$, and we have completed the proof.
\end{proof}

\begin{remark}\label{re:5.1}
\emph{For the Stokes systems with constant coefficients, there holds the same type Caccioppoli's inequality
near boundary as in Lemma $\ref{lemma:5.5}$ (see for example \cite{MGMG}).}
\end{remark}

\subsection{$W^{1,p}$ estimates uniformly down to the scale $\varepsilon$}

\begin{lemma}\label{lemma:5.1}
Let $(u_0,p_0)\in H^1(D_4;\mathbb{R}^d)\times L^2(D_4)/\mathbb{R}$ be the weak solution to
$\mathcal{L}_0(u_0) + \nabla p_0 = 0$ in $D_4$, $\emph{div}(u_0) = 0$ in $D_4$ and $u_0 = 0$ on $\Delta_4$.
Given $p=2d/(d-1)$,
then for any $0<t<1$ we have
\begin{equation}\label{pri:5.1}
\int_0^t\int_{|x^\prime|<2}|u_0(x^\prime,\psi(x^\prime)+s)|^{p}dx^\prime ds
\leq Ct^{p+\tau}\int_0^{4m_0}\int_{|x^\prime|<4}|u_0(x^\prime,\psi(x^\prime)+s)|^p dx^\prime ds,
\end{equation}
where $\tau>0$, and $C$ depends on $\mu,d$ and $M$.
\end{lemma}

\begin{proof}
The ideas of the proof is quite similar to that in \cite[Lemma 5.3]{JGZSLS}, and we provide the proof
for the sake of the completeness. Due to $u_0 = 0$ on $\Delta_4$, it follows from the
Poincar\'e's inequality and H\"older's inequality
that
\begin{equation}\label{f:5.1}
\begin{aligned}
\int_0^t\int_{|x^\prime|<2}|u_0(x^\prime,\psi(x^\prime)+s)|^p dx^\prime ds
&\leq Ct^p\int_0^t\int_{|x^\prime|<2}|\nabla u_0(x^\prime,\psi(x^\prime)+s)|^p dx^\prime ds\\
&\leq Ct^{p+\frac{\epsilon}{p+\epsilon}}\bigg(\int_0^t\int_{|x^\prime|<2}|\nabla u_0(x^\prime,\psi(x^\prime)+s)|^{p+\epsilon}
dx^\prime ds\bigg)^{\frac{p}{p+\epsilon}}\\
&\leq Ct^{p+\tau}\bigg(\int_{D_2}|\nabla u_0|^{p+\epsilon}
dx\bigg)^{\frac{p}{p+\epsilon}},
\end{aligned}
\end{equation}
where $\tau = \epsilon/(p+\epsilon)$.
It is enough to estimate the quantity
$\|\nabla u_0\|_{L^p(D_2)}$
since the case $p+\epsilon$ follows from the self-improvement
property of the weak reverse H\"older's inequality (see \cite[Theorem 6.38]{MGLM}).
Hence it follows from the Sobolev imbedding theorem that
$\|\nabla u_0\|_{L^p(D_r)}
\leq C\|\nabla u_0\|_{W^{\frac{1}{2},2}(D_{r})}
\leq C\|\nabla u_0\|_{H^1(\partial D_r)}$ for any $r\in[2,5/2]$, where we use
\cite[Theorem 2.2]{RMBZS} in the last inequality. Due to $u_0=0$ on $\Delta_r$,
we have
\begin{equation*}
\Big(\int_{D_2}|\nabla u_0|^p dx\Big)^{\frac{2}{p}}
\leq C\int_{\partial D_{r}\setminus\Delta_r}( |\nabla u_0|^2 + |u_0|^2 ) dS.
\end{equation*}
Integrating both sides of the above inequality with respect to $r$ over [2,5/2], we obtain
\begin{equation*}
\Big(\int_{D_2}|\nabla u_0|^p dx\Big)^{\frac{2}{p}}
\leq C\int_{D_{\frac{5}{2}}}
( |\nabla u_0|^2 + |u_0|^2 ) dx
\leq C\int_{D_{\frac{5}{2}}}|\nabla u_0|^2dx,
\end{equation*}
where we use Poincar\'e's inequality in the last step. This together with $\eqref{f:5.1}$ gives
\begin{equation*}
\begin{aligned}
\int_0^t\int_{|x^\prime|<2}|u_0(x^\prime,\psi(x^\prime)+s)|^p dx^\prime ds
&\leq Ct^{p+\tau}\bigg(\int_0^{\frac{5}{2}m_0}\int_{|x^\prime|<\frac{5}{2}}
|\nabla u_0(x^\prime,\psi(x^\prime)+s)|^2 dx^\prime ds\bigg)^{\frac{p}{2}}\\
&\leq Ct^{p+\tau}\int_0^{4m_0}\int_{|x^\prime|<4}
|u_0(x^\prime,\psi(x^\prime)+s)|^p dx^\prime ds,
\end{aligned}
\end{equation*}
where we employ Caccioppoli's inequality $\eqref{pri:5.10}$ (see Remark $\ref{re:5.1}$) and H\"older's inequality in the last step.
We have completed the proof.
\end{proof}

The following lemma is the key ingredient in the whole proof of Theorem $\ref{thm:1.2}$, and
its proof is based on the convergence rate $\eqref{pri:3.10}$.
We mention that there is a new argument, originally motivated by \cite{SZ,SACS,SZW12}.

\begin{lemma}\label{lemma:5.2}
Let $\tau>0$ be given in Lemma $\ref{lemma:5.1}$.
There exists a positive constant $C$, depending only on $\mu,d$ and $M$, such that for any
$\varepsilon<r\leq 1$,
\begin{equation}\label{pri:5.5}
\int_0^r\int_{|x^\prime|<1}|u_\varepsilon(x^\prime,\psi(x^\prime)+t)|^p dx^\prime dt
\leq Cr^{p+\tau}\int_0^{4}\int_{|x^\prime|<4}|u_\varepsilon(x^\prime,\psi(x^\prime+t))|^p dx^\prime dt
\end{equation}
where $p=2d/(d-1)$, and
$(u_\varepsilon,p_\varepsilon)\in H^1(D_5;\mathbb{R}^d)\times L^2(D_5)/\mathbb{R}$ satisfies
$\mathcal{L}_\varepsilon(u_\varepsilon) +\nabla p_\varepsilon = 0$,
$\emph{div}(u_\varepsilon) = 0$ in $D_5$, and $u_\varepsilon = 0$ on $\Delta_5$.
\end{lemma}

\begin{proof}
By setting
$\tilde{u}_\varepsilon = u_\varepsilon/\|u_\varepsilon\|_{L^2(D_8)}$ and
$\tilde{p}_\varepsilon = p_\varepsilon/\|p_\varepsilon\|_{L^2(D_8)}$
we may assume $\|u_\varepsilon\|_{L^2(D_5)} = \|p_\varepsilon\|_{L^2(D_5)} = 1$, and then it follows
from Caccioppoli's inequality $\eqref{pri:5.10}$
that $\|\nabla u_\varepsilon\|_{L^2(D_4)}\leq C$. Hence
due to the homogenization theory we arrive at $u_\varepsilon \rightharpoonup v$ weakly in
$H^1(D_4;\mathbb{R}^d)$ and
strongly in $L^p(D_4;\mathbb{R}^d)$ with $p=2d/(d-1)$, and $p_\varepsilon\rightharpoonup p_0$ weakly in $L^2(D_5)$, where
$(v,p_0)$ satisfies homogenized equations: $\mathcal{L}_0(v)+\nabla p_0 = 0$ in $D_4$, $\text{div}(v) = 0$
in $D_4$ with $v=0$ on $\Delta_4$. Thus a direct result is
\begin{equation}\label{f:5.3}
 \|v\|_{L^p(D_4)}\leq C\|u_\varepsilon\|_{L^p(D_4)}.
\end{equation}
Moreover, from Caccioppoli's inequality $\eqref{pri:5.10}$ (see Remark $\ref{re:5.1}$) we also have
\begin{equation*}
 \Big(\dashint_{D_{\frac{3}{2}}}|\nabla v|^2 dx\Big)^{\frac{1}{2}}
 \leq C \Big(\dashint_{D_2} |v|^2 dx\Big)^{\frac{1}{2}},
\end{equation*}
and then in view of co-area formula it is not hard to see that there exists $s\in[1,3/2]$ such that
\begin{equation}\label{f:5.4}
\|\nabla v\|_{L^2(\partial D_s/\Delta_4)} + \|v\|_{L^2(\partial D_s/\Delta_4)}
\leq C\|v\|_{L^2(D_2)}.
\end{equation}
We now let $v=u_\varepsilon$ on $\partial D_s$, and then
\begin{equation*}
\|u_\varepsilon - v\|_{L^p(D_1)}
\leq \|u_\varepsilon - v\|_{L^p(D_s)} \leq C\varepsilon^{\frac{1}{2}}\|v\|_{H^1(\partial D_s)}
\leq C\varepsilon^{\frac{1}{2}}\|v\|_{L^2(D_2)}
\leq C\varepsilon^{\frac{1}{2}}\|v\|_{L^p(D_2)},
\end{equation*}
where we use the estimate $\eqref{pri:3.10}$ in the second inequality, and $\eqref{f:5.4}$ in the third one.
The last step above is due to H\"older's inequality.
By scaling we may have
\begin{equation}
\Big(\dashint_{D_r} |u_\varepsilon - v|^p dx\Big)^{\frac{1}{p}}
\leq C\left(\frac{\varepsilon}{r}\right)^{\frac{1}{2}}\Big(\dashint_{D_{2r}}|v|^p dx\Big)^{\frac{1}{p}}
\end{equation}
for any $\varepsilon<r\leq 1$, and this together with a covering technique leads to
\begin{equation}\label{f:5.2}
\int_0^r\int_{|x^\prime|<1} |u_\varepsilon(x^\prime,\psi(x^\prime)+t)
- v(x^\prime,\psi(x^\prime)+t)|^p dx^\prime dt
\leq C\left(\frac{\varepsilon}{r}\right)^{\frac{p}{2}}\int_{0}^{2m_0r}\int_{|x^\prime|<2}
|v(x^\prime,\psi(x^\prime)+t)|^p dx^\prime dt.
\end{equation}
Hence we have
\begin{equation*}
\begin{aligned}
\int_0^r\int_{|x^\prime|<1}|u_\varepsilon(x^\prime,\psi(x^\prime)+t)|^p dx^\prime dt
&\leq \int_0^r\int_{|x^\prime|<1} |u_\varepsilon
- v|^p dx^\prime dt
+ \int_0^r\int_{|x^\prime|<1}|v|^p dx^\prime dt\\
&\leq C \int_{0}^{2m_0r}\int_{|x^\prime|<2}
|v(x^\prime,\psi(x^\prime)+t)|^p dx^\prime dt\\
&\leq Cr^{p+\tau}\int_0^{4m_0}\int_{|x^\prime|<4}|v|^pdx^\prime dt
\leq Cr^{p+\tau}\int_0^{4m_0}\int_{|x^\prime|<4}|u_\varepsilon|^pdx^\prime dt,
\end{aligned}
\end{equation*}
where we employ the estimate $\eqref{f:5.2}$ in the second step, and $\eqref{pri:5.1}$ in the third one.
We mention that the last step is due to the estimate $\eqref{f:5.3}$, and the proof is complete.
\end{proof}

The following theorem should be regarded as a $W^{1,p}$ estimate for $u_\varepsilon$ uniformly down to
the microscopic scale $\varepsilon$.

\begin{thm}\label{thm:5.3}
Assume that $A$ satisfies $\eqref{a:1}-\eqref{a:3}$.
Let $(u_\varepsilon,p_\varepsilon)\in H^1(D_5;\mathbb{R}^d)\times L^2(D_5)/\mathbb{R}$ be the weak solution
to $\mathcal{L}_\varepsilon(u_\varepsilon)+\nabla p_\varepsilon = 0$,
$\emph{div}(u_\varepsilon) = 0$ in $D_5$ and
$u_\varepsilon = 0$ on $\Delta_5$. Then we have
\begin{equation}\label{pri:5.6}
\bigg(\int_\varepsilon^{m_0}\int_{|x^\prime|<1}\left|\frac{u_\varepsilon(x^\prime,\psi(x^\prime)+t)}{t}\right|^p
dx^\prime dt\bigg)^{1/p}
\leq C \bigg(\int_0^{4m_0}\int_{|x^\prime|<4}\big|\nabla u_\varepsilon\big|^2 dx^\prime dt\bigg)^{1/2},
\end{equation}
where $p=2d/(d-1)$, and $C$ depends only on $\mu,d$ and $\Omega$.
\end{thm}

\begin{proof}
We first choose a positive integer $k_0$ such that $2^{k_0}\varepsilon < 1/2 < 2^{k_0+1}\varepsilon$, and then
\begin{equation}\label{f:5.5}
\begin{aligned}
\int_\varepsilon^{m_0}\int_{|x^\prime|<1}
\left|\frac{u_\varepsilon(x^\prime,\psi(x^\prime)+t)}{t}\right|^pdx^\prime dt
&\leq \bigg\{\sum_{k=0}^{k_0}\int_{2^k\varepsilon}^{2^{k+1}\varepsilon}\int_{|x^\prime|<1}
+\int_{\frac{1}{2}}^{m_0}\int_{|x^\prime|<1}\bigg\}
\left|\frac{u_\varepsilon(x^\prime,\psi(x^\prime)+t)}{t}\right|^p dx^\prime dt \\
& =: I_1 + I_2.
\end{aligned}
\end{equation}
It is clear to see that
\begin{equation}\label{f:5.6}
 I_2= \int_{\frac{1}{2}}^{m_0}\int_{|x^\prime|<1}
 \left|\frac{u_\varepsilon(x^\prime,\psi(x^\prime)+t)}{t}\right|^p dx^\prime dt
  \leq C\int_0^{m_0}\int_{|x^\prime|<1} |u_\varepsilon(x^\prime,\psi(x^\prime)+t)|^pdx^\prime dt.
\end{equation}
We proceed to handle $I_1$ by applying Lemma $\ref{lemma:5.2}$, and obtain
\begin{equation}\label{f:5.7}
\begin{aligned}
I_1 & = \sum_{k=0}^{k_0}\int_{2^k\varepsilon}^{2^{k+1}\varepsilon}\int_{|x^\prime|<1}
\left|\frac{u_\varepsilon(x^\prime,\psi(x^\prime)+t)}{t}\right|^p dx^\prime dt \\
&\leq C\sum_{k=0}^{k_0}(2^k\varepsilon)^{-p}\cdot(2^{k+1}\varepsilon)^{p+\tau}
\int_0^{4m_0}\int_{|x^\prime|<4}|u_\varepsilon(x^\prime,\psi(x^\prime)+t)|^p dx^\prime dt
\leq C\int_{D_4} |u_\varepsilon|^p dx,
\end{aligned}
\end{equation}
where we use the fact that $2^{k_0}\varepsilon<1/2<2^{k_0+1}\varepsilon$ in the last step.
Plugging the estimates $\eqref{f:5.6}$ and $\eqref{f:5.7}$ back into $\eqref{f:5.5}$ we have
\begin{equation*}
\int_\varepsilon^{m_0}\int_{|x^\prime|<1}
\left|\frac{u_\varepsilon(x^\prime,\psi(x^\prime)+t)}{t}\right|^pdx^\prime dt
\leq \int_{D_4}|u_\varepsilon|^p dx \leq C\bigg(\int_{D_4} |\nabla u_\varepsilon|^2 dx\bigg)^{p/2},
\end{equation*}
where we employ the Sobolev-Poincar\'e inequality in the last inequality, and we have completed the proof.
\end{proof}

\subsection{Operators with $\text{VMO}(\mathbb{R}^d)$ coefficients}

\begin{thm}\label{thm:5.1}
Let $\mathcal{L} = -\emph{div}(A(x)\nabla)$
with $A\in \emph{VMO}(\mathbb{R}^d)$ satisfying $\eqref{a:1}$ and $\eqref{a:3}$. Suppose that
$(u,P)\in H^1(D_{4r};\mathbb{R}^d)\times L^2(D_{4r})/\mathbb{R}$ is the weak solution to $\mathcal{L}(u)
+ \nabla P = 0$ and $\emph{div}(u) = 0$ in $D_{4r}$ with $u=0$ on $\Delta_{4r}$. Then for
$2\leq p < \frac{2d}{d-1}+\epsilon$ we have
\begin{equation}\label{pri:5.4}
 \Big(\dashint_{D_r}|\nabla u|^pdx\Big)^{1/p}
 \leq C_p \Big(\dashint_{D_{2r}}|\nabla u|^2dx\Big)^{1/2},
\end{equation}
where $C_p$ depends on $\mu,\omega,d,p$ and $M$.
\end{thm}

\begin{proof}
The result directly follows from Lemma $\ref{lemma:5.3}$ and \cite[Theorem 3.5]{JGZSLS}, and we are done.
\end{proof}
\begin{remark}
\emph{Theorem $\ref{thm:5.1}$ has been shown in \cite{SGZWS} without a proof
in the case of $\partial\Omega\in C^1$ for $2\leq p<\infty$. We also mention that
the approximation argument employed here is quite similar to
that shown in \cite{JGZSLS} and originally investigated in \cite{LI}.}
\end{remark}

\begin{lemma}\label{lemma:5.3}
Assume the coefficient of $\mathcal{L}$ satisfies the same conditions as in Theorem $\ref{thm:5.1}$. Then
there exist a function $h(r)$ and some constants $C_0>0$ and $p>\frac{2d}{d-1}$ with the following properties:
\begin{itemize}
  \item $\lim_{r\to 0} h(r) = 0$;
  \item if $(u,P)\in H^1(D_{3r};\mathbb{R}^d)\times L^2(D_{3r})/\mathbb{R}$
  is the weak solution to $\mathcal{L}(u) +\nabla P = 0$ and $\emph{div}(u)=0$ in $D_{3r}$
  with $u = 0$ on $\Delta_{3r}$.
  Then there exists $v\in W^{1,p}(D_{2r};\mathbb{R}^d)$ such that
  \begin{equation}\label{pri:5.2}
  \Big(\dashint_{D_{2r}}\big|\nabla (u-v)\big|^2 dx\Big)^{1/2}
  \leq h(r)\Big(\dashint_{D_{3r}} |\nabla u|^2 dx\Big)^{1/2},
  \end{equation}
  \vspace{-0.5cm}
  \begin{equation}\label{pri:5.3}
  \Big(\dashint_{D_r}|\nabla v|^p dx\Big)^{1/p}
  \leq C_0\Big(\dashint_{D_{3r}} |\nabla u|^2 dx\Big)^{1/2}.
  \end{equation}
\end{itemize}
\end{lemma}

\begin{proof}
The proof is similar to that in\cite[Lemma 3.4]{JGZSLS} and we provide it for the sake of the completeness.
First of all, we denote an operator with constant coefficients by
\begin{equation*}
 \bar{\mathcal{L}} = - \text{div}(\bar{A}\nabla)=-\frac{\partial}{\partial x_i}
 \Big\{\bar{a}_{ij}^{\alpha\beta}\frac{\partial}{\partial x_j}\Big\}
 \quad\text{and}\quad
 \bar{a}_{ij}^{\alpha\beta} = \dashint_{B(x,r)} a_{ij}^{\alpha\beta}(x) dx,
 \quad \text{with}~B(x,r)\cap\Omega = D_r,
\end{equation*}
and let $(v,q)\in H^1(D_{2r};\mathbb{R}^d)\times L^2(D_{2r})/\mathbb{R}$ be the weak solution to
$\bar{\mathcal{L}}(v) + \nabla q = 0$  and $\text{div}(v) = 0$ in $D_{2r}$ with
$v= u$ on $\partial D_{2r}$. Then it is clear to see that
\begin{equation}\label{pde:5.1}
\left\{\begin{aligned}
\bar{\mathcal{L}}(v-u) + \nabla (q-P) & = \text{div}\big[(\tilde{A}-A)\nabla u\big]
&\quad& \text{in}~~D_{2r},\\
\text{div}(v-u) & = 0 &\quad& \text{in}~~D_{2r},\\
 v- u &= 0 &\quad& \text{on}~\partial D_{2r},
\end{aligned}\right.
\end{equation}
where we use the fact that $\mathcal{L}(u)+\nabla P = 0$ in $D_{2r}$ in the first line of $\eqref{pde:5.1}$.
By taking $v-u$ as the test function, we have
\begin{equation*}
\int_{D_{2r}}\bar{A}\nabla(v-u)\cdot\nabla(v-u) dx
= \int_{D_{2r}} (A-\bar{A})\nabla u\cdot \nabla(v-u)dx.
\end{equation*}
On account of the ellipticity condition and Young's inequality, it follows that
\begin{equation*}
\begin{aligned}
\dashint_{D_{2r}} |\nabla (v-u)|^2 dx &\leq C\dashint_{D_{2r}}|A-\bar{A}|^2|\nabla u|^2 dx\\
&\leq C\Big(\dashint_{D_{2r}}|A-\bar{A}|^{2s^\prime} dx\Big)^{\frac{1}{s^\prime}}
\Big(\dashint_{D_{2r}}|\nabla u|^{2s} dx\Big)^{\frac{1}{s}}
\leq h^2(r)\dashint_{D_{3r}}|\nabla u|^{2} dx,
\end{aligned}
\end{equation*}
where $s>1$ is properly chosen and $s^\prime$ is the conjugate index of $s$. In the last step,
we employ the weak reverse H\"older's inequality
\begin{equation}
\Big(\dashint_{D_{2r}}|\nabla u|^{2s}dx\Big)^{\frac{1}{2s}}
\leq C\Big(\dashint_{D_{3r}}|\nabla u|^2 dx\Big)^{\frac{1}{2}},
\end{equation}
and define
\begin{equation}
h(r) = C\sup_{x\in\Omega}\sum_{i,j,\alpha,\beta}\Big(\dashint_{B(x,r)}
|a_{ij}^{\alpha\beta}-\bar{a}_{ij}^{\alpha\beta}|^{2s^\prime} dy\Big)^{\frac{1}{2s^\prime}}.
\end{equation}
Since $a_{ij}^{\alpha\beta}\in \text{VMO}$,
by the John-Nirenberg inequality (see \cite[Corollary 6.12]{JD}), it is not hard to see
$h(r)\to 0$ as $r\to 0$.

Then we turn to prove the estimate $\eqref{pri:5.3}$. Noting that $(v,q)$ satisfies
$\bar{\mathcal{L}}(v) + \nabla q = 0$ and $\text{div}(v) = 0$ in $D_{2r}$ with
$v=0$ on $\Delta_{2r}$. It follows from \cite[Theorem 1.3]{JGJK} that
\begin{equation*}
\Big(\dashint_{D_r}|\nabla v|^p dx\Big)^{\frac{1}{p}}
\leq C \Big(\dashint_{D_{2r}}|\nabla v|^2 dx\Big)^{\frac{1}{2}}
\leq C \Big(\dashint_{D_{3r}}|\nabla u|^2 dx\Big)^{\frac{1}{2}}
\end{equation*}
by choosing $p=2d/(d-1)+\epsilon$,
where we employ the estimate $\eqref{pri:5.2}$ in the last step. We have completed the proof.
\end{proof}

\begin{thm}\label{thm:5.2}
Suppose that $A\in\emph{VMO}(\mathbb{R}^d)$ satisfies $\eqref{a:1}-\eqref{a:3}$. Let
$(u_\varepsilon,p_\varepsilon)\in H^1(D_{5r};\mathbb{R}^d)\times L^2(D_{5r})$ be the weak solution to
$\mathcal{L}_\varepsilon(u_\varepsilon)+\nabla p_\varepsilon = 0$ and $\emph{div}(u_\varepsilon) = 0$ in
$D_{5r}$ with $u_\varepsilon = 0$ on $\Delta_{5r}$. Then we have
\begin{equation}\label{pri:5.7}
 \Big(\dashint_{D_r}|\nabla u_\varepsilon|^p dx\Big)^{1/p}
 \leq C\Big(\dashint_{D_{4r}}|\nabla u_\varepsilon|^2 dx\Big)^{1/2}
\end{equation}
for $p=2d/(d-1)+\epsilon$, where $C$ depends only on $\mu,\omega,d$ and $M$.
\end{thm}

To prove Theorem $\ref{thm:5.2}$, we need to introduce the following lemma
\begin{lemma}\label{lemma:5.4}
Under the same conditions as in Theorem $\ref{thm:5.2}$, suppose that
$(u_\varepsilon,p_\varepsilon)$ is the weak solution to
$\mathcal{L}_\varepsilon(u_\varepsilon)+\nabla p_\varepsilon = 0$ and $\emph{div}(u_\varepsilon) = 0$ in
$D_{3r}$ with $u_\varepsilon = 0$ on $\Delta_{3r}$. Then for any $2\leq p<\infty$ we have
\begin{equation}
\int_{0}^{rm_0}\int_{|x^\prime|<r}|\nabla u_\varepsilon(x^\prime,\psi(x^\prime)+t)|^p dx^\prime dt
\leq C_p\int_0^{2rm_0}\int_{|x^\prime|<2r}
\left|\frac{u_\varepsilon(x^\prime,\psi(x^\prime)+t)}{t}\right|^p dx^\prime dt,
\end{equation}
where $C_p$ depends on $\mu,\omega,p,d$ and $M$.
\end{lemma}

\begin{proof}
The desired result follows from \cite[Lemma 7.3]{SGZWS}
and the arguments developed in \cite{SZW10}.
The core idea is not connected to homogenization topics and we thus omit the proof.
\end{proof}

\begin{flushleft}
\textbf{Proof of Theorem \ref{thm:5.2}.} By rescaling we may assume that $r=1$,
and by the self-improvement property we only prove this theorem for
$p=2d/(p-1)$.
On account of Lemma $\ref{lemma:5.4}$, it is equivalent to estimating the quantity
\end{flushleft}
\begin{equation}\label{f:5.8}
\int_0^{m_0}\int_{|x^\prime|<1}\left|\frac{u_\varepsilon(x^\prime,\psi(x^\prime)+t)}{t}\right|^p dx^\prime dt
= \bigg\{\int_0^{\varepsilon}\int_{|x^\prime|<1}+\int_{\varepsilon}^{m_0}\int_{|x^\prime|<1}\bigg\}
\left|\frac{u_\varepsilon(x^\prime,\psi(x^\prime)+t)}{t}\right|^p dx^\prime dt.
\end{equation}
Note that the second term in the right-hand side of $\eqref{f:5.8}$ immediately
follows from Theorem $\ref{thm:5.3}$.
In fact we only need to handle the estimate at microcosmic scale $\varepsilon$,
where the smoothness of the coefficient of $\mathcal{L}_\varepsilon$ comes into play.
To estimate the first term in the right-hand side of $\eqref{f:5.8}$, we first obtain
\begin{equation*}
\begin{aligned}
\int_0^{\varepsilon}\int_{|x^\prime|<\varepsilon}\left|\frac{u_\varepsilon(x^\prime,\psi(x^\prime)+t)}{t}\right|^p dx^\prime dt
&\leq C \int_0^{\varepsilon}\int_{|x^\prime|<\varepsilon}
\left|\nabla u_\varepsilon(x^\prime,\psi(x^\prime)+t)\right|^p dx^\prime dt\\
&\leq C\varepsilon^{d-\frac{pd}{2}}\bigg(\int_{D_{2\varepsilon}}
|\nabla u_\varepsilon|^2 dx\bigg)^{\frac{p}{2}}
\leq C\varepsilon^{d-\frac{pd}{2}-p}\bigg(\int_{D_{3\varepsilon}}
|u_\varepsilon|^2 dx\bigg)^{\frac{p}{2}}\\
&\leq C\varepsilon^{-p}\int_0^{3m_0\varepsilon}\int_{|x^\prime|<2}
|u_\varepsilon|^p dx^\prime dt,
\end{aligned}
\end{equation*}
where we employ Hardy's inequality in the first step, and the estimate $\eqref{pri:5.4}$ in the
second one, and Caccioppoli's inequality $\eqref{pri:5.10}$ in the third one, and H\"older's inequality
in the last one. The above estimate coupled with a covering argument gives
\begin{equation}\label{f:5.9}
\begin{aligned}
\int_0^{\varepsilon}\int_{|x^\prime|<1}
\left|\frac{u_\varepsilon(x^\prime,\psi(x^\prime)+t)}{t}\right|^p dx^\prime dt
&\leq C\varepsilon^{-p}\int_0^{3m_0\varepsilon}\int_{|x^\prime|<2}
|u_\varepsilon(x^\prime,\psi(x^\prime)+t)|^p dx^\prime dt \\
&\leq C\int_0^{4m_0}\int_{|x^\prime|<4}|u_\varepsilon(x^\prime,\psi(x^\prime)+t)|^p dx^\prime dt
\leq C\bigg(\int_{D_4}|\nabla u_\varepsilon|^2 dx\bigg)^{p/2},
\end{aligned}
\end{equation}
where we use the estimate $\eqref{pri:5.5}$ in the second inequality, and the Sobolev-Poincar\'e inequality
in the last one. Collecting $\eqref{f:5.8}$, $\eqref{f:5.9}$ and $\eqref{pri:5.6}$ leads to
the desired estimate $\eqref{pri:5.7}$, and we have completed the proof.
\qed

Due to the real methods originally developed by Z.Shen in
\cite{SZW15}, from Theorem $\ref{thm:5.2}$ we have the following theorem.

\begin{thm}\label{thm:5.4}
Assume the same conditions as in Theorem $\ref{thm:5.2}$. For any $f\in L^p(\Omega;\mathbb{R}^{d\times d})$
with $\frac{2d}{d+1}-\epsilon< p<\frac{2d}{d-1}+\epsilon$, there exists a unique solution
$(u_\varepsilon,p_\varepsilon)\in W_0^{1,p}(\Omega;\mathbb{R}^d)\times L^p(\Omega)/\mathbb{R}$ satisfying
\begin{equation}\label{pde:5.2}
\left\{\begin{aligned}
\mathcal{L}_\varepsilon(u_\varepsilon)+\nabla p_\varepsilon &= \emph{div}(f) &\quad&\emph{in}~~\Omega,\\
\emph{div}(u_\varepsilon) &= 0 &\quad&\emph{in}~~\Omega,\\
u_\varepsilon &=0 &\quad& \emph{on}~\partial\Omega.
\end{aligned}\right.
\end{equation}
Moreover, the solution satisfies the uniform estimate
\begin{equation}\label{pri:5.8}
 \|\nabla u_\varepsilon\|_{L^p(\Omega)} + \|p_\varepsilon\|_{L^p(\Omega)/\mathbb{R}}
 \leq C\|f\|_{L^{p}(\Omega)},
\end{equation}
where $C$ depends only on $\mu,\omega,d$ and $\Omega$.
\end{thm}

\begin{proof}
The proof is standard, and we provide a proof for the sake of the completeness.
Let $B=B(x_0,r)$ and $nB=B(x_0,nr)$ with $n\in\mathbb{R}_+$.
In the case of $p>2$, the existence of the solution comes down to the case $p=2$, and we focus on the
estimate $\eqref{pri:5.8}$. To do so, we split the source term $f$ up into $\varphi f$ and $(1-\varphi)f$,
where $\varphi\in C_0^\infty(6B)$ is a cut-off function such that $\varphi = 1$ in $4B$ and
$\varphi = 0$ outside $5B$, and then we construct the following auxiliary equations
\begin{equation*}
(\text{i})\left\{\begin{aligned}
\mathcal{L}_\varepsilon(v_\varepsilon) + \nabla q_\varepsilon &= \text{div}(\varphi f) &\quad&\text{in}~~\Omega,\\
\text{div}(v_\varepsilon) &= 0 &\quad&\text{in}~~\Omega,\\
v_\varepsilon & = 0 &\quad&\text{on}~\partial\Omega,
\end{aligned}\right.
\qquad\quad
(\text{ii})\left\{\begin{aligned}
\mathcal{L}_\varepsilon(w_\varepsilon) + \nabla r_\varepsilon &= \text{div}\big[(1-\varphi) f\big] &\quad&\text{in}~~\Omega,\\
\text{div}(w_\varepsilon) &= 0 &\quad&\text{in}~~\Omega,\\
w_\varepsilon & = 0 &\quad&\text{on}~\partial\Omega.
\end{aligned}\right.
\end{equation*}
It is well known that $(v_\varepsilon,q_\varepsilon)$ and $(w_\varepsilon,r_\varepsilon)$ belongs to
$H_0^1(\Omega;\mathbb{R}^d)\times L^2(\Omega)/\mathbb{R}$, and it is not hard to see that
$u_\varepsilon = v_\varepsilon + w_\varepsilon$ and $p_\varepsilon = q_\varepsilon+r_\varepsilon$.
We denote $F = |\nabla u_\varepsilon|$, $F_B=|\nabla v_\varepsilon|$, $R_B= |\nabla w_\varepsilon|$ and
$g=|f|$. Hence from (i) we have
\begin{equation}\label{f:5.10}
\begin{aligned}
\dashint_{2B\cap\Omega}|F_B|^2 dx  = \dashint_{2B\cap\Omega}|\nabla v_\varepsilon|^2 dx
&\leq \frac{1}{|2B\cap\Omega|}\int_{\Omega}|\nabla v_\varepsilon|^2 dx\\
&\leq \frac{1}{|2B\cap\Omega|}\int_{\Omega}|\varphi f|^2 dx
\leq C\dashint_{4B}|f|^2 dx
= C\dashint_{4B}g^2 dx,
\end{aligned}
\end{equation}
where we use the estimate $\eqref{pri:2.6}$ in the second inequality.
Since $(1-\varphi)f = 0$ in $4B\cap\Omega$, we have
$\mathcal{L}_\varepsilon(w_\varepsilon)+\nabla r_\varepsilon = 0$ in $4B\cap\Omega$.
Combining the conditions $\text{div}(w_\varepsilon) = 0$ in $4B\cap\Omega$ and
$w_\varepsilon = 0$ on $4B\cap\partial\Omega$, it follows from Theorem $\ref{thm:5.2}$ that
\begin{equation}\label{f:5.11}
\begin{aligned}
\Big(\dashint_{B\cap\Omega}|R_B|^q dx\Big)^{\frac{1}{p}}
= \Big(\dashint_{B\cap\Omega}|\nabla w_\varepsilon|^q dx\Big)^{\frac{1}{p}}
&\leq C\Big(\dashint_{2B\cap\Omega}|\nabla w_\varepsilon|^2 dx\Big)^{\frac{1}{2}}\\
&\leq C\bigg\{\Big(\dashint_{2B\cap\Omega}|\nabla u_\varepsilon|^2 dx\Big)^{\frac{1}{2}}
+\Big(\dashint_{4B\cap\Omega}|\nabla v_\varepsilon|^2 dx\Big)^{\frac{1}{2}}\bigg\}\\
&\leq C\bigg\{\Big(\dashint_{2B\cap\Omega}|F|^2 dx\Big)^{\frac{1}{2}}
+\Big(\dashint_{4B\cap\Omega} g^2 dx\Big)^{\frac{1}{2}}\bigg\},
\end{aligned}
\end{equation}
where $q=2d/(d-1)+\epsilon$, and we employ the estimate $\eqref{f:5.10}$ in the last inequality.
Until now two conditions of \cite[Theorem 6.2]{JGJK} have already been satisfied by the estimates
$\eqref{f:5.10}$ and $\eqref{f:5.11}$, and then
for $2<p<q$ we obtain
\begin{equation}\label{f:5.12}
 \|\nabla u_\varepsilon\|_{L^p(\Omega)} \leq C\|f\|_{L^p(\Omega)}.
\end{equation}

We mention that the case $\frac{2d}{d+1}-\epsilon<p<2$ directly follows from the duality argument, and then
we handle the pressure term. By observing $\nabla p_\varepsilon =
\text{div}(A(\cdot/\varepsilon)\nabla u_\varepsilon + f)=:\tilde{F}$, it follows from Lemma $\ref{lemma:2.1}$ that
\begin{equation*}
\|p_\varepsilon\|_{L^p(\Omega)/\mathbb{R}}
\leq C\|\tilde{F}\|_{W^{-1,p}(\Omega)} \leq C\Big\{\|\nabla u_\varepsilon\|_{L^p(\Omega)}
+ \|f\|_{L^p(\Omega)}\Big\}\leq C\|f\|_{L^p(\Omega)},
\end{equation*}
where we use the estimate $\eqref{f:5.12}$ for $p\in(\frac{2d}{d+1}-\epsilon,\frac{2d}{d-1}+\epsilon)$.
We have completed the proof.
\end{proof}

\begin{remark}\label{remark:5.1}
\emph{Assume the same conditions as in Theorem $\ref{thm:5.4}$, and we replace the source term $\text{div}(f)$ in
the right-hand side of
$\eqref{pde:5.2}$ into $f_0\in L^p(\Omega;\mathbb{R}^d)$. Then by the duality argument we have
the uniform estimate
\begin{equation}\label{pri:5.9}
 \|\nabla u_\varepsilon\|_{L^p(\Omega)} + \|p_\varepsilon\|_{L^p(\Omega)/\mathbb{R}}
 \leq C\|f_0\|_{L^p(\Omega)}
\end{equation}
for the same range of $p$ as in Theorem $\ref{thm:5.4}$. It is not a sharp estimate but is sufficient
for us to establish the same type estimate for $F\in W^{-1,p}(\Omega;\mathbb{R}^d)$.
Note that for any $F\in W^{-1,p}(\Omega;\mathbb{R}^d)$ with $1\leq p<\infty$, there exist
$f_0,f_1,\cdots,f_d\in L^p(\Omega;\mathbb{R}^d)$ such that $F = f_0 + \partial f_i/\partial x_i$, and
$\|F\|_{W^{-1,p}(\Omega)} = \max_{0\leq i\leq d}\big\{\|f_i\|_{L^p(\Omega)}\big\}$.
Thus from the linearity of $(\textbf{DS}_\varepsilon)$, if the source term $\text{div}(f)$ of
$\eqref{pde:5.2}$ is substituted for $F\in W^{-1,p}(\Omega;\mathbb{R}^d)$ in Theorem $\ref{thm:5.4}$,
and then we can derive
\begin{equation}
 \|\nabla u_\varepsilon\|_{L^p(\Omega)} + \|p_\varepsilon\|_{L^p(\Omega)/\mathbb{R}}
 \leq C\|F\|_{W^{-1,p}(\Omega)},
\end{equation}
where we actually set $F=f_0+\text{div}(f)$, and employ the estimates $\eqref{pri:5.8}$ and $\eqref{pri:5.9}$.}
\end{remark}

\begin{flushleft}
\textbf{Proof of Theorem \ref{thm:1.2}.}
First of all, the proof of the estimate $\eqref{pri:1.5}$ has been given in Theorem $\ref{thm:5.4}$ and
Remark $\ref{remark:5.1}$ for the case $h=0$ and $g=0$. The remaining thing is to handle the inhomogeneous
equations. Let $G$ be the extension of $g$ such that $G=g$ on $\partial\Omega$ in trace sense and
$\|G\|_{W^{1,p}(\Omega)}\leq C\|g\|_{B^{1-1/p,p}(\partial\Omega)}$.
Thus $u_\varepsilon-G = 0$ on $\partial\Omega$.
\end{flushleft}
\vspace{-0.3cm}

Consider the equations: $\text{div}(v) = h-\text{div}(G)$ in $\Omega$ and $v=0$ on $\partial\Omega$.
By noting that $\int_\Omega h dx - \int_{\Omega}\text{div}(G) dx = \int_\Omega h dx
- \int_{\partial\Omega}n\cdot g dS = 0$ (see the compatibility condition $\eqref{a:4}$), we have the
unique existence of $v\in W^{1,p}_0(\Omega;\mathbb{R}^d)$ according to Lemma $\ref{lemma:2.2}$. Moreover,
it follows from the estimate $\eqref{pri:2.2}$ that
\begin{equation}\label{f:5.13}
\|v\|_{W^{1,p}_0(\Omega)}\leq C\Big\{\|h\|_{L^p(\Omega)}+\|\nabla G\|_{L^p(\Omega)}\Big\}
\leq C\Big\{\|h\|_{L^p(\Omega)}+\|g\|_{B^{1-1/p,p}(\partial\Omega)}\Big\}.
\end{equation}
We now observe that $\text{div}(u_\varepsilon-G-v) = 0$ in $\Omega$ and $u_\varepsilon-G-v = 0$ on $\partial\Omega$.
If $w_\varepsilon = u_\varepsilon-G-v$, then
\begin{equation*}
\mathcal{L}_\varepsilon(w_\varepsilon)+\nabla p_\varepsilon
= \tilde{F} \quad \text{in}~\Omega,
\qquad \text{div}(w_\varepsilon) = 0 \quad\text{in}~\Omega,
\qquad w_\varepsilon =0 \quad \text{on}~\partial\Omega,
\end{equation*}
where $\tilde{F} = F+\text{div}\big[A(\cdot/\varepsilon)\nabla(G+v)\big]\in W^{-1,p}(\Omega;\mathbb{R}^d)$.
It follows from the estimate $\eqref{pri:5.9}$ that
\begin{equation}\label{f:5.14}
\begin{aligned}
 \|\nabla w_\varepsilon\|_{L^p(\Omega)} + \|p_\varepsilon\|_{L^p(\Omega)/\mathbb{R}}
 & \leq C\Big\{\|F\|_{W^{-1,p}(\Omega)} + \|\nabla G\|_{L^p(\Omega)} + \|\nabla v\|_{L^p(\Omega)}\Big\} \\
 & \leq C\Big\{\|F\|_{W^{-1,p}(\Omega)} + \|h\|_{L^p(\Omega)} + \|g\|_{B^{1-1/p,p}(\partial\Omega)}\Big\}.
\end{aligned}
\end{equation}
Hence the desired estimate $\eqref{pri:1.5}$ consequently follows from $\eqref{f:5.13}$ and $\eqref{f:5.13}$,
and we have completed the proof.
\qed

\begin{center}
\textbf{Acknowledgements}
\end{center}

The author wants to express his sincere appreciation to Professor Zhongwei Shen
for his constant and illuminating instruction.
This work was supported by the National Natural Science Foundation of China (Grant NO.11471147).

\end{document}